\documentclass[11pt]{amsart} 

\usepackage{amsmath,amssymb,amsthm} 
\usepackage{graphicx} 
\usepackage[arrow, matrix, curve]{xy} 
\usepackage{xcolor}
\usepackage{makecell} 
\usepackage{hyperref} 
\usepackage[abs]{overpic}		
\usepackage{enumitem}
\usepackage{rotating}

\let\oldmarginpar\marginpar
\renewcommand\marginpar[1]{\-\oldmarginpar[\raggedleft\footnotesize #1]%
	{\raggedright\footnotesize #1}}
\definecolor{darkgreen}{rgb}{0,.647, .317}
\definecolor{darkorange}{rgb}{.964,.404, .136}
\definecolor{npink}{cmy}{.10,.80,0}

\DeclareMathOperator{\tb}{tb}
\DeclareMathOperator{\rot}{rot}
\DeclareMathOperator{\lk}{lk}

\def\dfn#1{{\textbf{#1}}}
\newcommand{\R}{\mathbb{R}}
\newcommand{\Z}{\mathbb{Z}}
\newcommand{\Q}{\mathbb{Q}}
\newcommand{\N}{\mathbb{N}}

\newcommand{\Op}{{\mathcal{O}} p}
\newcommand{\xist}{\xi_{\mathrm{st}}}

\theoremstyle{plain}
\newtheorem{Theorem}{Theorem}[section]
\newtheorem{thm}[Theorem]{Theorem}
\newtheorem{lem}[Theorem]{Lemma}
\newtheorem{prop}[Theorem]{Proposition}
\newtheorem{cor}[Theorem]{Corollary}
\newtheorem*{Theorem-ohne}{Theorem}
\newtheorem{con}[Theorem]{Conjecture}

\newtheorem{ques}[Theorem]{Question}

\theoremstyle{definition}
\newtheorem{defi}[Theorem]{Definition}
\newtheorem{rem}[Theorem]{Remark}
\newtheorem{ex}[Theorem]{Example}

\begingroup\expandafter\expandafter\expandafter\endgroup
\expandafter\ifx\csname pdfsuppresswarningpagegroup\endcsname\relax
\else
\fi

\begin{document}


\title[Stein traces and characterizing slopes]{Stein traces and characterizing slopes}

\author{Roger Casals}
\address{University of California Davis, Dept. of Mathematics, Shields Avenue, Davis, CA 95616, USA}
\email{casals@math.ucdavis.edu}

\author{John Etnyre}
\address{School of Mathematics, Georgia Institute
	of Technology,  Atlanta, GA}
\email{etnyre@math.gatech.edu}

\author{Marc Kegel}
\address{Humboldt-Universit\"at zu Berlin, Rudower Chaussee 25, 12489 Berlin, Germany.}
\email{kegemarc@math.hu-berlin.de, kegelmarc87@gmail.com}

\date{\today}


\begin{abstract}
We show that there exists an infinite family of pairwise non-isotopic Legendrian knots in the standard contact $3$-sphere whose Stein traces are equivalent. This is the first example of such phenomenon. Different constructions are developed in the article, including a contact annulus twist, explicit Weinstein handlebody equivalences, and a discussion on dualizable patterns in the contact setting. These constructions can be used to systematically construct distinct Legendrian knots in the standard contact $3$-sphere with contactomorphic $(-1)$-surgeries and, in many cases, equivalent Stein traces. In addition, we also discuss characterizing slopes and provide results in the opposite direction, i.e. describe cases in which the Stein trace, or the contactomorphism type of an $r$-surgery, uniquely determines the Legendrian isotopy type.
\end{abstract}

\makeatletter
\@namedef{subjclassname@2020}{%
  \textup{2020} Mathematics Subject Classification}
\makeatother

\makeatletter
\def\blfootnote{\xdef\@thefnmark{}\@footnotetext}
\makeatother

\subjclass[2020]{53D35; 53D10, 57K10, 57R65, 57K10, 57K33} 

\keywords{Characterizing slopes, contact Dehn surgery, Legendrian knots, Stein manifolds, Weinstein handles}

\maketitle


\section{Introduction}

The existence of distinct Legendrian knots in $(S^3,\xi_{st})$ with equivalent Stein traces, or even just contactomorphic $(-1)$-surgeries, has remained an interesting open question in low-dimensional contact and symplectic topology. This manuscript shows that many instances of such pairs of Legendrian knots, and even infinite families, do indeed exist. In particular, our results imply that, in general, the contactomorphism type of a $(-1)$-surgery along a Legendrian knot $L$ in $(S^3,\xi_{st})$, or even the equivalence class of the Stein trace, knows relatively little about the Legendrian knot $L$. In fact, we present different techniques, including explicit Weinstein-Kirby calculus and abstract contact topological arguments, that lead to the construction of such Legendrian knots. In contrast, the article also explores different cases in which the contactomorphism type(s) of a contact $r$-surgery along $L$ in $(S^3,\xi_{st})$ does actually recover the Legendrian isotopy class of $L$. The manuscript concludes with a discussion on conjectural matters which are naturally inferred from our results.

\subsection{Scientific Context} A natural and much studied problem in low-dimensional topology asks when the diffeomorphism type of a surgery on a knot in $S^3$ determines the smooth isotopy class of the knot \cite{Brakes80,KirbyList,Osoinach06}. In fact, there has been abundant work on the construction of families of knots which share not only a common $3$-dimensional $n$-surgery, but also have a common $4$-dimensional $n$-trace, see e.g.~ \cite{AbeJongLueckeOsoinach15, AbeJongOmaeTakeuchi13, MillerPiccirillo18, Piccirillo19}.\footnote{The $n$-trace of a knot $K$ in $ S^3$ is the smooth $4$-manifold obtained as the result of attaching a $2$-handle to the $3$-sphere boundary of the $4$-ball along the knot $K$ with framing $n$.} As a complement to these results, it is also known that any surgery on the unknot~\cite{KronheimerMrowkaOzsvathSzabo07}, the trefoil and the figure eight knot~\cite{OzsvathSzabo19}, actually characterizes the smooth type of these knots. For instance, if the $r$-surgery on a knot $K$ in $S^3$ is diffeomorphic to the $r$-surgery on the unknot, then $K$ is smoothly isotopic to the unknot; similarly for the trefoil and the figure eight knot. Given all these results, the problem is, in a sense, quite well-studied in smooth low-dimensional topology. The aim of this article is to study the analogous problem in low-dimensional contact and symplectic topology, and provide some (encouraging) initial answers.

In precise terms, let $L$ in $ (S^3,\xist)$ be a Legendrian knot in the standard contact $3$-sphere, which is seen as the contact boundary of the standard symplectic $4$-ball $(D^4,\lambda_{st})$. By performing contact $(\pm 1)$-surgery on $L$  we produce a new contact manifold $L(\pm 1)$. Contact $(-1)$-surgery is sometimes also called Legendrian surgery. In fact, we may also construct a new symplectic (even Stein) $4$-manifold $W_L$, whose contact boundary is $L(-1)$. The Stein 4-manifold $W_L$ is said to be the \textit{Stein trace} of $L$, and it is obtained by attaching a Weinstein $2$-handle along $L$ to $(D^4,\lambda_{st})$, following \cite[Section 3]{Weinstein91}. For the necessary background, see e.g.~\cite{ArnoldGivental01,Geiges08,Weinstein91}. 

In this manuscript, two Stein manifolds are said to be \textit{equivalent} if they are symplectomorphic after possibly deforming the symplectic structure on one of the manifolds. 

\begin{defi}
	Let $L$ in $(S^3,\xist)$ be a Legendrian knot. By definition, $L$ is said to be \textit{Stein characterized} if whenever $W_{L}$ is equivalent to $W_{L'}$ for some Legendrian knot $L'$ in $(S^3,\xist)$ then $L'$ is Legendrian isotopic to $L$. Similarly, the contact surgery slope $-1$ (or $+1$) is called \textit{characterizing} for $L$ if whenever $L(-1)$ (or $L(+1)$) is contactomorphic to $L'(-1)$ (or $L'(+1)$) for some Legendrian knot $L'$ in $(S^3,\xist)$, then $L'$ is Legendrian isotopic to $L$.\hfill$\Box$
\end{defi}

Since the boundary of $W_L$ carries a natural contact structure contactomorphic to $L(-1)$, it follows that if the slope $-1$ is characterizing for $L$, then $L$ is also Stein characterized. (Note that we will later also discuss contact $r$-surgeries for other slopes $r\not=\pm1$.) The definition of a non-characterizing slope is extended to general contact manifolds in the natural way.

\subsection{Non-characterizing slopes}
In \cite{Etnyre08}, the first examples of Legendrian knots in $(S^3,\xist)$ where $(+1)$ is non-characterizing were given, using convex surface theory. By using the cancellation lemma one sees directly that there exist Legendrian knots $L$ and $L'$, the dual surgery knots, in a contact manifold different from $(S^3,\xist)$ such that $(-1)$ is a non-characterizing contact surgery slope. It remained open if there exist non-isotopic Legendrian knots in $(S^3,\xist)$ for which $(-1)$ is a non-characterizing slope~\cite[Question~14]{Etnyre08}. 

In this manuscript, we give several constructions that lead to pairs, and even infinite families, of Legendrian knots in $(S^3,\xist)$ that have the contactomorphic contact $(-1)$-surgeries and, in some cases, even equivalent $4$-dimensional Stein traces. This answers the above question in the affirmative. Our first construction is the contact annulus twist, a contact version of the smooth annulus twist developed in the series of articles \cite{AbeJongLueckeOsoinach15, AbeJongOmaeTakeuchi13, MillerPiccirillo18, Osoinach06}. This is discussed in Section~\ref{sectionannulustwist} and will lead to constructions of infinite families of Legendrian knots in $(S^3,\xist)$ that have the same Legendrian surgeries and, in some cases, the same Stein traces. In particular, we prove the following result:

\begin{thm}\label{thm:main} Let $(S^3,\xi_{st})$ be the standard contact $3$-sphere. Then the following holds:
\begin{itemize}[nosep, noitemsep, topsep=0pt]
	\item[(i)] There exist infinitely many pairwise non-isotopic Legendrian knots $L_n$ in $(S^3,\xist)$, $n\in\N_0$\footnote{We denote the natural numbers including zero by $\N_0$ and the natural numbers without zero by $\N$.}, such that their Stein traces $W_{L_n}$ are all equivalent to $W_{L_0}$, for any $n \in \N_0$. These knots $L_n$ all have $(\tb,\rot)=(1,0)$.
	\end{itemize}
$\quad\,\,\,\text{For more general $tb$, we have:}$ 
\begin{itemize}[nosep, noitemsep, topsep=0pt]
	\item[(ii)]  For each $m\in\Z$, there exist pairs of non-isotopic Legendrian knots $L_m,L'_m$ in $(S^3,\xist)$, with $\tb(L_m)=\tb(L_m')=m$, such that their Stein traces $W_{L_m}$ and $W_{L'_m}$ are equivalent.	
\end{itemize}
\end{thm}

The Legendrian knots $L_n$ in Theorem \ref{thm:main}.(i), as well as the pairs in Theorem \ref{thm:main}.(ii), are all pairwise distinguished by their smooth isotopy types, and thus they are readily not Legendrian isotopic. Theorem \ref{thm:main}.(i) implies that there exist infinitely many non-isotopic Legendrian knots $L_n$ in $(S^3,\xist)$, $n\in\N_0$, such that $L_n(-1)$ is contactomorphic to $L_0(-1)$, and thus $(-1)$ is not a characterizing slope for either of these Legendrian knots. 

In Section~\ref{sec:RGB} we also construct examples of Legendrian knots with contactomorphic Legendrian surgeries but non-equivalent Stein traces.

\begin{thm}\label{thm:notExtends}
	There exist pairs of Legendrian knots $L$ and $L'$ in $(S^3,\xist)$ that have contactomorphic Legendrian surgeries but whose Stein traces are not homeomorphic.
\end{thm}

In particular, there exists a contact manifold $(M,\xi)$ that admits (at least) two non-equivalent simply connected Stein fillings, and each of these two Stein fillings is obtained by attaching a single Weinstein $2$-handle to the standard symplectic $4$-ball.

\begin{figure}[htbp]{\small
		\begin{overpic}
			{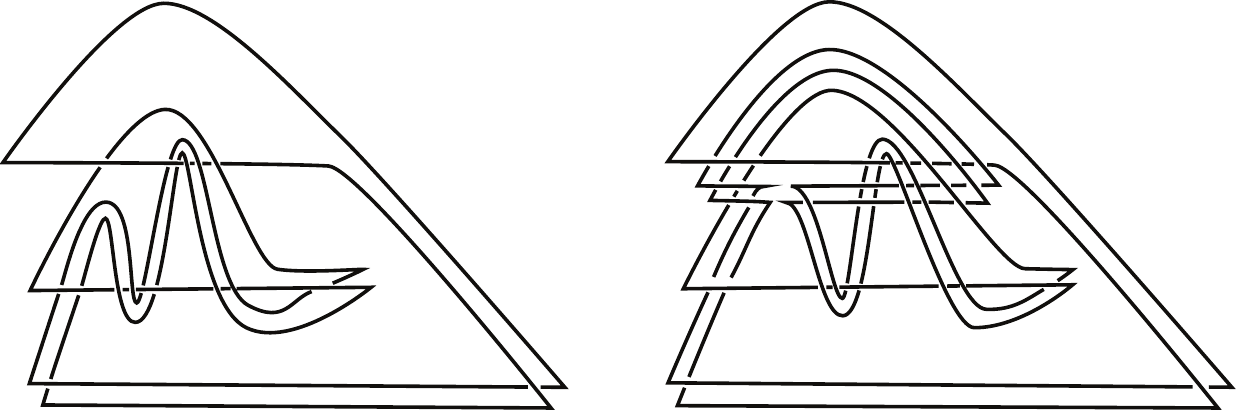}
			\put(120, 60){$L_0$}
			\put(313, 60){$L_1$}
			\put(115, 13){$(-1)$}
			\put(315, 13){$(-1)$}
	\end{overpic}}
	\caption{Instances of Legendrian knots $L_0$ and $L_1$ for Theorem \ref{thm:main}.(i).}
	\label{fig:primeex}
\end{figure}
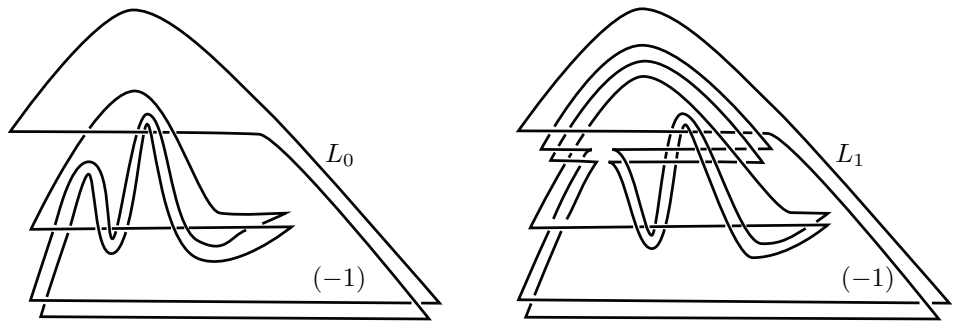

\begin{ex}
Instances of $L_0$ and $L_1$ for Theorem \ref{thm:main}.(i) are shown in Figure~\ref{fig:primeex} and, in general, $L_n$ is obtained by a contact annulus twist on $L_{n-1}$, as we have developed in Section~\ref{sectionannulustwist}. An explicit Stein equivalence between $W_{L_0}$ and $W_{L_1}$ is shown in Figure \ref{fig:ExampleEquivalence}. Of course, once we have given our candidates $L_0$ and $L_1$ in $(S^3,\xi_{st})$ one may try to explicitly show they give the same Stein manifolds. However, without some good conceptual framework to systematically construct such pairs (and even infinite families) of knots, as developed in Sections~\ref{sectionannulustwist} and \ref{sec:other}, it is not easy to find them, much less find the explicit handle slides to demonstrate that they are equivalent. \hfill$\Box$
\end{ex}

\begin{center}
	\begin{figure}[htbp]
		\centering
		{\small
		\begin{overpic}
			{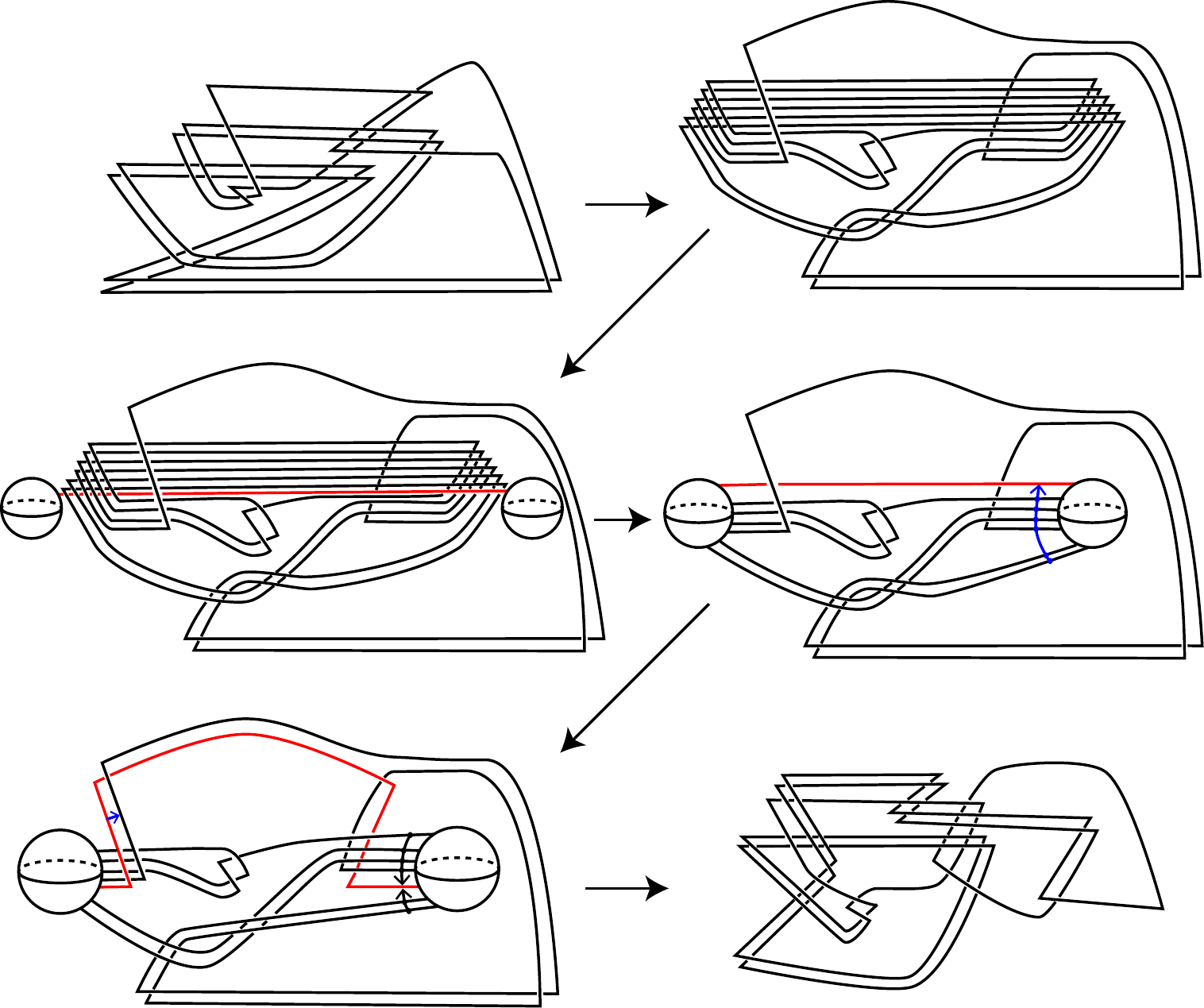}
			\put(0, 350){$(i)$}
			\put(240, 350){$(ii)$}
			\put(0, 230){$(iii)$}
			\put(240, 230){$(iv)$}
			\put(0, 95){$(v)$}
			\put(240, 95){$(vi)$}
			\put(40, 330){$L_0$}
			\put(320, 345){$L_0$}
			\put(150, 272){$(-1)$}
			\put(390, 272){$(-1)$}
			\put(150, 142){$(-1)$}
			\put(390, 140){$(-1)$}
			\put(5, 198){\color{red}$(-1)$}
			\put(290, 195){\color{red}$(-1)$}
			\put(79, 86){\color{red}$(-1)$}
			\put(140, 10){$(-1)$}
			\put(390, 20){$(-1)$}
	\end{overpic}}
		\caption{A sequence of equivalent Weinstein handlebodies, starting at the Stein trace $W_{L_0}$ and ending at $W_{L_1}$. This explicitly exhibits two Legendrian knots, $L_0$ and $L_1$ in $(S^3,\xi_{st})$, and the equivalence between their Stein traces. $(i)$ to $(ii)$ is a Legendrian isotopy. In $(iii)$ we introduce a canceling $1$-/$2$-handle pair. We go from $(iv)$ to $(iii)$ by sliding the black $2$-handle over the red one as indicated with the blue arrow. Finally we can go from $(v)$ to $(iv)$, respectively $(vi)$, by handle slides as indicated with the blue, respectively black arrows.  To our knowledge, this is the first example of such a phenomenon. Theorem~\ref{thm:main} and the techniques we develop in Sections~\ref{sectionannulustwist} and \ref{secondannulus} vastly generalize this equivalence and provide more conceptual arguments and techniques to construct them.}
		\label{fig:ExampleEquivalence}
	\end{figure}
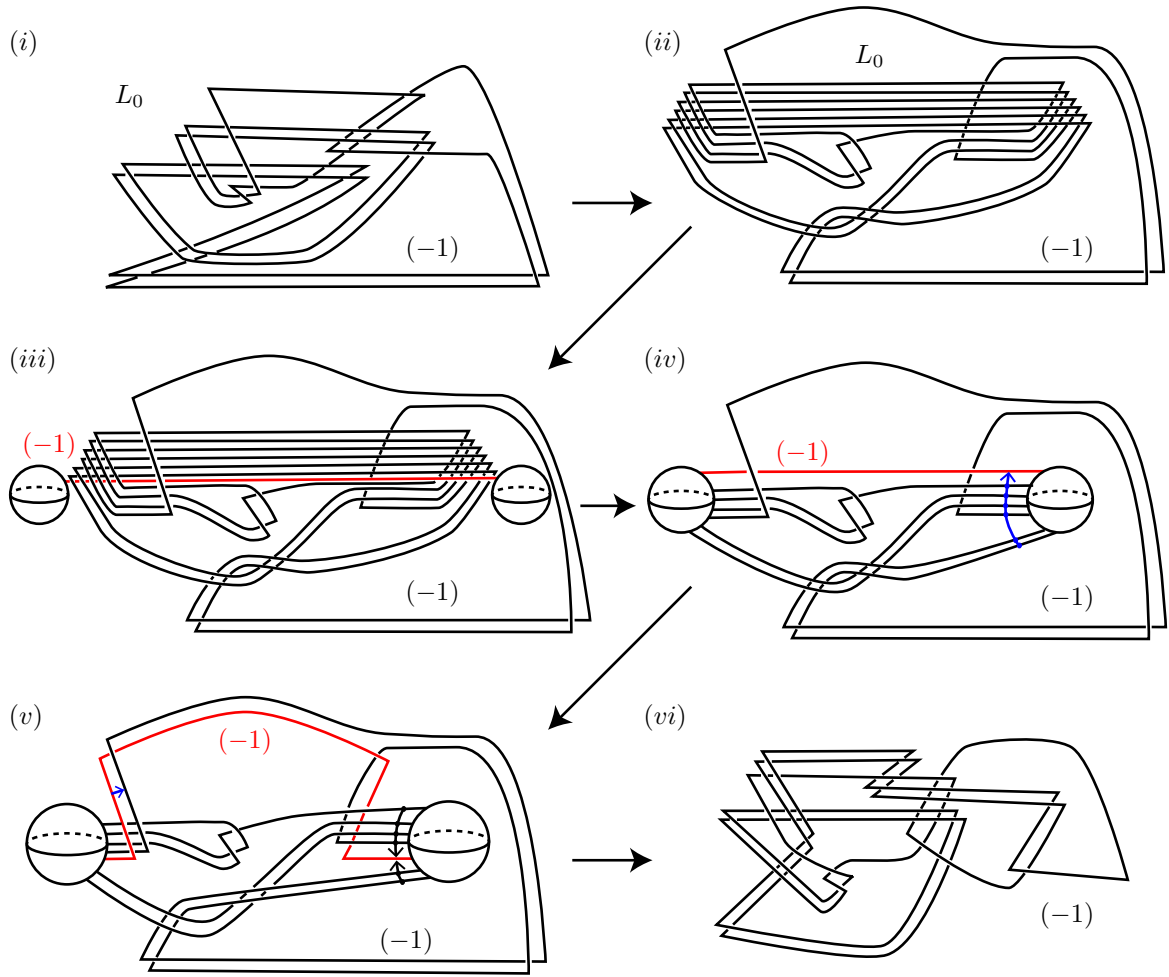
\end{center}

The proof of Theorem~\ref{thm:main}.(i) is given in Section~\ref{sectionannulustwist}. First, we construct the candidate knots $L_n$ in $(S^3,\xi_{st})$, $n\in\N_0$, and prove that their contact $(-1)$-surgeries are all contactomorphic. This is achieved by explicitly describing the knots $L_n$ in their front projections and exhibiting the contactomorphisms between their Legendrian surgeries via sequences of contact handle slides. This requires us to generalize contact handle slides~\cite{DingGeiges09} to surgery coefficients of the form $\pm1/n$ and then develop a contact version of the annulus twist~\cite{Osoinach06}, which both might be of independent interest. Then, the equivalence of Stein traces in Theorem~\ref{thm:main}.(i) is obtained by showing that these $3$-manifold contactomorphisms extend to an equivalence of their $4$-dimensional Stein traces. This is achieved by generalizing a construction of Akbulut~\cite{Akbulut77}, which allows to extend certain diffeomorphisms of a $3$-manifold over a $4$-manifold that it bounds, to contact and symplectic topology. Theorem \ref{thm:main}.(ii) is proven similarly.

In Section~\ref{sec:other} we will also develop other methods to create pairs of Legendrian knots that share the same Legendrian surgery or Stein traces. In particular, we will generalize dualizable patterns \cite{Brakes80, MillerPiccirillo18} and RGB links \cite{Piccirillo19,Tagami1} to the contact and symplectic framework.

\subsection{Characterizing slopes}\label{sec:main_char}
In contrast to Theorem~\ref{thm:main} above, we also show that certain Legendrian knots $L$ in $(S^3,\xi_{st})$ have characterizing slopes. Let us start by proving that certain Legendrian knots are determined by their Stein traces: 
\begin{thm}\label{Steinchar}
Let $L$ in $(S^3,\xi_{st})$ be a Legendrian realization of the unknot, the right- or left-handed trefoil, or the figure eight knot, and let $L'$ in $(S^3,\xi_{st})$ be another Legendrian knot. Suppose that the Stein trace $W_L$ of $L$ is equivalent to the Stein trace $W_{L'}$ of $L'$. Then $L'$ is Legendrian isotopic to $L$ in $(S^3,\xi_{st})$. 
\end{thm}

In fact, for many of these Legendrian knots $(-1)$ is also a characterizing slope, as we show in the following result.
\begin{thm}\label{minus1char}
Let $L$ and $L'$  be two Legendrian knots in $(S^3,\xi_{st})$ such that $L$ is
\begin{enumerate}
\item any Legendrian unknot,
\item a right-handed Legendrian trefoil with $\rot(L)=0$, or $\tb(L)\in \{1, 0,-1\}$,
\item a left-handed Legendrian trefoil or a Legendrian figure eight knot with 
\begin{enumerate}
	\item $\rot(L)=0$, or 
	\item  $\tb(L)\geq -10$ (and in the case that $\tb(L)=-6$ we have $\rot(L)\neq0$), or 
\item $\rot(L)\geq \sqrt{6(1-\tb(L))}$.
\end{enumerate}
\end{enumerate}
If contact $(-1)$-surgery on $L$ is contactomorphic to contact $(-1)$-surgery on $L'$, then the Legendrian knot $L'$ is Legendrian isotopic to $L$ in $(S^3,\xi_{st})$. 
\end{thm}
We can also show that a few of these knots also have $(+1)$ as a characterizing slope.
\begin{thm}\label{1characterize}
Let $L$ in $(S^3,\xi_{st})$ be a Legendrian realization with $\rot(L)=0$ of the unknot, the right- or left-handed trefoil, or the figure eight knot, and let $L'$ in $(S^3,\xi_{st})$ be another Legendrian knot. Suppose that the contact $(+1)$-surgery $L(+1)$ is contactomorphic to $L'(+1)$. Then the Legendrian knot $L'$ is Legendrian isotopic to $L$ in $(S^3,\xi_{st})$. 
\end{thm}
We note that, among other realizations of $L$ in $(S^3,\xi_{st})$, the theorem covers the maximal Thurston--Bennequin invariant realizations of these knots except for the left-handed trefoil. The results for the unknot with $\tb=-1$ also follow from results on contact surgery numbers in~\cite{EtnyreKegelOnaranPre}. In fact, they also can handle the case when $\tb=-2$, and the arguments in the proofs of Theorem~\ref{minus1char} and~\ref{1characterize} can handle those cases as well, and some other cases, even though they are not stated in the theorems. 

\subsection{Results for general surgery slopes}\label{ssec:generalslope}
Let $L$ be a Legendrian knot in $(S^3,\xist)$. In Section~\ref{sec:char} we consider more general characterizing slopes, and study whether contact $(r)$-surgery $L(r)$ on the Legendrian knot $L$ can be characterizing, with $r\in\Q\setminus\{0\}$ a non-vanishing rational number. In this case $L(r)$ might not be a unique contact 3-manifold, and thus we will momentarily provide details on what is meant by a contact $(r)$-surgery being {\it characterizing}. For known results on characterizing slopes in the smooth setting, see e.g.~\cite{BakerMotegi18, KronheimerMrowkaOzsvathSzabo07, Lackenby19, McCoy19, McCoy20, NiZhang14, OzsvathSzabo19} and the references therein. 

Consider the smooth 3-manifold $M$ obtained from $S^3$ by topological $\tb(L)+r$ surgery on $L$. The manifold $M$ has two pieces, one is the complement $S^3\setminus\mathring{\nu L}$ of a standard neighborhood $\nu L$ of $L$ in $S^3$ and the other is a solid torus $V$ which is glued to $S^3\setminus\mathring{\nu L}$, so that its meridian is the curve of slope $\tb(L)+r$ on $\partial (S^3\setminus \mathring{\nu L})$. We can restrict $\xist$ to $S^3\setminus \mathring{\nu L}$ and would like to extend this over $V$. By work of Giroux~\cite{Giroux00} and Honda~\cite{Honda00a}, there are finitely many tight contact structures on $V$ with the given boundary conditions if $r\not=0$, and at least one such tight contact structure exists. We say that any one of these extensions is a contact structure on $M$ obtained by contact $(r)$-surgery on $L$. The set of all contact structures constructed this way will be denoted by $L(r)$.\footnote{For reference, we note that $L(r)$ consists of a unique contact manifold if and only if $r=1/n$ for some integer $n$. So our notation for $L(\pm 1)$ agrees with the notation above. If $n>1$ is an integer, then $L(n)$ consists of two contact manifolds, and if $n$ is a negative integer, then $L(n)$ consists of $|n|$ contact manifolds, which are not always different.}

The collection $L(r)$ is said to be contactomorphic to the collection $L'(r')$ if there exists a bijection between the elements in $L(r)$ and $L'(r')$ such that paired elements are contactomorphic manifolds. Then, a contact surgery coefficient $p/q\in\Q$ is called {\it characterizing} if whenever the collection $L(p/q)$ is contactomorphic to the collection $L'(p/q)$ for some Legendrian knot $L'$ in $(S^3,\xist)$, then $L'$ is Legendrian isotopic to $L$ in $(S^3,\xi_{st})$. This generalizes the notion for contact $(\pm1)$-surgery.\footnote{We observe that the question which contact surgery coefficients are characterizing differs from the topological question since we have always infinitely many Legendrian realizations of a single topological knot type and the framings are measured with respect to the contact framing.}

We show that certain Legendrian knots $L$ in $(S^3,\xi_{st})$ have more general characterizing contact surgery coefficients, as follows:

\begin{thm}\label{uknotchar}
	Let $L$ in $(S^3,\xi_{st})$ be a Legendrian realization of the unknot.
\begin{itemize}
	\item[(i)] If $r<4$, $r\in\Q$, except $r\in\{0,2,3\}$, then $r$ is a characterizing contact slope for $L$.	
	\item[(ii)] In addition, $-\tb(L)$ and $(1-q\cdot \tb(L))q^{-1}$, for $q\in\Z\setminus\{0\}$, are characterizing contact slopes for $L$.
\end{itemize}
\end{thm}

\begin{ex}\label{ex:noncharunknot}
Note that it is not true that any slope is characterizing for Legendrian unknots. For instance, contact $(+6)$-surgery on a Legendrian unknot with $\tb=-11$ and $\rot=0$ is contactomorphic to contact $(+6)$-surgery on a Legendrian right-handed trefoil with $\tb=-1$ and $\rot=0$.  This may be seen by noting that the corresponding smooth surgeries produce $L(5,1)$. The contact structures from both surgeries are readily seen to be overtwisted and homotopic as plane fields, thus contactomorphic. By stabilizing these knots further we get an infinite sequence of non-characterizing contact slopes.\hfill$\Box$
\end{ex}

Let us write $n\gg1$ to mean that $n$ is sufficiently large. For more general knot types, we will prove the following result:

\begin{thm}\label{thm:reciprokal}
Let $K$ in $S^3$ be a Legendrian simple knot, and $L$ in $(S^3,\xist)$ a Legendrian realization of $K$. Then the following holds:

\begin{itemize}
	
	\item[(i)] If $K$ is hyperbolic, then $(\pm1/n)$ is a contact characterizing slope of $L$ for $n\gg1$.
	
	\item[(ii)] If $\tb(L)\in \{-1, 0, 1\}$, then $(\pm1/n)$ is a contact characterizing slope of $L$ for $n\gg1$.
	
	\item[(iii)] If $K$ is the left- or right-handed trefoil, or the figure eight knot. Then $(\pm1/n)$ is a contact characterizing slope of $L$ for $n\geq3$. 
	
	\item[(iv)] In addition, we have:
	
	\begin{enumerate}
			 
		\item If $K$ is a left- or right-handed trefoil, then $-1-\tb(L)$ is a contact characterizing slope of $L$.
		
		\item If $K$ is a figure eight knot, then $1-\tb(L)$ is a contact characterizing slope of $L$.
	\end{enumerate}	

\end{itemize}
\end{thm}

In Section \ref{sec:QuestionsConj}, the manuscript concludes with a few natural questions and conjectures stemming in Theorems \ref{thm:main}, \ref{Steinchar}, \ref{minus1char}, \ref{1characterize}, \ref{uknotchar}, and~\ref{thm:reciprokal} above. Section \ref{sec:QuestionsConj} can be read directly after this introduction, if the reader so desires.

\subsection*{Conventions}
Throughout this paper, we assume the reader to be familiar with Dehn surgery and contact topology on the level of~\cite{GompfStipsicz99, Geiges08}. 
For the background on contact surgery and Kirby calculus of symplectic manifolds, in particular for details on the cancellation lemma, contact handle slides and handle cancellations, we refer to~\cite{DingGeiges09, DingGeigesStipsicz04, EtnyreKegelOnaranPre}.  

We work in the smooth category. All manifolds, maps, and ancillary objects are assumed to be smooth. We assume all $3$-manifolds to be connected closed oriented, and all contact structures to be positive and co-orientable. Legendrian links in $(S^3,\xist)$ are always presented in their front projection.

In this article we consider unoriented Legendrian knots. Whenever we speak about the rotation number of a Legendrian knot, we mean the absolute value of the rotation number of the Legendrian knot with one of its orientations (the absolute value will be independent of the chosen orientation). This is necessary since the change of orientation of a Legendrian knot will in general change its isotopy class (as an oriented Legendrian knot) but its Stein traces and contact surgeries will not be affected, up to symplectic equivalence and contactomorphism, by an orientation change.  \hfill$\Box$

\subsection*{Acknowledgments}
We would like to thank Bob Gompf, Oleg Lazarev, and Lisa Piccirillo for helpful comments on an earlier draft of this paper.  We also thank the anonymous referee for many helpful comments and suggestions that have improved the paper. 
The first author is supported by the NSF CAREER Award DMS-1942363 and the Alfred P. Sloan Foundation. He is also thankful to the second author and the Georgia Institute of Technology for their hospitality during his May 2019 visit, where we first started to discuss this project. The second author thanks Lisa Piccirillo for very helpful conversations about the annulus twist and a beautiful set of lectures at the 2021 Tech Topology Summer School that informed our understanding of some of the constructions in Section~\ref{sec:other}.
The second author was partially supported by NSF grant DMS-1906414 and DMS-2203312.
The third author would like to thank the \textit{Mathematisches Forschungsinstitut Oberwolfach} where large parts of this project were carried out as a \textit{Oberwolfach Research Fellow} in August 2020.



\section{Contact surgery and contact handle slides}
In this section, we first state known facts about contact surgeries, and computation of invariants of contact structures after contact surgery, that we use in the manuscript. Then, we discuss the notion of a contact handle slide in the context of a contact surgery with coefficient $(\pm 1/n)$. The section ends with a discussion on some contact surgeries on the standard Legendrian unknot that will be used later in the article. 

Throughout this paper we will take the standard contact structure on $\R^3$ (and $S^3$ minus a point) to be give by the kernel of $dz-y\, dx$ and with this convention, the front projection of a Legendrian link will project out the $y$-coordinate. 

\subsection{Contact surgery} Two of the homotopical invariants of a contact $3$-manifold are the Euler class of the underlying $2$-plane distribution and the $d_3$-invariant. We will need to compute these invariants for contact structures obtained via contact surgeries in Section~\ref{sec:char}. For that we will use the formulas by Gompf~\cite{Gompf98} and Ding--Geiges--Stipsicz~\cite{DingGeigesStipsicz04} (and their slight extension for more general surgery coefficients from~\cite{DurstKegel16}).

\begin{thm}\label{computInv}
		Let $(M,\xi)$ be the contact manifold obtained from $(S^3,\xist)$ by contact surgery along an oriented Legendrian link $\mathcal{L}=L_1 \cup\ldots\cup L_k$ with contact surgery coefficients $(\pm 1/n_i)$, $n_i\in \N$, of the $L_i$.
		\begin{itemize}
			\item Then, the Poincar\'e dual of the Euler class $e(\xi)\in H^2(M;\Z)$ is given by
			\[
			\operatorname{PD}\big(e(\xi)\big)=\sum_{i=1}^k n_i\rot(L_i)[\mu_i] \in H_1(M;\Z),
			\]
			where $\mu_i$ is the meridian of the component $L_i$.\\
			\item The Euler class $e(\xi)$ is torsion if and only if there exists a rational solution $b\in \Q^k$ of $Qb = \rot$, where $\rot$ is the vector of rotation numbers of $\mathcal L$ and $Q$ is the generalized linking matrix of $\mathcal L$. In this case, the $d_3$-invariant is well defined and computes as
				\[
			d_3(\xi)= \frac 14 \left(\sum_{i=1}^k  n_i b_i \rot(L_i) +(3 -n_i) \operatorname{sign}_i
			\right) - \frac 34 \sigma(Q),  
			\]
			where $\operatorname{sign}_i$ denotes the sign of the contact surgery coefficient of $L_i$ and the generalized linking matrix is
			\begin{align*}
				Q:=\begin{pmatrix}
					p_1&q_2 l_{12} &\cdots&q_n l_{1n}\\
					q_1 l_{21} & p_2&&\\
					\vdots&&\ddots\\
					q_1 l_{n1}&&& p_n
				\end{pmatrix},
			\end{align*}	
		with $p_i/q_i$ the topological surgery coefficient of $L_i$ (i.e. measured with respect to the Seifert longitude) and $l_{ij}$ the linking number of $L_i$ and $L_j$.
		\end{itemize}
	\end{thm}

Note that our definition of the $d_3$-invariant differs from the definition in the cited papers~\cite{Gompf98,DingGeigesStipsicz04,DurstKegel16} by $1/2$. We choose this convention as it has the $d_3$-invariants of all the contact structures on $S^3$ being integers and it is additive under connected sums.

\begin{rem} Note that the $d_3(\xi)$ invariant associated to a contact structure $\xi$, as described in Theorem \ref{computInv}, is a constant multiple of R. Gompf's $\theta(\xi)$-invariant, as introduced in \cite[Section 4]{Gompf98}. In a minor abuse of notation, this is also referred to as $d_3(\xi)$ in the literature, in line with the $d_3(\xi_1,\xi_2)$ secondary obstruction class associated to a pair $\xi_1,\xi_2$ of contact 2-plane fields. We point out that \cite[Corollary 4.6]{Gompf98} therein also deduces that two Legendrian knots $L_1,L_2$ in $(S^3,\xi_{st})$ with contactomorphic $(-1)$-surgeries must have $tb(L_1)=tb(L_2)$ and $|\rot(L_1)|=|\rot(L_2)|$ .\hfill$\Box$
\end{rem}

Throughout the article, especially in Section~\ref{sectionannulustwist}, we repeatedly use the surgery Cancellation Lemma, which reads as follows.
\begin{lem}[Ding and Geiges~\cite{DingGeiges04}]\label{cancellationlemma}
Let $(M,\xi)$ be a contact manifold, and $L$ in $(M,\xi)$ a Legendrian knot and $L'$ a (small) Reeb push-off of $L$. Then, the result of contact $(\pm 1)$-surgery on $L$ and contact $(\mp 1)$-surgery on $L'$ is contactomorphic to $(M,\xi)$.  
\end{lem}

The contact $(\pm 1/n)$-surgery on $L$ is equivalent to contact $(\pm 1)$-surgery on $L$ and $n-1$ Legendrian push-offs of $L$. Hence, Lemma \ref{cancellationlemma} also shows that contact $(\pm 1/n)$-surgery on $L$ and contact $(\mp 1/n)$-surgery on its Reeb push-off cancel as well.

\subsection{Contact handle slides}
In Section \ref{sectionannulustwist}, we need a version of contact handle slides for contact $(\pm 1/n)$-surgeries on Legendrian knots, which were obtained earlier for contact $(\pm1)$-surgeries in \cite{DingGeiges09}, using convex surface theory, and later generalized in~\cite{Avdek13}. We generalize these handle slides to contact $(\pm 1/n)$-surgeries. 

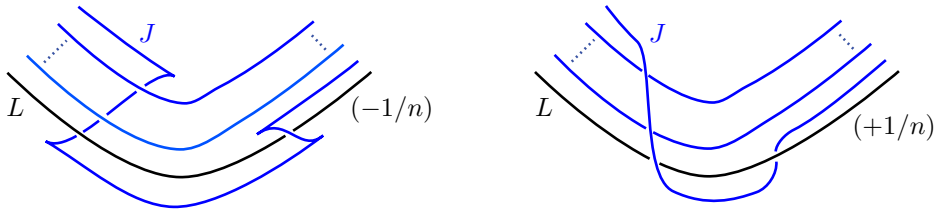
\begin{figure}[htbp]{\small
  \begin{overpic}
  {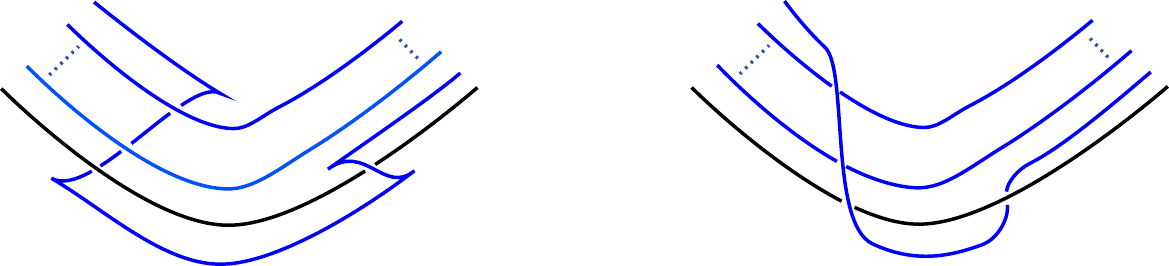}
     \put(0, 35){$L$}
      \put(200, 35){$L$}
      \put(50, 63){\color{blue} $J$}
      \put(243, 63){\color{blue}$J$}
      \put(130, 35){$(-1/n)$}
      \put(320, 28){$(+1/n)$}
  \end{overpic}}
  \caption{The Legendrian knot $J$ consists of $n$ push-offs of $L$ away from the shown segment. We call $J$ the $(\pm1,n)$-Legendrian cable of $L$.}
  \label{fig:tb-1}
\end{figure}

We begin by introducing the notation of a $(\pm1,n)$-Legendrian cable $J$ along a Legendrian knot $L$. This is the Legendrian knot $J$ depicted in Figure \ref{fig:tb-1}: it is obtained by considering the Reeb $n$-copy of $L$ and inserting one of the patterns in Figure \ref{fig:tb-1}. For the $(-1,n)$-Legendrian cable, we use the pattern depicted on the left, and for the $(1,n)$-Legendrian cable, we use the pattern depicted on the right. The first row of Figure~\ref{fig:cable_welldef} shows that the pattern in Figure \ref{fig:tb-1} (left), for the $(-1,n)$-cable, exactly reverses as we go through a cusp. That said, the second and third rows of Figure~\ref{fig:cable_welldef} illustrate that the reversed pattern is Legendrian isotopic to the original one. Therefore this front description of the cabling operation is well-defined.

\begin{rem} For the sake of clarity, we explain the isotopy in the second and third rows of Figure~\ref{fig:cable_welldef} in more detail here. In the $1$-jet space, after one has pulled the leftmost right cusp up (2nd row, first to second diagrams), the rest of the isotopy can be described as iterating the following two steps. First, we take the leftmost crossings in the strand piece to the left of the leftmost right cusp (these crossings are circled in pink in the figure), and then we move these crossings to the left around $S^1$ until they appear on the right (2nd row, second to third diagrams, these crossings still in pink). Second, use these crossings that appear on the right to pull the rightmost left cusp down as much as possible (2nd row, third to fourth diagrams). Then we move the remaining crossings, after pulling the rightmost left cusp down, to the left. (These crossing are now depicted in green in the fourth diagram of the 2nd row. These crossing are, in a sense, all the previous pink crossings minus the top one.) Once we have performed these two steps, we iterate this process: move the new crossings on the left more to the left around the $S^1$ until they appear on the right and use their appearance on the right to pull the rightmost left cusp down a bit more. Then iterate this process. In every iteration, the number of crossings we move around decreases exactly by one, thus the process ends. This explains the second and third rows of Figure~\ref{fig:cable_welldef}.
	
Independently of this argument in the $1$-jet space of $S^1$, this isotopy can also be explicitly seen after the cabling occurs, in the cabled front, by performing Legendrian Reidemeister moves near cusps and crossings. In either case, this shows that the isotopy class of a Legendrian cable is independent of the position where we introduce the pattern of Figure~\ref{fig:tb-1}. This works analogously for the $(1,n)$-Legendrian cable.\hfill$\Box$
\end{rem}

The statement for contact handle slides in these general surgeries reads as follows.

\begin{figure}[htbp]{\small
		\begin{overpic}[width=\textwidth]
			{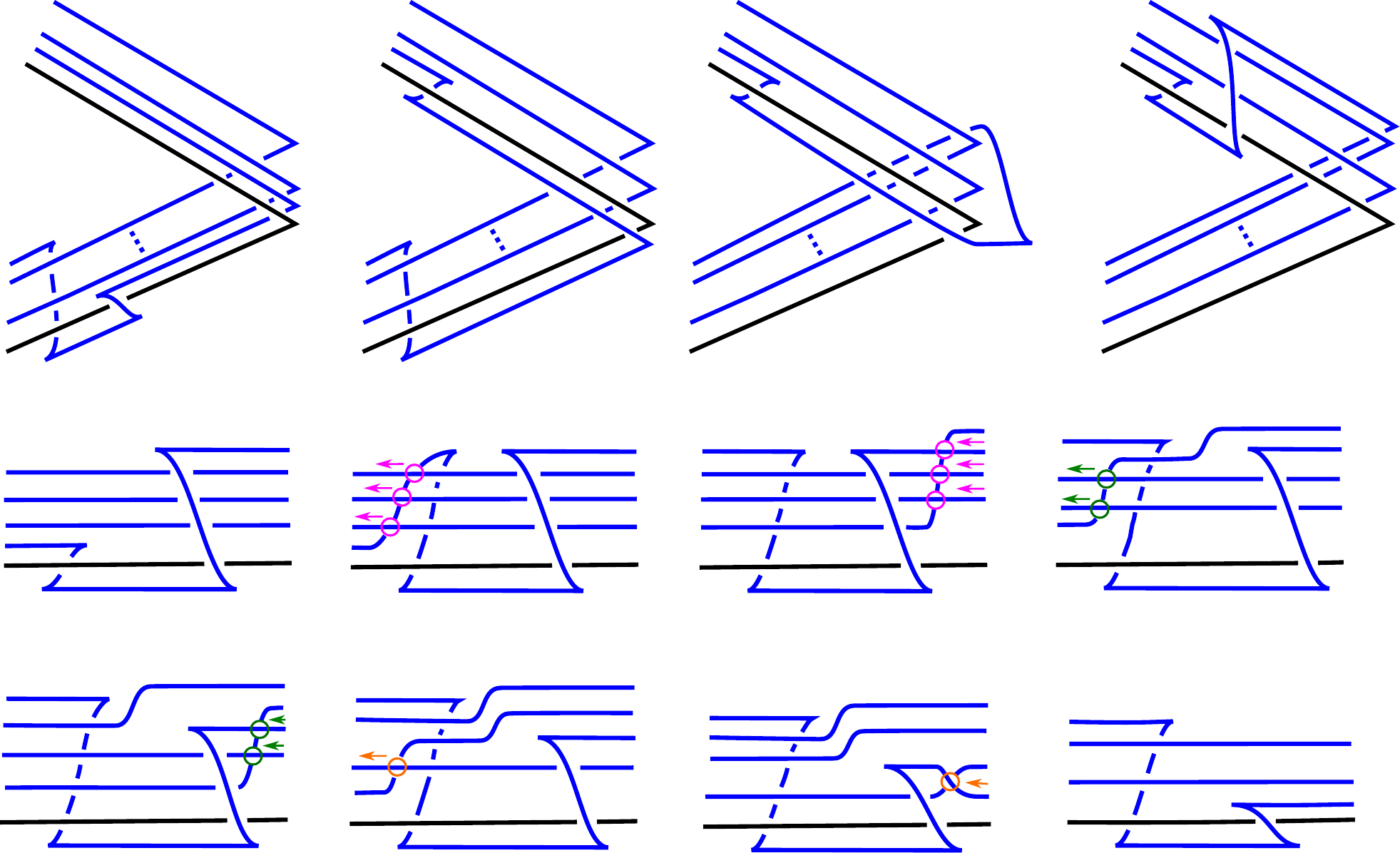}
			\put(110, 210){$\cong$}
			\put(210, 210){$\cong$}
			\put(330, 210){$\cong$}
			\put(96, 105){$\cong$}
			\put(208, 105){$\cong$}
			\put(319, 105){$\cong$}
			\put(430, 105){$\cong$}
			\put(96, 25){$\cong$}
			\put(208, 25){$\cong$}
			\put(319, 25){$\cong$}
	\end{overpic}}
	\caption{(First row) A sequence of Legendrian Reidemeister moves that moves the pattern of a $(-1,n)$-Legendrian cable through a right cusp. It works similar for a left cusp. This isotopy reflects the pattern.	(Second and third rows) A sequence of Legendrian Reidemeister moves that maps the pattern of a $(-1,n)$-Legendrian cable to its reflected pattern, thus proving that they are Legendrian isotopic. This isotopy in the second and third rows is shown entirely in the $1$-jet space of $S^1$: since the satellite operation identifies a neighborhood of the companion knot with this $1$-jet space, this argument establishes that the two diagrams for the cable are also isotopic.}
	\label{fig:cable_welldef}
\end{figure}
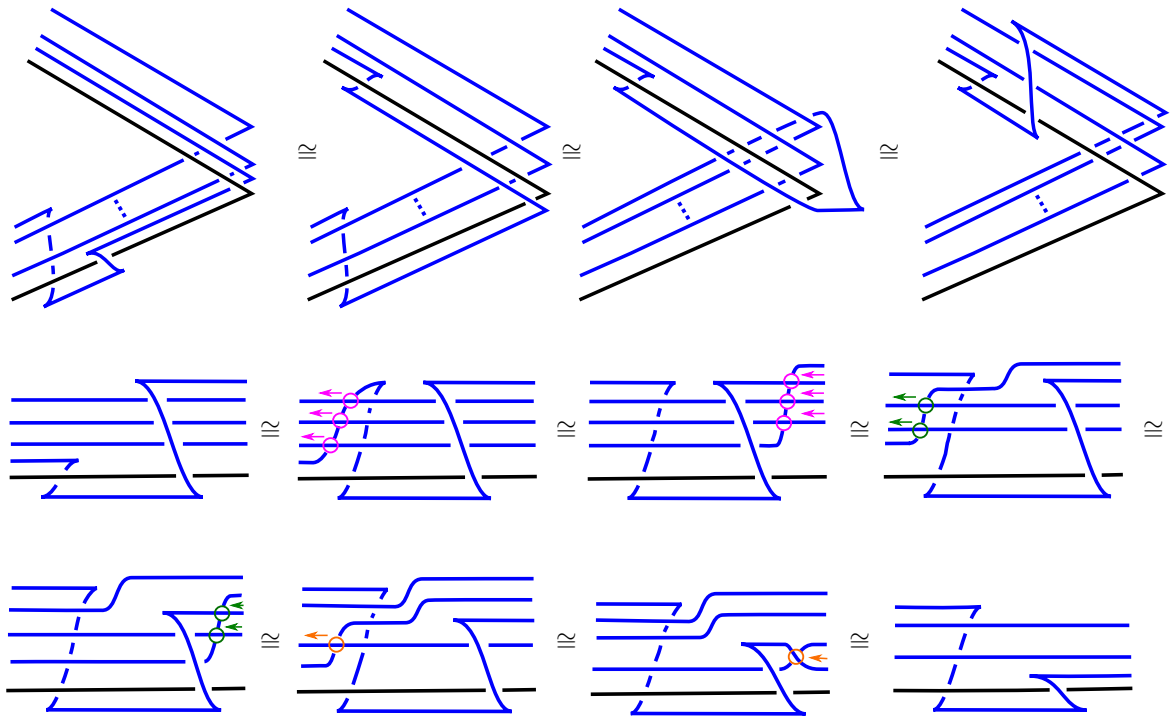

\begin{lem}[Contact handle slides]\label{lem:contactHandleSlide} 
	Let $L$ in $(S^3,\xist)$ be a Legendrian knot along which we perform a contact $(\pm1/n)$-surgery, $J$ its $(\pm1,n)$-Legendrian cable, and $K$ a Legendrian knot in the exterior of $L$. 
	
	Then, the Legendrian knot $K$ seen as a knot in the surgered manifold $L(\pm1/n)$ is isotopic to the Legendrian connected sum $K\#J$\footnote{Here our notion of connected sum is an internal one happening in the surgered manifold. We will explain this in detail in the proof and also argue why $J$ represents a standard Legendrian unknot in $L(\pm1/n)$ and thus the connected sum is well-defined.}. Figure~\ref{fig:contactHandleSlide1} depicts local models for fronts for $K\#J$. In addition, if $K$ initially comes equipped with a contact surgery coefficient $r\in\Q\setminus\{0\}$, then the contact surgery coefficient $r$ is not changing under the isotopy, i.e.\ the contact surgery coefficient of $K\#J$ in $L(\pm1/n)$ is again $r$.
\end{lem}

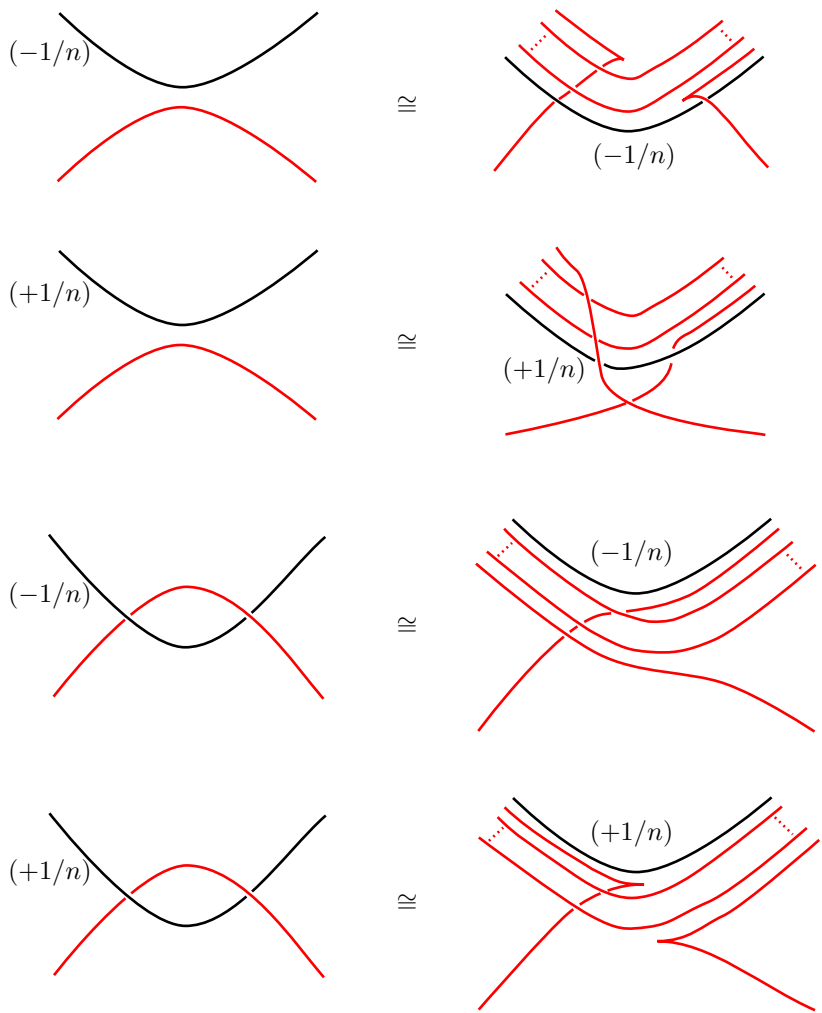
\begin{figure}[htbp]{\small
  \begin{overpic}
  {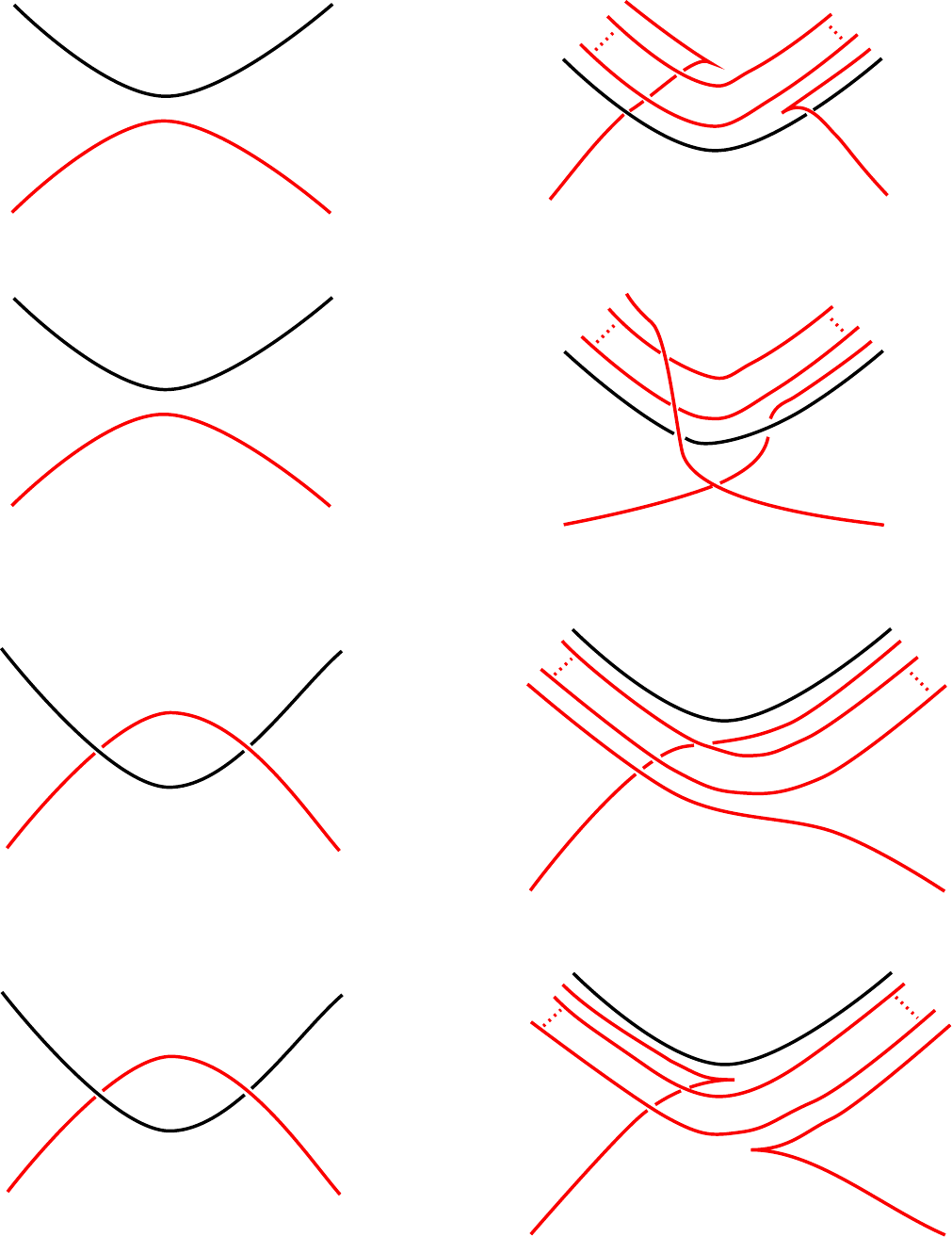}
     \put(-15, 360){$(-1/n)$}
      \put(-15, 269){$(+1/n)$}
      \put(-15, 155){$(-1/n)$}
      \put(-15, 50){$(+1/n)$}
      \put(206, 321){$(-1/n)$}
      \put(172, 240){$(+1/n)$}
      \put(205, 171){$(-1/n)$}
      \put(205, 65){$(+1/n)$}
      \put(132, 340){$\cong$}
      \put(132, 250){$\cong$}
      \put(132, 143){$\cong$}
      \put(132, 38){$\cong$}
  \end{overpic}}
  \caption{Local models for the contact handle slides in the framework of contact $(\pm1/n)$-surgeries. These appear in the statement of Lemma \ref{lem:contactHandleSlide}. Here we write $\cong$ for isotopy of Legendrian knots.}
  \label{fig:contactHandleSlide1}
\end{figure}

\begin{proof}
	First, we show that $J$ represents a standard Legendrian unknot, with $\tb(J)=-1$, in the surgered manifold $L(\pm1/n)$. For that, notice that $J$ can be isotoped to lie on the boundary of a tubular neighborhood $\partial(\nu L)$ of $L$ and represents there the curve $\pm\mu+n\lambda_C$, where $\lambda_C$ represents the contact longitude of $L$. Therefore, in the contact surgered manifold $L(\pm1/n)$, the knot $J$ bounds a meridional disk of the newly glued-in solid torus. We may isotop the boundary of $\nu L$ so that $J$ is a ruling curve that intersects the dividing curves of $\partial (\nu L)$ two times. Thus the twisting relative to $\partial (\nu L)$, which is the same as the twisting relative to the meridional disk, is $-1$. So we see that Thurston-Bennequin invariant of $J$ is $-1$  and thus represents the maximal Thurston-Bennequin invariant Legendrian unknot in $L(\pm1/n)$.
	
The statement of Lemma \ref{lem:contactHandleSlide} is deduced as follows. Given a Legendrian knot $K$ in the exterior of $L$, $K$ can be considered as a Legendrian knot in the surgered manifold $L(\pm1/n)$. Since $J$ represents a Legendrian unknot with $\tb=-1$ in the surgered manifold, the isotopy type of $K$ will not change under taking a Legendrian connected sum with $J$. Here we think of the connected sum as an internal operation. We take a Legendrian band in $L(\pm1/n)$ connecting points on $K$ and $J$. By surgering $K$ and $J$ along that band we get the connected sum as shown in Figure~\ref{fig:ConnectedSum}. Since $J$ is a standard unknot unlinked from $K$ it follows that the isotopy type of $K\#J$ is independent of the choice of the band. Note that although $J$ and $K$ look linked in the surgery diagram $J$ is isotopic to the meridional disk of the newly glued-in solid torus in $L(\pm1/n)$ and thus is unlinked from $K$.

Now, to argue that Figure~\ref{fig:contactHandleSlide1} indeed represents the connected sum $K\#J$, we note that the Legendrian connected sum can come in different incarnations in the front projections. Figure~\ref{fig:ConnectedSum} depicts three different possibilities. The first move in Figure~\ref{fig:contactHandleSlide1} follows directly from the third version of the connected sum, and the second move follows from the second version of the connected sum. For moves three and four, we first modify $J$ by Legendrian isotopies as shown in Figure~\ref{fig:move3} and~\ref{fig:move4}. Then we observe that a part of $J$ is parallel to the part of $K$ shown in Figure~\ref{fig:contactHandleSlide1} on the left and thus perform the third version of the Legendrian connected sum from Figure~\ref{fig:ConnectedSum} yields the claimed diagrams on the right of Figure~\ref{fig:contactHandleSlide1}.
\end{proof}

	\begin{figure}[htbp]{\small
  \begin{overpic}
  {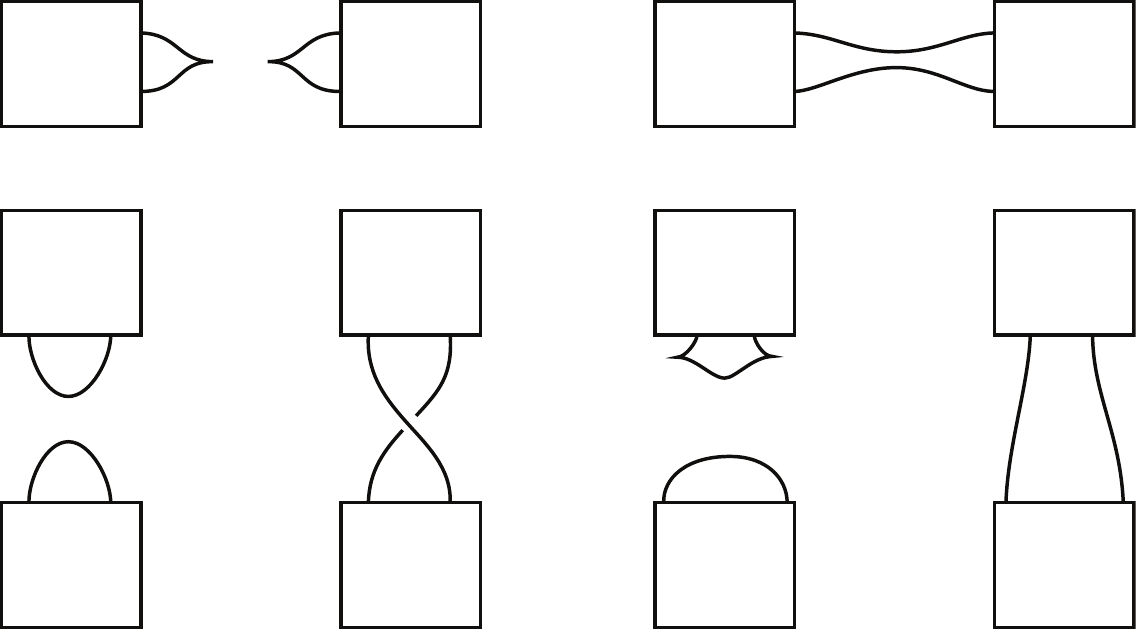}
     \put(17.5, 160){$J$}
      \put(114, 160){$K$}
      \put(206, 160){$J$}
      \put(304, 160){$K$}
      \put(16, 100){$J$}
      \put(16, 16){$K$}
      \put(114, 100){$J$}
      \put(114, 16){$K$}
      \put(205, 100){$J$}
      \put(205, 16){$K$}
      \put(304, 100){$J$}
      \put(304, 16){$K$}
      \put(66, 162){$\#$}
      \put(16, 58){$\#$}
      \put(205, 55.5){$\#$}
      \put(158, 160){$=$}
      \put(65, 60){$=$}
      \put(255, 60){$=$}
  \end{overpic}}
  \caption{Three different incarnations of the Legendrian connected sum.}
  \label{fig:ConnectedSum}
\end{figure}

\begin{figure}[htbp]{\small
  \begin{overpic}
  {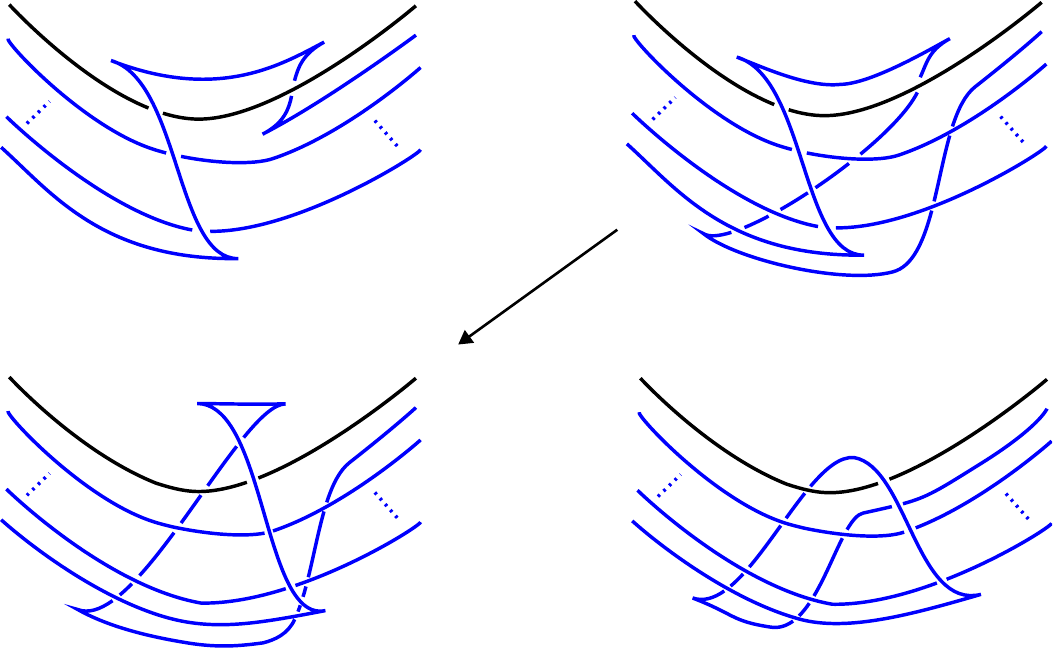}
     \put(15, 185){$L$}
      \put(82, 185){$(-1/n)$}
      \put(56, 170){\color{blue} $J$}
      \put(197, 185){$L$}
      \put(261, 185){$(-1/n)$}
      \put(238, 170){\color{blue} $J$}
      \put(150, 158){$\cong$}
      \put(145, 110){$\cong$}
     \put(15, 77){$L$}
      \put(82, 77){$(-1/n)$}
      \put(100, 10){\color{blue} $J$}
      \put(197, 77){$L$}
      \put(265, 77){$(-1/n)$}
      \put(290, 10){\color{blue} $J$}     
         \put(150, 43){$\cong$}
  \end{overpic}}
  \caption{Performing a Legendrian connected sum of the Legendrian knot $K$ with the Legendrian cable $J$. This yields the third move from Figure~\ref{fig:contactHandleSlide1}.}
  \label{fig:move3}
\end{figure}
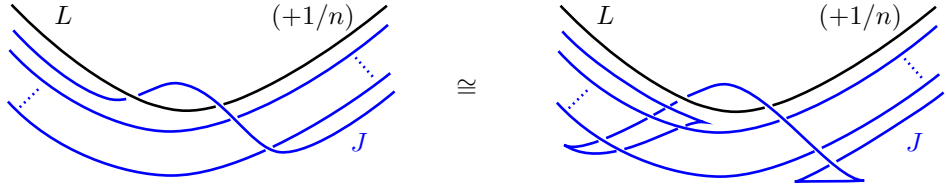
\begin{figure}[htbp]{\small
  \begin{overpic}
  {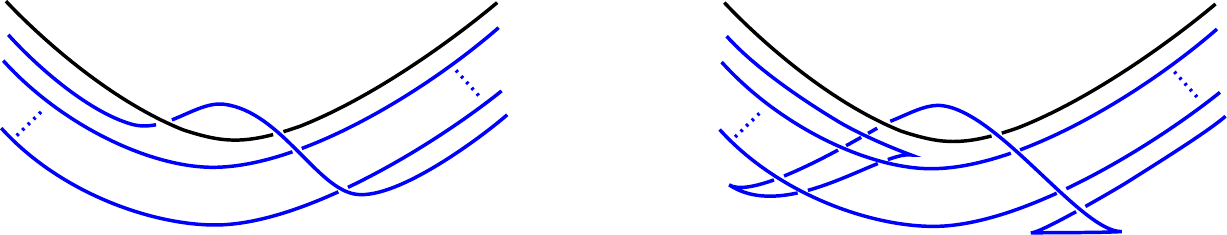}
     \put(18, 60){$L$}
      \put(100, 60){$(+1/n)$}
      \put(130, 12){\color{blue} $J$}
      \put(223, 60){$L$}
      \put(307, 60){$(+1/n)$}
      \put(340, 12){\color{blue} $J$}
      \put(170, 33){$\cong$}
  \end{overpic}}
  \caption{Performing a connected sum of $K$ with $J$ yields the last move from Figure~\ref{fig:contactHandleSlide1}.}
  \label{fig:move4}
\end{figure}

For the record, the name {\it handle slide} might be a bit misleading because the rational surgery does {\it not} correspond to a $4$-dimensional handle attachment. However, we are isotoping the knot $K$ over the meridional disk of the newly glued-in solid torus corresponding to the surgery knot $L$. Since the same is happening if one does an actual handle slide, the operation at hand can be seen as a generalization of a handle slide, and thus the terminology.

\begin{rem} A few minor comments might be in order:
	
\begin{enumerate}
	\item There are other ways to perform the connected sum $K\#J$ in a front projection: another example is shown in~\cite{DingGeiges09}, where the connected sum is performed using the first incarnation from Figure~\ref{fig:ConnectedSum}.	
	\item We observe that there is a contactomorphism of the standard contact structure on $\R^3$ that is given in the front projection by a reflection on the horizontal axis and thus we can also reflect the front projections from Figures~\ref{fig:contactHandleSlide1} along the horizontal axis and obtain more front versions of contact handle slides. If we choose our model of the standard contact structure to be $(\R^3,\ker dz - y\, dx)$ such a contactomorphism is given by $(x,y,z)\mapsto(x,-y,-z)$.
	\item Note that a contact handle slide for general integer surgery is not possible. Indeed, the $(\pm p,q)$-Legendrian cable of $L$ represents a $\tb=-p$ unknot in $L(\pm p/q)$.	
	\item It is also possible to express the contact handle slides in Lemma \ref{lem:contactHandleSlide}, over a Legendrian knot with coefficient $\pm 1/n$, as an $n$-fold sequence of handle slides over $(\pm1)$-framed Legendrian push-offs of the surgery knot.\hfill$\Box$
\end{enumerate}
\end{rem}

\subsection{Contact surgeries on the unknot}
Let $U$ in $(S^3,\xist)$ be the standard Legendrian unknot, with $\tb(U)=-1$. In the proof of Theorem~\ref{thm:main}, and in Section~\ref{starmove}, we use contact $(1+1/n)$-surgeries on $U$ in $(S^3,\xi_{st})$ for $n>0$. In the following result, we record the possible resulting contact structures and the way certain Legendrian knots can be presented in the surgery.

\begin{thm}\label{surgerytoS3}
Let $U$ in $(S^3,\xist)$ be the standard Legendrian unknot, and let $(U(1+1/n),\xi_k)$, $k\in \{-n, -n+2, \ldots, n\}$, denote the $(n+1)$ $($possibly$)$ distinct contact $(1+1/n)$-surgeries $\xi_k$ along $U$ in $(S^3,\xist)$.

Then, the contact manifold $(U(1+1/n),\xi_k)$ is contactomorphic to $(S^3,\xist)$, for any $k\in \{-n, -n+2, \ldots, n\}$. In addition, if $L$ is a Legendrian link in the exterior of $U$ linking $U$ as shown on the left of Figure~\ref{fig:speicalsurg}, then under the contactomorphism from $U(1+1/n)$ to $(S^3,\xist)$, $L$ changes as indicated on the right of Figure~\ref{fig:speicalsurg}. 
\end{thm}

\begin{proof}
	Contact $(1+1/n)$-surgery for $n>0$ along a Legendrian knot is contactomorphic to contact $(+1)$-surgery along the same knot followed by a contact $(-1)$-surgery along its $n$-fold stabilization. The $(n+1)$ contact structures $\xi_k$ correspond to the different possible choices of the $n$-fold stabilizations~\cite{DingGeigesStipsicz04}. That $(U(1+1/n),\xi_k)$ are all contactomorphic to $(S^3,\xist)$ follows now from the fact that $U(+1)$ is the fillable contact structure on $S^1\times S^2$ and that contact $(-1)$-surgery preserves fillability. Hence we get a tight contact structure on $S^3$, which must be the standard tight contact structure. 
	
	To see how the Legendrian link $L$ in the exterior of $U$ changes under this contactomorphism, we perform handle slides (the first move from Figure~\ref{fig:contactHandleSlide1} for $n=1$, cf.~\cite{DingGeiges09}) as shown in the middle of Figure~\ref{fig:speicalsurg}. The last figure is obtained by an isotopy and by canceling the unknots.
\end{proof}

\begin{figure}[htbp] 
{\small
 \begin{overpic}
  {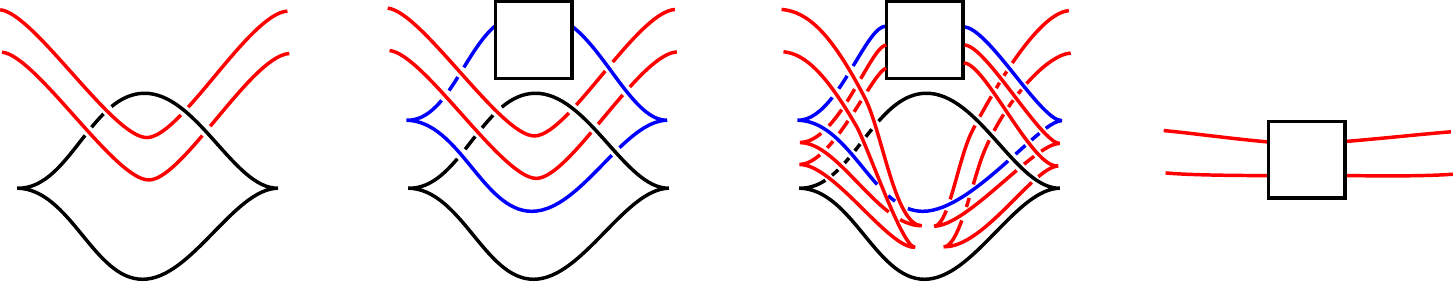}
     \put(60, 0){$(1+\frac1n)$}
      \put(173, 0){$(+1)$}
      \put(187, 53){\color{blue} $(-1)$}
      \put(151, 66){$n$}
      \put(287, 0){$(+1)$}
      \put(305, 50){\color{blue} $(-1)$}
       \put(265, 66){$n$}
        \put(375, 33){$n$}
      \put(90, 25){$\cong$}
	\put(210, 25){$\cong$}
	\put(320, 25){$\cong$}
  \end{overpic}}
	\caption{On the left, the unknot $U$ in black and a linking Legendrian link $L$ in red. On the right, the image of $L$ under contact $(1+1/n)$-surgery on $U$. The box labeled by $n$ is our notation for an $n$-fold stabilization of the knot running through that box, see Figure~\ref{fig:def_box}. The signs of the stabilizations are determined by the precise contact structure on $U(1+1/n)$.}
	\label{fig:speicalsurg}
\end{figure}

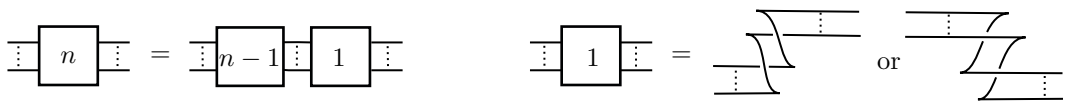
\begin{figure}[htbp] 
	{\small
		\begin{overpic}
			{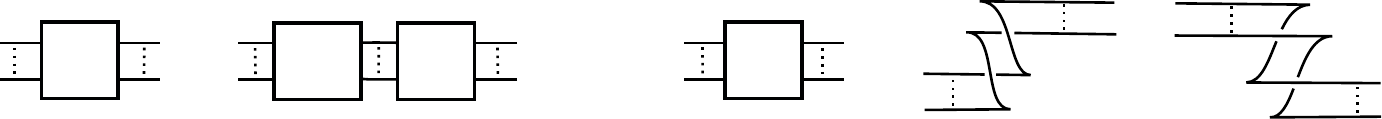}
			\put(20, 14){$n$}
			\put(80, 13){$n-1$}
			\put(123, 13){$1$}
			\put(219, 13){$1$}
			\put(54, 15){$=$}
			\put(251, 15){$=$}
			\put(329, 12){or}
	\end{overpic}}
	\caption{The definition of the box labeled $n$, which is for example used in Figure~\ref{fig:speicalsurg}.}
	\label{fig:def_box}
\end{figure}

\begin{rem}\label{effectstab} In the proof of Theorem \ref{thm:main} we use that Theorem~\ref{surgerytoS3} implies that any $n$-fold stabilization of a Legendrian knot can be affected by some contact $(1+1/n)$-surgery on its meridian.\hfill$\Box$
\end{rem}

\section{Contact annulus presentations}\label{sectionannulustwist}

In this section we prove Theorem \ref{thm:main} and, in particular, develop the notion of a contact annulus twist and pre-Lagrangian annulus presentations. The statements in Theorem \ref{thm:main} are then proven by applying these constructions and an extension result for contactomorphisms. The section is organized in parallel to the developments in the smooth category, as follows.

In the smooth context, the first infinite family of knots in $S^3$ sharing the same $0$-surgery were created in \cite{Osoinach06}. In that paper, Osoinach introduced an annulus twist to construct these examples. First, we develop that construction in the contact topological setting in Subsection~\ref{annulustwist}. Osoinach's construction was later generalized by Abe, Jong, Omae, and Takeuchi in \cite{AbeJongOmaeTakeuchi13}, where the notion of an annulus presentation for a knot was introduced. Second, in Subsection~\ref{atwist}, we generalize this annulus presentation to contact and symplectic topology. Third, in Subsection~\ref{strace} we develop results based on the work of Akbulut \cite{Akbulut77}, and in particular its application in \cite{AbeJongLueckeOsoinach15, AbeJongOmaeTakeuchi13}, to show that some of the previously constructed Legendrian knots have the same Stein trace. 

These constructions all produce knots with the same $0$-surgery, and this translates into the fact that all the Legendrian knots that we produce have Thurston--Bennequin invariant $1$. In the smooth context, the annulus twist was also modified in~\cite{AbeJongLueckeOsoinach15} in order to produce examples with any integer surgery coefficient. We adapt this construction to the contact category in Subsections~\ref{starmove} and produce pairs of Legendrian knots with any Thurston--Bennequin invariant that share equivalent Stein traces. In Subsection~\ref{sec:other} we generalize other smooth constructions of pairs of knots with the same surgery to the contact and symplectic framework. 

\subsection{Contact annulus twists}\label{annulustwist}
The first crucial ingredient in the proof of Theorem~\ref{thm:main} is a contact version of the smooth annulus twist. We expect this contact annulus twisting to be useful in future situations, and we present it as the separate Lemma \ref{lem:twist}. First, we use the following notion.

\begin{defi}
Let $(M,\xi)$ be a contact manifold and $K$ in $(M,\xi)$ a Legendrian knot. A \textit{pre-Lagrangian annulus} $A$ in $(M,\xi)$ in the Legendrian knot type $K$ is any embedded annulus formed by flowing the Legendrian knot $K$ for a short time under a Reeb flow associated to a contact 1-form for $\xi$.\hfill$\Box$
\end{defi}

Note that a pre-Lagrangian annulus $A$ in the Legendrian knot type $K$ is foliated by copies of $K$; we also denote the boundary of $A$ by $K\cup K'$, where $K'$ always denotes a (small) Reeb push-off of $K$. The contact annulus twist reads as follows.

\begin{lem}[Contact annulus twist]\label{lem:twist}
Let $K$ in $(M,\xi)$ be a Legendrian knot and $L$ a Legendrian link in the exterior of $K$. Consider a pre-Lagrangian annulus  $A$ in the knot type $K$, with boundary $\partial A=K\cup K'$, such that $L$ intersects $A$ transversely in its interior $\mathring{A}$.

For any $n\in \N$ and a choice of contact surgery coefficients on $L$, let $(M_\pm,\xi_\pm)$ be the result of contact $(\pm 1/n)$-surgery on $K$, contact $(\mp 1/n)$-surgery on $K'$ and contact surgery on $L$ with the chosen coefficients, then
\begin{itemize}
	\item[(i)] The manifold $(M_+,\xi_+)$ is contactomorphic to the result of contact surgery on the link $L'$ obtained form $L$ by performing a connected sum with the $(-1,n)$-Legendrian cable of $K'$ at each intersection point of $L$ with $A$. In addition, the contact surgery coefficients on $L'$ are the same as those on $L$. See Figure~\ref{fig:twist} (top).\\
	
	\item[(ii)] The manifold $(M_-,\xi_-)$ is contactomorphic to the result of contact surgery on the link $L''$ obtained form $L$ by performing a connected sum with the $(1,n)$-Legendrian cable of $K'$ at each intersection point of $L$ with $A$. Similarly, the contact surgery coefficients on $L''$ are the same as those on $L$. See Figure~\ref{fig:twist} (bottom).
\end{itemize}
By definition, $L'$ or $L''$ is said to be obtained from $L$ by an \dfn{n-fold contact annulus twist}. 
\end{lem}

\begin{figure}[htbp]{\small
  \begin{overpic}
  {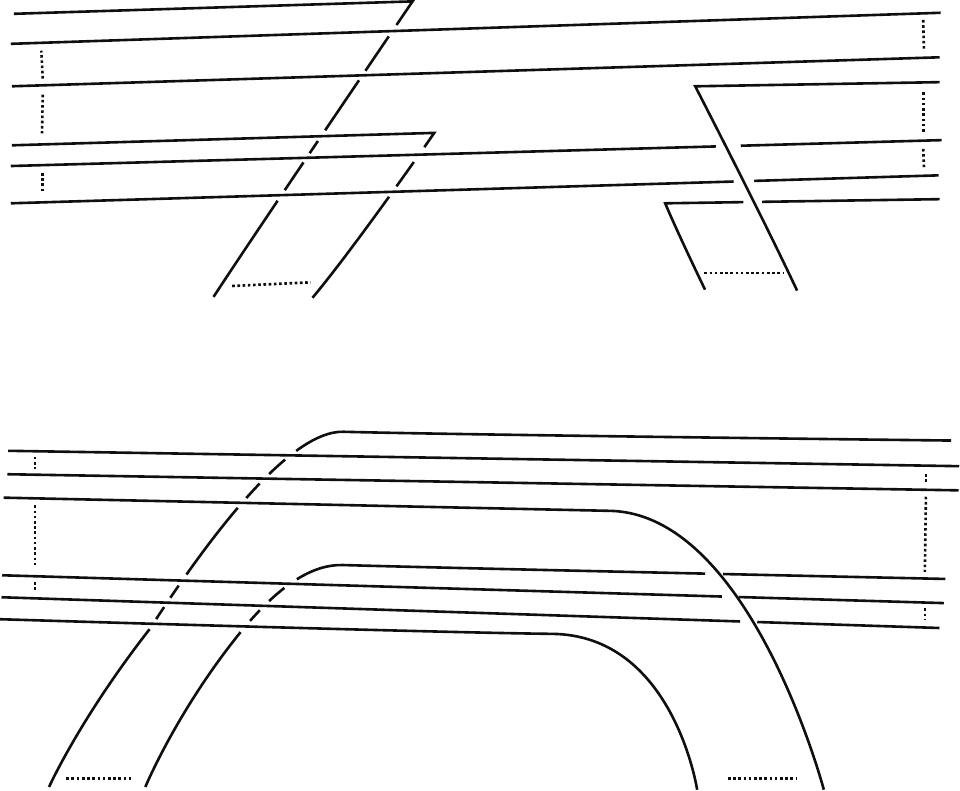}
     \put(180, 150){$L'$}
      \put(50, 150){$(n)$}
     \put(180, 10){$L''$}
      \put(52, 10){$(n)$}
  \end{overpic}}
  \caption{The knots $L'$ and $L''$ that feature in the Contact Annulus Twist stated in Lemma \ref{lem:twist}. They are obtained by contact handle slides, as in the proof of the lemma. Away from the shown part of the picture we have $n$ strands parallel to the annulus for every intersection point of $L$ with the annulus. For concrete examples we refer to Figures~\ref{fig:exampleContactomorphism},~\ref{fig:examplenmove} and~\ref{fig:notExtend}.}
  \label{fig:twist}
\end{figure}
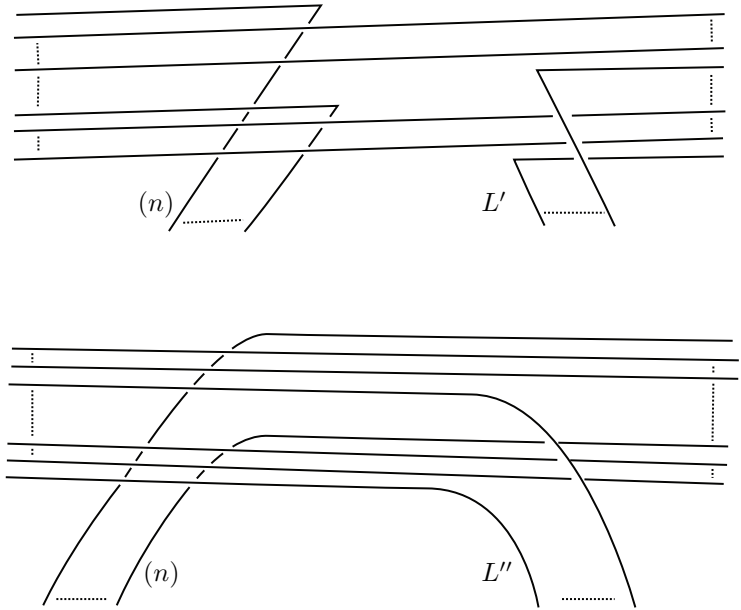

\begin{proof}
In case that the Legendrian $L$ in $(M,\xi)$ does not intersect $A$, this is the Cancellation Lemma, as stated in Lemma~\ref{cancellationlemma}. For each intersection point of $L$ with $A$, we use Lemma~\ref{lem:contactHandleSlide} to perform a contact handle slide of $L$ over $K'$ along an arc contained in $A$. See Figure~\ref{fig:annulus}. Then the new link $L'$ has no intersections with $A$ and thus $K$ and $K'$ can be canceled, obtaining the claimed result. Alternatively, we can also slide $L$ over $K$ and get the same surgery diagram. A similar argument works for the manifold $(M_-,\xi_-)$. 
\end{proof}

\begin{figure}[htbp]{\small
		\begin{overpic}
			{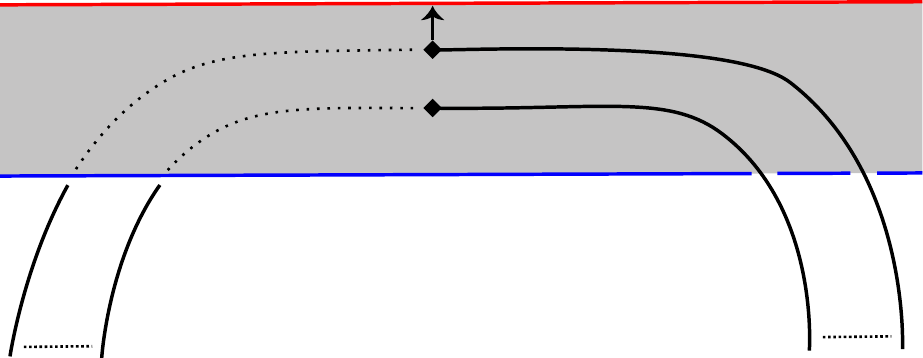}
			\put(-9, 75){$A$}
			\put(270, 51){\color{blue} $K$}
			\put(270, 100){\color{red} $K'$}
			\put(220, 10){$L$}
			\put(35, 10){$r$}
	\end{overpic}}
	\caption{The pre-Lagrangian annulus $A$ and its intersections with $L$. The surgery coefficients on $L$ are indicated by $r$. The black arrow indicates a handle slide. Note that we are only drawing a local picture of $A$ and $L$, this is part of a larger contact surgery diagram in $(M,\xi)$.}
	\label{fig:annulus}
\end{figure}
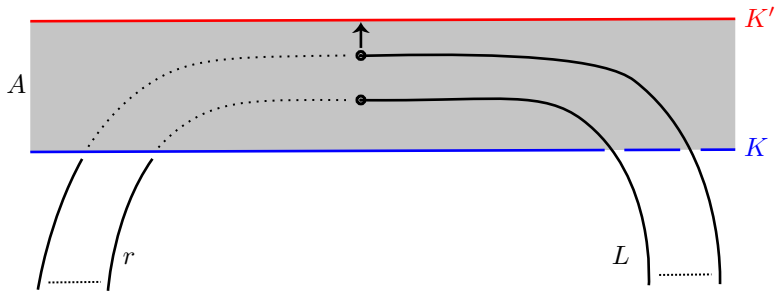

\subsection{Contact annulus presentations}\label{atwist}

Let $K$ in $(M,\xi)$ be a Legendrian knot and a pre-Lagrangian annulus $A$ in the knot type $K$, with boundary $\partial A=K\cup K'$. Let $\gamma$ in $(M,\xi)$ be a Legendrian arc that begins on $K$, ends on $K'$ and transversely intersects $A$ in its interior. Form a knot $L_{A,\gamma}$ in $(M,\xi)$ by removing a small neighborhood of the end points of $\gamma$ from $K$ and $K'$, and then taking $\gamma$ and a push-off of $\gamma$ to create a Legendrian band sum of $K$ and $K'$. By definition, the pair $(A,\gamma)$ of a pre-Lagrangian annulus $A$ and the Legendrian arc $\gamma$ is said to be a \textit{contact annulus presentation} of the Legendrian knot $L_{A,\gamma}$. This is illustrated in Figure~\ref{fig:annuluspres}. In this manuscript we only consider the case where the annulus $A$ union the band along $\gamma$ is an (immersed) oriented surface; it is then verified that $\tb(L_{A,\gamma})=1$ and $\rot(L_{A,\gamma})=0$. 
 
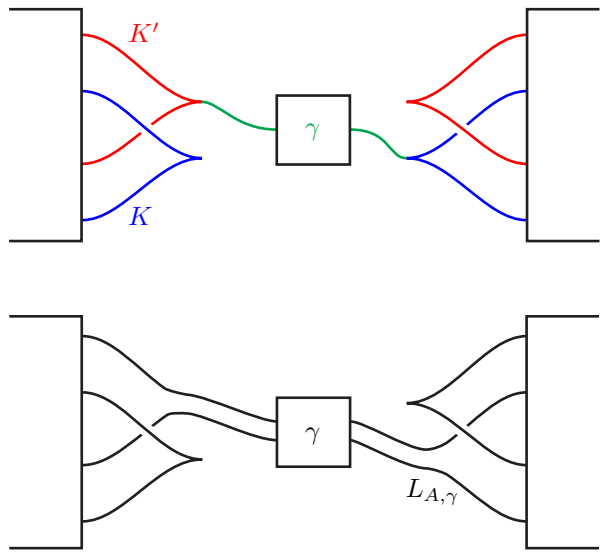
\begin{figure}[htbp]{\small
  \begin{overpic}
  {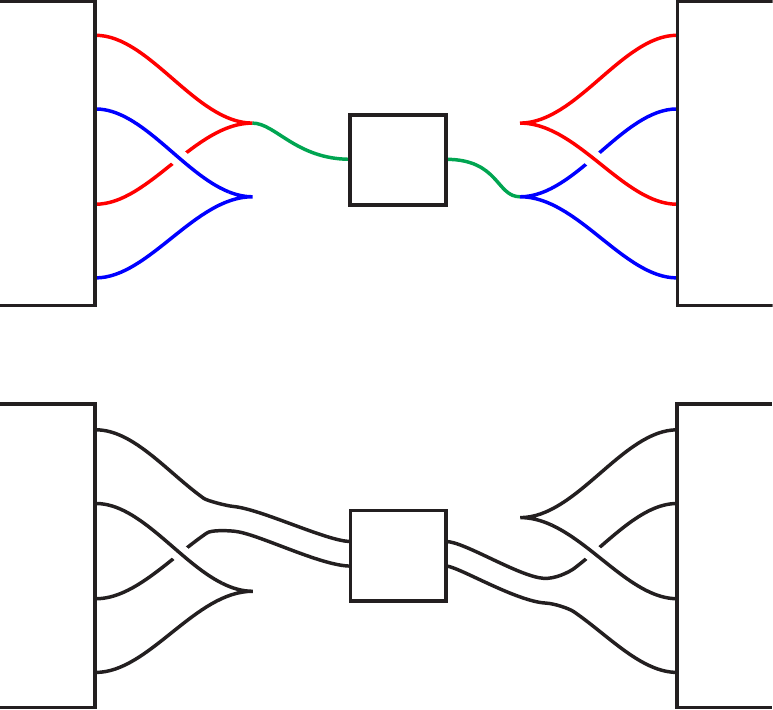}
    \put(112, 157){\color{darkgreen}$\gamma$}
    \put(45, 192){\color{red} $K'$}
    \put(45, 123){\color{blue} $K$}
     \put(112, 42){$\gamma$}
     \put(150, 20){$L_{A,\gamma}$}
  \end{overpic}}
  \caption{Contact annulus presentation of a Legendrian $L$. In the top figure, we have depicted $K$ and its Legendrian push-off $K'$, co-bounding an annulus $A$, and the Legendrian arc $\gamma$. The box around $\gamma$ indicates that $\gamma$ is allowed to intersect the annulus $A$ again. In the bottom figure, we have depicted the resulting Legendrian knot $L_{A,\gamma}$, obtained as the band sum of $K$ and $K'$ along $\gamma$.}
  \label{fig:annuluspres}
\end{figure}

Given a contact annulus presentation $(A,\gamma)$ and associated knot $L_{A,\gamma}$, we denote by $A'\subset \mathring{A}$ a pre-Lagrangian sub-annulus in the interior of $A$ that contains all the intersections of $\gamma$ with $A$, as well as the band about $\gamma$ used to form $L_{A,\gamma}$. Let $K_1,K_2$ be the boundary components of $A'$, $\partial A'=K_1\cup K_2$, with $K_1$ closer on $A$ to $K$ than $K_2$. The following lemma, with the notations as above, is crucial to show that the Legendrian knots in Theorem~\ref{thm:main}.(i) have contactomorphic $(-1)$-surgeries.

\begin{lem}\label{secondannulus}
 The Legendrian knots $K_1$ and $K_2$, considered as Legendrian knots in the surgered contact manifold $L_{A,\gamma}(-1)$, are Legendrian isotopic.
\end{lem}

\begin{proof}
 We represent our Legendrian knots in surgery diagrams for $(M,\xi)$, drawn as front projections, and the Legendrian arc $\gamma$ will be attached to $K$ and $K'$ at cusps in the front diagram. Figure~\ref{fig:part1} (top) shows the Legendrian knot $L_{A,\gamma}$ and the two Legendrian knots, $K_1$ and $K_2$, that bound the pre-Lagrangian sub-annulus $A'$.  We perform a contact $(-1)$-surgery along $L_{A,\gamma}$. After performing a handle slide of $K_2$ over $L_{A,\gamma}$, we obtain Figure~\ref{fig:part1} (bottom). Note that the Legendrian knot $L_{A,\gamma}$ is made out of four Legendrian arcs: $\widetilde{K'}$ and $\widetilde{K}$, where $\widetilde{K}$ denotes $K$ with a small neighborhood of the $\partial \gamma$ removed from it (and similarly for $\widetilde{K'}$), and two copies of $\gamma$, that we denote $\gamma'$ and $\gamma''$.
 
	\begin{figure}[htbp]{\small
			\begin{overpic}
				{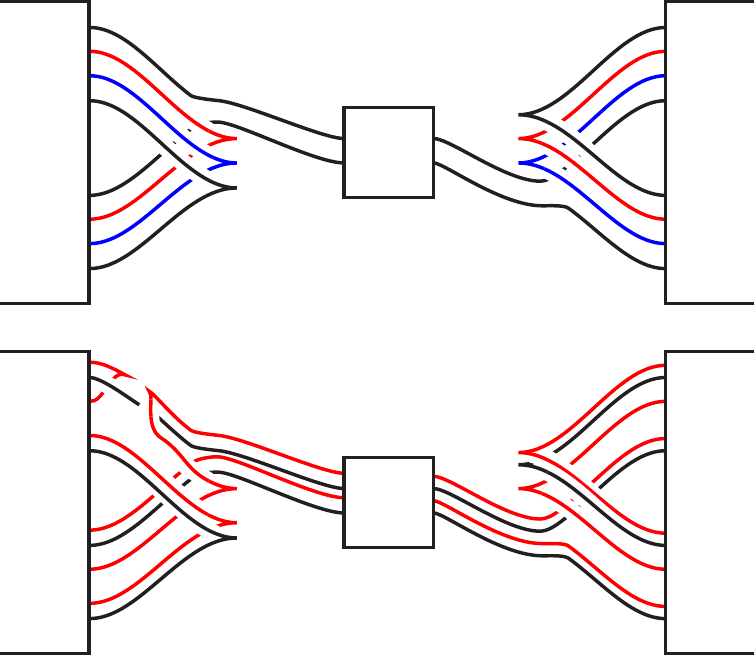}
				\put(109, 143){$\gamma$}
				\put(11, 172){\color{red} $K_2$}
				\put(11, 117){\color{blue} $K_1$}
				\put(158, 110){$(-1)$}
				\put(70, 164){$L_{A,\gamma}$}
				\put(109, 41){$\gamma$}
				\put(70, 65){$L_{A,\gamma}$}
				\put(158, 8){$(-1)$}
		\end{overpic}}
		\caption{The top diagram is $L_{A,\gamma}$ together with the Legendrian knots $K_1$ and $K_2$. The bottom figure is the result of Legendrian surgery on $L_{A,\gamma}$  and sliding $K_2$ over $L_{A,\gamma}$.}
		\label{fig:part1}
	\end{figure}
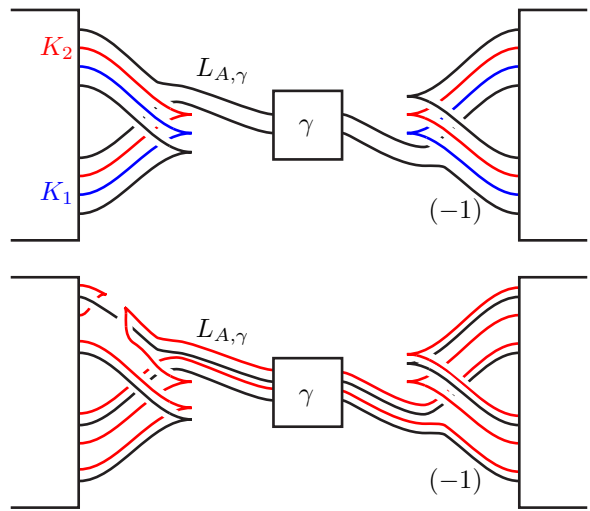

The Legendrian knot $K_2$, after the slide, and the Legendrian arc $\widetilde{K'}$ are joined by a cusp, as depicted in the upper left of the front diagram in Figure~\ref{fig:part1} (bottom). Since $K_2$ and $K'$ are parallel, there is a Legendrian isotopy ending in the upper diagram of Figure~\ref{fig:part2}. 
	\begin{figure}[htbp]{\small
			\begin{overpic}
				{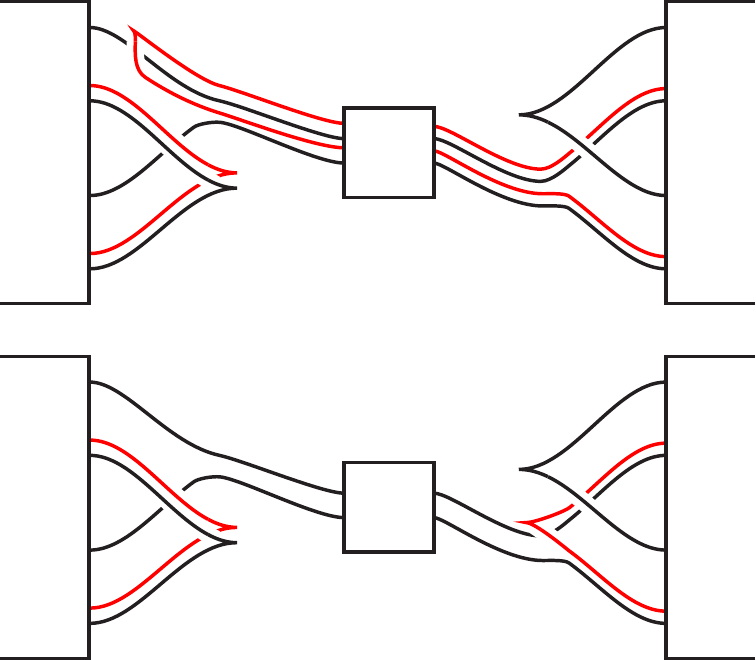}
				\put(109, 144){$\gamma$}
				\put(109, 41){$\gamma$}
				\put(70, 169){$L_{A,\gamma}$} 
				\put(70, 62){$L_{A,\gamma}$}
				\put(160, 8){$(-1)$}
				\put(160, 110){$(-1)$}
		\end{overpic}}
		\caption{Steps in the Legendrian isotopy from $K_2$ to $K_1$, as used in the proof for Lemma \ref{secondannulus}.}
		\label{fig:part2}
	\end{figure}
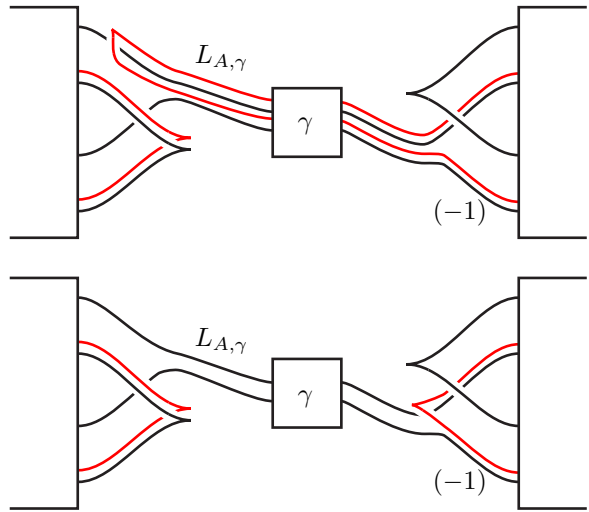
Now, the Legendrian curves $\gamma'$ and $\gamma''$ are Legendrian push-offs of $\gamma$, and connected with a cusp as in the upper diagram in Figure~\ref{fig:part2}. Thus, there is a Legendrian isotopy to the bottom diagram in Figure~\ref{fig:part2}, which is isotopic to $K_1$, as required. 
\end{proof}

Let us now readily deduce from Lemma \ref{secondannulus} that contact annulus twists preserve the contact type of the contact $(-1)$-surgery.

\begin{thm}\label{mainannulustwist}
Let $(M,\xi)$ be a contact 3-manifold and $(A,\gamma)$ a contact annulus presentation. Consider the Legendrian knots $L_n$ in $(M,\xi)$, $n\in\N_0$, obtained by performing an $n$-fold contact annulus twist using the knots $\partial A'$ applied to $L_{A,\gamma}$. Then, the contactomorphism type of the Legendrian $(-1)$-surgery on $L_n$ in $(M,\xi)$ is independent of $n$, i.e. the contact $(-1)$-surgery on $L_n$ is contactomorphic to the contact $(-1)$-surgery on $L_{A,\gamma}$ for all $n\in\N_0$.
\end{thm}

\begin{proof} Let $(M',\xi')$ be the result of a Legendrian surgery on $L_{A,\gamma}$. By Lemma \ref{secondannulus} above and Lemma~\ref{cancellationlemma}, performing a Legendrian surgery on $L_{A,\gamma}$, and then a contact $(\pm 1/n)$-surgery on $K_1$ and contact $(\mp 1/n)$-surgery on $K_2$, results in $(M',\xi')$. At the same time, performing the same surgeries on $K_2$ and $K_1$ before the Legendrian surgery on $L_{A,\gamma}$ results in the Legendrian knots $L_n$, and thus the Legendrian surgery on $L_n$ also yields the contact 3-manifold $(M',\xi')$. 
\end{proof}

\begin{rem}\label{explicitdiffeom}
It is reasonable to denote $L_0:=L_{A,\gamma}$, so that $n\in\N_0$ includes the $n=0$ case. For future use, we notice that there is an explicit contactomorphism from the Legendrian surgery on $L_0$ to the Legendrian surgery on $L_n$. Indeed, it is given by starting with Legendrian surgery on $L_0$, adding a canceling pair of surgeries along $K_1$ and $K_2'$ (from Lemma~\ref{secondannulus}), then sliding $K_2'$ over $L_0$ to obtain $K_2$, and finally performing the contact annulus twist, as in Lemma~\ref{lem:twist}, using $K_1$ and $K_2$. An example is shown in Figure~\ref{fig:exampleContactomorphism}.\hfill$\Box$
\end{rem}

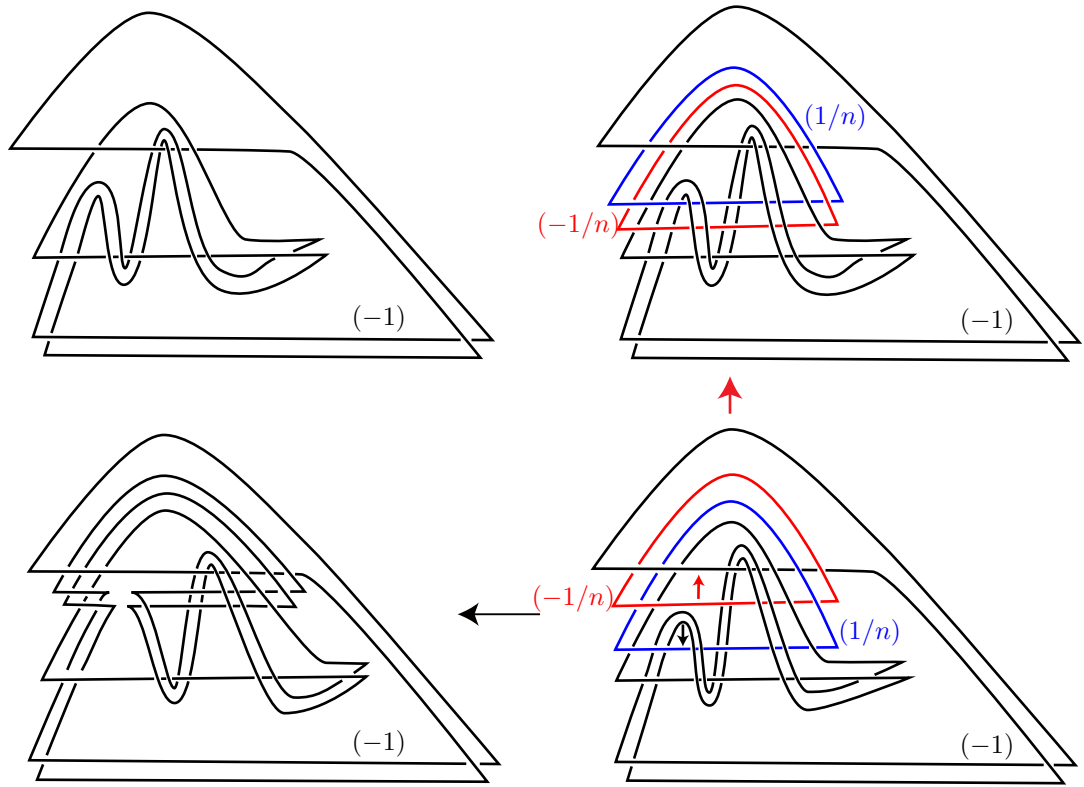
\begin{figure}[htbp] 
{\small
  \begin{overpic}
  {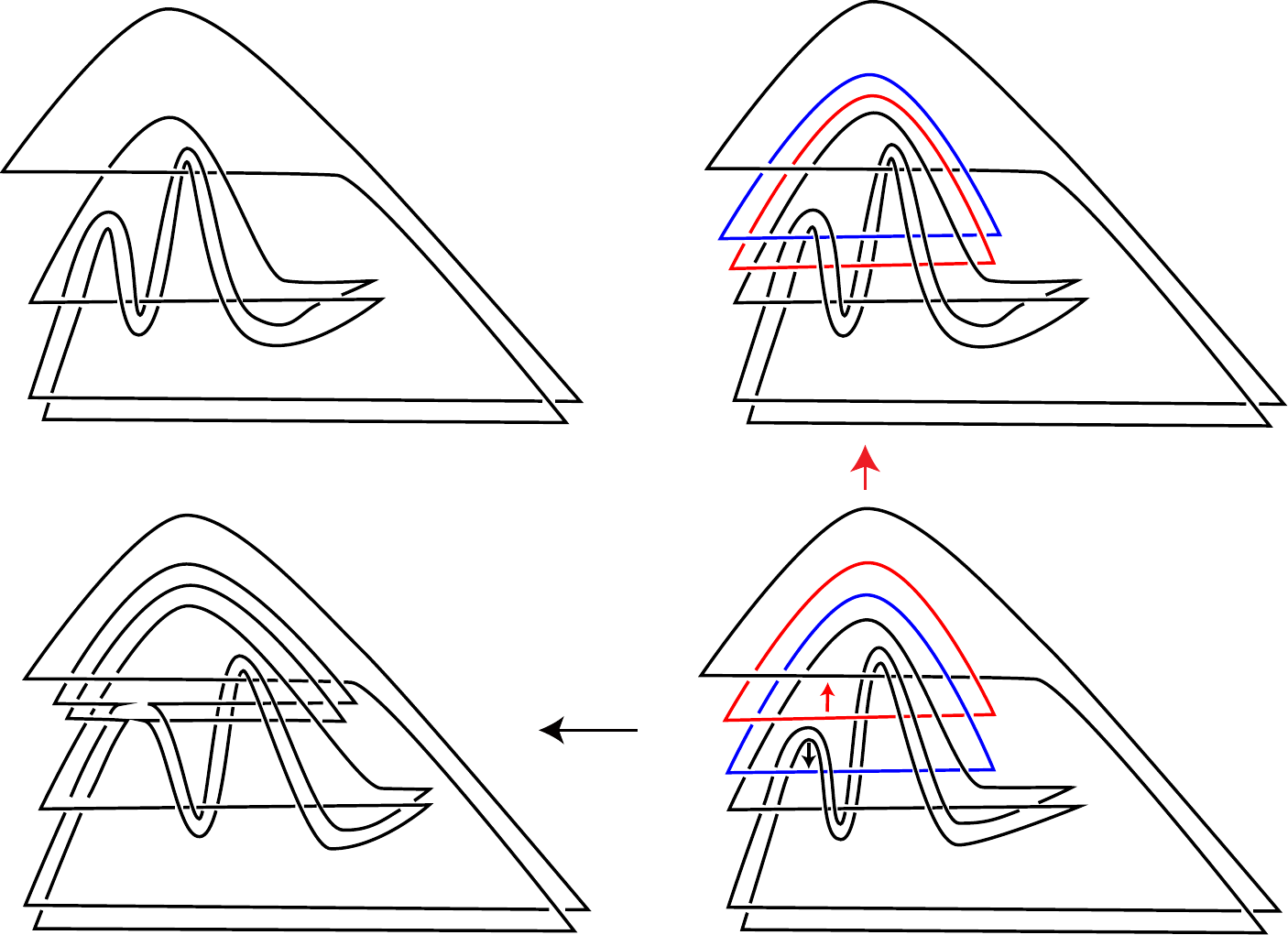}
    \put(130, 174){$(-1)$}
     \put(360, 173){$(-1)$}
     \put(301, 250){\color{blue}$(1/n)$}
      \put(200, 209){\color{red}$(-1/n)$}
        \put(130, 13){$(-1)$}
     \put(360, 12){$(-1)$}
     \put(314, 55){\color{blue}$(1/n)$}
      \put(198, 68	){\color{red}$(-1/n)$}
  \end{overpic}}
	\caption{The top left figure shows contact $(-1)$ surgery along the Legendrian knot $L_0$. In the top right we have introduced a canceling pair. In the bottom right surgery diagram we can perform a handle slide as indicated with the red arrow to arrive at the top right picture. On the other hand we can slide the two black arcs as indicated with the black arrow to get Legendrian surgery along $L_n$ (depicted in the bottom left for $n=1$). The composition of these contact Kirby moves yields an explicit contactomorphism $\phi\colon L_0(-1)\rightarrow L_n(-1)$.}
	\label{fig:exampleContactomorphism}
\end{figure}

The construction in Theorem~\ref{mainannulustwist} produces many families of Legendrian knots $L_n$ in $(M,\xi)$ on which Legendrian surgery produces the same contact 3-manifold. That said, it is not clear that these Legendrian knots $L_n$ are necessarily all distinct. In fact, there are situations where all these Legendrian knots $L_n$ are isotopic. Let us now show that, with the appropriate choice of contact annulus presentation $(A,\gamma)$, they are indeed all non-isotopic.

\begin{thm}\label{thm:infinitesurgeries}
Let $(S^3,\xi_{st})$ be the standard contact 3-sphere. There exists a contact annulus presentation $(A,\gamma)$ such that the $n$-fold contact annulus twists $L_n$ in $(S^3,\xi_{st})$ are pairwise not Legendrian isotopic. In particular, there exists an infinite family of distinct Legendrian knots $L_n$ in $(S^3,\xi_{st})$ whose contact $(-1)$-surgeries are all contactomorphic.
\end{thm}

\begin{proof} Consider the Legendrian knot $L_{A,\gamma}$ in $(S^3,\xi_{st})$ in Figure~\ref{fig:mainex}: the corresponding pre-Lagrangian annulus $A$ is formed from a maximal Thurston--Bennequin unknot, and $\gamma$ is readily deduced from the picture. Figure~\ref{fig:mainex} also depicts the Legendrian knots $K_1$ and $K_2$ in $(S^3,\xi_{st})$, that we can use to perform the contact annulus twists.
\begin{figure}[htbp]{\small
  \begin{overpic}
  {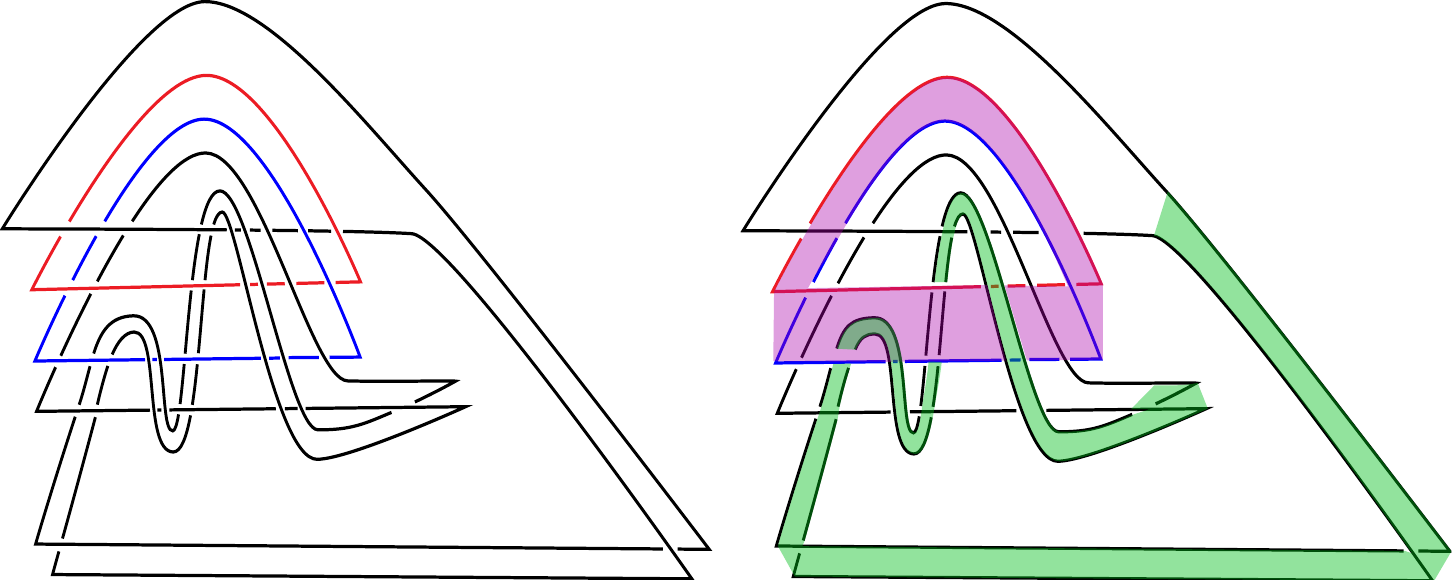}
    \put(105, 90){\color{red}$K_2$}
     \put(135, 105){$L_{A,\gamma}$}
     \put(105, 65){\color{blue}$K_1$}
     
     \put(205, 70){\color{purple}$A$}
     \put(355, 65){\color{darkgreen}$\gamma$}
  \end{overpic}}
  \caption{Left: The Legendrian knot $L_{A,\gamma}$ in the proof of Theorem \ref{thm:infinitesurgeries}. Right: The pre-Lagrangian annulus $A$ and the Legendrian band $\gamma$ used to construct $L_{A,\gamma}$.}
  \label{fig:mainex}
\end{figure}
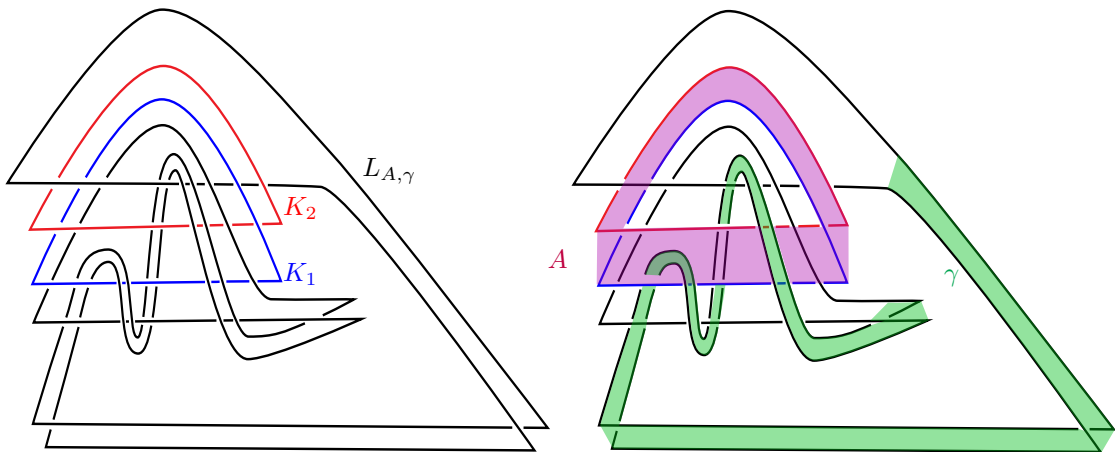

Let us denote $L_0:=L_{A,\gamma}$, as in Remark \ref{explicitdiffeom}, and denote by $L_n$ in $(S^3,\xi_{st})$ the image of $L_0$ under contact $(1/n)$-surgery on $K_1$ and contact $(-1/n)$-surgery on $K_2$. The resulting knots $L_0$ and $L_1$ are shown in Figure~\ref{fig:primeex}. By Theorem~\ref{mainannulustwist}, the contactomorphism type of the contact $(-1)$-surgery on $L_n$ is independent of $n\in N$, and thus it suffices to argue that the $L_n$ in $(S^3,\xi_{st})$ are pairwise non-isotopic Legendrian knots. It is readily seen that their Thurston--Bennequin invariants and rotation numbers are $(\tb(L_n),\rot(L_n))=(1,0)$ for all $n\in\N_0$: we shall instead show that the $L_n$ are pairwise distinct as smooth knots, i.e. $L_n$ is not smoothly isotopic to $L_m$ if $n\neq m$, $n,m\in\N_0$. This can be done with different methods, we proceed as follows, in line with the strategy originally used by Osoinach in the smooth case, \cite{Osoinach06}. 

The software SnapPy~\cite{CullerDunfieldGoernerWeeks}, as part of SageMath, readily verifies that the exterior of the three component link $L=L_{A,\gamma}\cup K_1\cup K_2$ in Figure~\ref{fig:mainex} is hyperbolic. Dehn filling the blue and red knots, $K_1$ and $K_2$, with topological surgery coefficients $\frac{-n \pm1}{n}$ yields the exterior of the knot $L_n$. Since the length of these surgery slopes increases monotonically to infinity it follows that, for $n$ sufficiently large, all the resulting $L_n$ are hyperbolic knots and the volumes of these $L_n$ converge strictly monotonically to the volume of $L$~\cite{NeumannZagier85, Thurston80}. In particular, it follows that at most finitely many of the $L_n$ are smoothly isotopic. This argument can be improved by using SnapPy and computing all short slopes, verifying the hyperbolicity of the corresponding fillings, and rigorously computing their volumes. We have performed these precise calculations, together with some additional data, in a Jupyter SageMath Notebook which can be accessed at~\cite{CasalsEtnyreKegel}. The computations do indeed conclude that all the knots $L_n$ are pairwise smoothly non-isotopic, as required. 

Alternatively, we can use the main result from~\cite{BakerGordonLuecke16} to deduce that infinitely many of the $L_n$ are smoothly non-isotopic.
\end{proof}

\begin{rem}
Showing that finitely many of the $L_n$ are different is readily achieved by computing classical topological knot invariants. For instance, we can compute the HOMFLY polynomials $p(L_0)$ and $p(L_1)$ of $L_0$ and $L_1$ to be different:
$$
p(L_0)(l,m)=l^{-6}m^4 + (l^{-2} - 2l^{-4} - 3l^{-6} - l^{-8})m^2 + (-l^{-2} + 2l^{-4} + 3l^{-6} + l^{-8}),$$
$$p(L_1)(l,m)=l^{-14}m^4 + (l^{-2} - l^{-4} + l^{-6} - l^{-8} - 4l^{-14} - l^{-16})m^2 + (-l^{-2} + l^{-8} + 3l^{-14} + 2l^{-16}).$$

It appears to be reasonable that a general formula for all the HOMFLY polynomials $p(L_n)$ could be found, and distinguish them all as well. Nevertheless, note that the Alexander polynomials $A(L_n)$ of all these knots $L_n$ in $S^3$ coincide, since the Alexander polynomial of a knot is an invariant of the $0$-surgery.\hfill$\Box$
\end{rem}

Theorem~\ref{thm:infinitesurgeries} presents the first instance of distinct Legendrian knots in $(S^3,\xi_{st})$ with contactomorphic Legendrian surgeries, and even shows that infinitely many such Legendrian knots exist. Theorem \ref{thm:main}.(i) is even stronger than that, as it claims that there are instances of distinct Legendrian knots whose Stein traces are equivalent. In fact, we shall prove Theorem~\ref{thm:main}.(i) by using Theorem \ref{thm:infinitesurgeries}: we now show that the Legendrian knots $L_n$ in $(S^3,\xi_{st})$ constructed in the proof of Theorem \ref{thm:infinitesurgeries} -- obtained as contact annulus twists on Figure~\ref{fig:mainex} -- have equivalent Stein traces.

\subsection{Extending contactomorphisms over Stein traces}\label{strace}
In this section we conclude the proof of Theorem \ref{thm:main}.(i). This is achieved by showing that we can extend the contactomorphisms in Theorem \ref{thm:infinitesurgeries} to the bounding Stein traces.

First, we state a result from \cite{ConwayEtnyreTosun21}, which is also a fact well-known to experts. Specifically, see \cite[Lemma~3.3]{ConwayEtnyreTosun21} for its proof.

\begin{lem}\label{nbhdoflagrangiandisk}
	Let $(W,\omega)$ be a symplectic manifold with convex boundary $\partial W$, $D \subset (W,\omega)$ a properly embedded Lagrangian disk transverse to $\partial W$. Then, there exists a neighborhood $\Op(D)\subset(W,\omega)$ of $D$ in $W$ symplectomorphic to a neighborhood of the co-core of a Weinstein 2-handle and such that the Liouville vector field defined near the $\partial W$ agrees with the Liouville vector field on the Weinstein 2-handle. In addition, this symplectomorphism can be assumed to agree with any preassigned symplectomorphism from a neighborhood of $\partial D$ to a neighborhood of the boundary of the co-core.\hfill$\Box$
\end{lem}

Lemma~\ref{nbhdoflagrangiandisk} can now be used in order to establish a sufficient criterion for contactomorphisms between Legendrian surgeries to extend to their Stein traces.

\begin{lem}\label{extend}
Let $L$ and $L'$ in $(S^3,\xist)$ be Legendrian knots such that there is a contactomorphism $\phi:L(-1)\longrightarrow L'(-1)$ between their Legendrian surgeries. Consider a Legendrian meridian $\mu$ in $L(-1)$ to $L$ that bounds the Lagrangian co-core $D$ of the $2$-handle in the Stein trace $W_L$, and suppose that the image $\phi(\mu)$ in $L'(-1)$ bounds a Lagrangian disk $D'$ in $W_{L'}$. Then, the contactomorphism $\phi$ extends to a symplectomorphism $\Phi:W_L\longrightarrow W_{L'}$ of the Stein traces $W_L$ and $W_{L'}$, possibly after deforming the symplectic structure on one of these Stein traces.
\end{lem}

\begin{rem}\label{rem:akbulut}
Lemma~\ref{extend} is a symplectic analogue of Akbulut's lemma~\cite{Akbulut77}. Note that the topological version of this lemma requires an additional hypothesis; however, in the symplectic framework, the existence of additional geometry suffices to conclude the same result without the hypothesis being required in the statement.\hfill$\Box$
\end{rem}

\begin{proof}
First, we extend the contactomorphism $\phi:\partial W_L\to \partial W_{L'}$ to a symplectomorphism in an open neighborhood of the boundary. By Lemma~\ref{nbhdoflagrangiandisk}, we can further extend this symplectomorphism to a symplectomorphism $\phi':N\longrightarrow N'$ from an open neighborhood $N$ of $\partial W_L\cup D$ in $W_L$ to an open neighborhood $N'$ of $\partial W_{L'}\cup D'$ in $W_{L'}$. It suffices to extend to the complements $B:= W_L\setminus N$ and $B':=W_{L'}\setminus N'$. For that, we notice that -- by choosing the neighborhood $N$ appropriately -- $B$ is a symplectic manifold with convex boundary and $\partial B=S^3$. In consequence, both contact boundaries $\partial B$ and $\partial B'$ are contactomorphic to the standard contact $3$-sphere $(S^3,\xi_{st})$. 

Since $B \subset W_L$ and $B'\subset W_{L'}$ are subsets of a Stein manifold, they are minimal. Under this hypothesis, the Gromov-McDuff Theorem \cite{Gromov85, McDuff90} implies that $B$ and $B'$ are symplectomorphic to a standard symplectic $4$-ball. In particular, possibly after a deformation of one of the Stein structures, they are symplectomorphic \cite[Theorem~16.6]{CieliebakEliashberg12}. This allows us to extend the symplectomorphism $\phi':N\longrightarrow N'$ to an equivalence $\Phi:W_L\longrightarrow W_{L'}$ of the Stein traces, as required.
\end{proof}

Lemma~\ref{extend} can now be applied to Legendrian knots with a contact annulus presentation, discussed in Subsections \ref{annulustwist} and \ref{atwist} above.

\begin{thm}\label{extendunknot}
Let $(A,\gamma)$ be a contact annulus presentation in $(S^3,\xist)$, with the pre-Lagrangian annulus $A$ in the type of the maximal-tb Legendrian unknot. Let $L_1$ be the result of an $1$-fold annulus twist on $L_{A,\gamma}$. Then the Stein traces of $L_{A,\gamma}$ and $L_1$ are equivalent.
\end{thm}

\begin{proof}
Let us denote $L_0:=L_{A,\gamma}$ for simplicity, and consider the Weinstein handlebody diagram for $W_{L_{A,\gamma}}$ shown in the upper left of Figure~\ref{fig:extend}, where the co-core $\mu$ of the Weinstein $2$-handle is also shown. The remaining three diagrams in Figure~\ref{fig:extend} depict a contactomorphism $\phi:L_0(-1)\longrightarrow L_1(-1)$, as described in Remark~\ref{explicitdiffeom}.
\begin{figure}[htbp]{\small
  \begin{overpic}
  {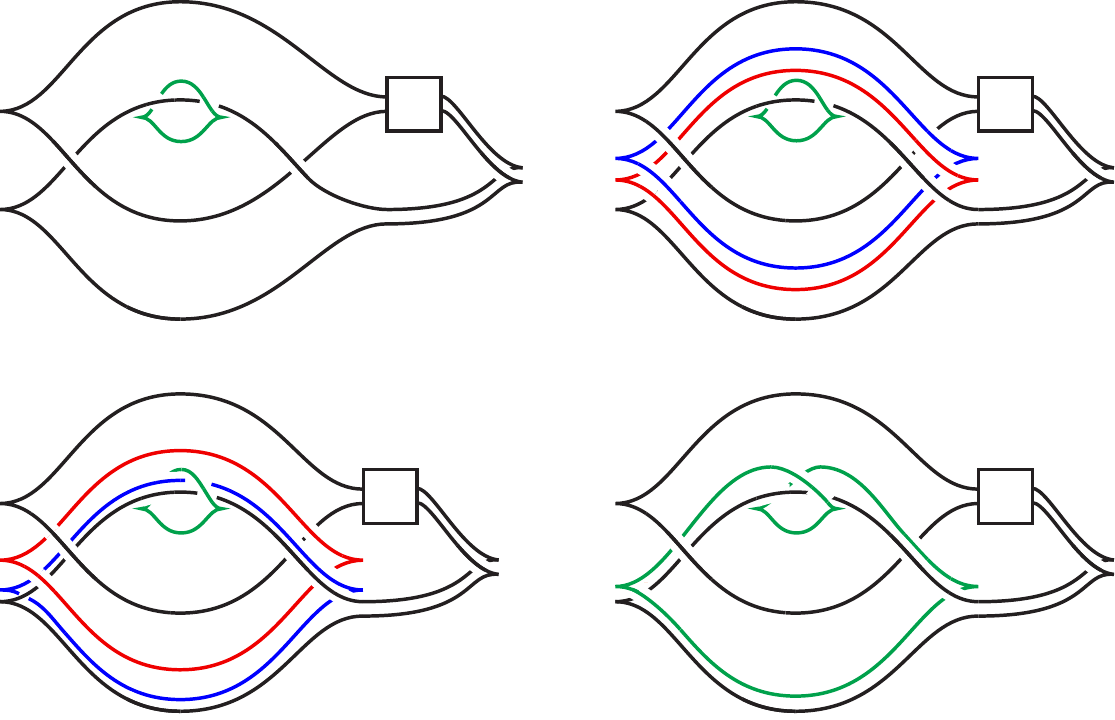}
    \put(116, 176){$\gamma$}
    \put(49, 160){\color{darkgreen}$\mu$}
     \put(90, 120){$L_{A,\gamma}$}
     \put(7, 120){$(-1)$}
     \put(288, 176){$\gamma$}
     \put(157, 161){$\color{blue} (-1)$}
     \put(165, 152){$\color{red} (1)$}
    \put(227, 160){\color{darkgreen}$\mu$}
     \put(259, 120){$L_{A,\gamma}$}
     \put(186, 120){$(-1)$}
     \put(-21, 35){$\color{blue} (-1)$}
     \put(-13, 45){$\color{red} (1)$}
     \put(110, 63){$\gamma$}
    \put(49, 46){\color{darkgreen}$\mu$}
     \put(90, 8){$L_{A,\gamma}$}
     \put(7, 8){$(-1)$}
      \put(287, 63){$\gamma'$}
    \put(223, 46){\color{darkgreen}$\phi(\mu)$}
     \put(261, 8){$L_{A,\gamma}$}
     \put(184, 8){$(-1)$}
  \end{overpic}}
  \caption{The diagrams used in the proof of Theorem \ref{extendunknot}, showing that $\phi(\mu)$ bounds a Lagrangian disk.}
  \label{fig:extend}
\end{figure}
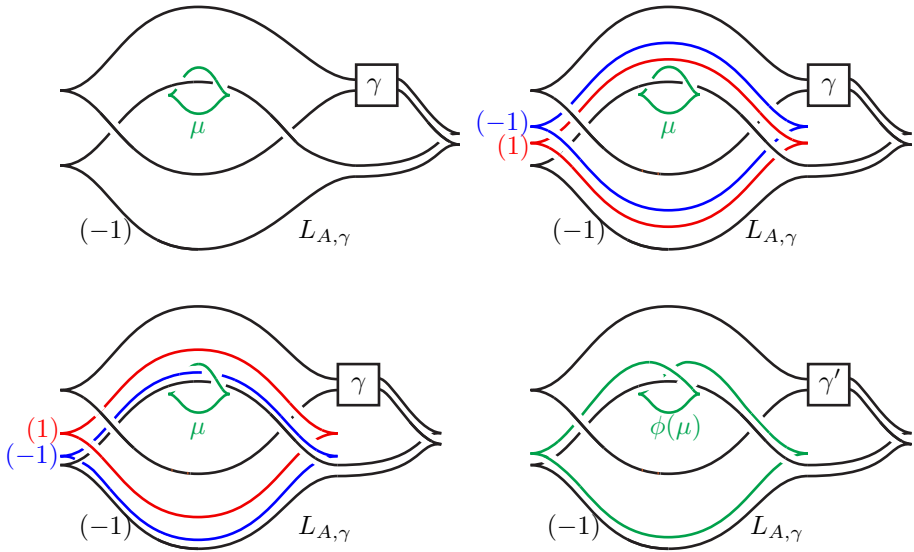
Specifically, in the upper right diagram, a canceling pair of contact $(\pm1)$-surgeries is added, which does not affect the contactomorphism type by Lemma~\ref{cancellationlemma}. The diagram depicted in the lower left shows a handle slide of the $(-1)$-framed unknot over the Legendrian knot $L_{A,\gamma}$ followed by an isotopy. Finally, in the lower right diagram, we have depicted the result of the contact annulus twist on $L_0$ and the image of the meridian $\mu$. Note that the Legendrian arc $\gamma'$ depicted in Figure~\ref{fig:extend} is obtained from the Legendrian arc $\gamma$ by a contact annulus twist. Now, the image $\phi(\mu)$ is a maximal-tb Legendrian unknot in the boundary of the standard symplectic $4$-ball, and thus it bounds an embedded Lagrangian disk. Hence, we can apply Lemma~\ref{extend} in order to extend the contactomorphism $\phi$ to an equivalence $\Phi:W_{L_0}\longrightarrow W_{L_1}$ of the Stein traces $W_{L_0}$ and $W_{L_1}$, as claimed.
\end{proof}

\begin{rem}
	Theorem~\ref{extendunknot} remains true for any pre-Lagrangian annulus $A$ in the type of a Legendrian knot $L$ in $(S^3,\xist)$ that bounds a Lagrangian disk in the standard symplectic $4$-ball. Indeed, the above proof shows that $\phi(\mu)$ is again isotopic to $L$ in the boundary of the standard symplectic $4$-ball and thus will bound again a Lagrangian disk such that Lemma~\ref{extend} applies. \hfill$\Box$
	
\end{rem}

Theorem~\ref{thm:main}.(i) now follows.

\begin{proof}[Proof of Theorem~\ref{thm:main}.(i)]
Consider the infinite family of Legendrian knots $L_n$ in $(S^3,\xi_{st})$ build in Theorem~\ref{thm:infinitesurgeries}. Each $L_n$ is the result of an $n$-fold contact annulus twist applied to the knot $L_{A,\gamma}$ in Figure~\ref{fig:mainex}, whose contact annulus presentation $(A,\gamma)$ indeed satisfies that $A$ is in the type of the standard Legendrian unknot. Theorem~\ref{thm:infinitesurgeries} states that the Legendrian $(-1)$-surgeries are contactomorphic. By construction, each $L_n$ is obtained from $L_{n-1}$ by a $1$-fold contact annulus twist. Hence, the hypothesis of Theorem \ref{extendunknot} are satisfied, and we deduce from Theorem~\ref{extendunknot} that the Stein traces of the $L_n$ must all be equivalent. 
\end{proof}

We now turn to Theorem~\ref{thm:main}.(ii).

\subsection{Contact $(*n)$ moves}\label{starmove}
This section proves one half of Theorem~\ref{thm:main}.(ii), showing that for any $m\leq 1$, $m\in \Z$, we can construct distinct Legendrian knots $L_m$ and $L'_m$ in $(S^3,\xist)$, both with $\tb(L_m)=\tb(L'_m)=m$, such that their Stein traces are equivalent, i.e. $W_{L_m}$ is Stein equivalent to $W_{L'_m}$. For that, we define the contact $(*n)$ move, adapting the topological $(*n)$ move developed in \cite[Section 3]{AbeJongLueckeOsoinach15} to the contact setting. 

In line with Subsection~\ref{strace}, let $(A,\gamma)$ be a contact annulus presentation in the standard contact $(S^3,\xist)$, with the pre-Lagrangian annulus $A$ in the type of the maximal-tb Legendrian unknot. Consider the associated Legendrian knot $L_{A,\gamma}$, depicted in the front diagram as in the upper left of Figure~\ref{fig:extend}, and let $U$ in $(S^3,\xist)$ be the Legendrian unknot depicted in green in this same front diagram. Note that $U$ is a standard Legendrian unknot but links $L_{A,\gamma}$ in a specific manner. In addition, for $n\in\N_0$, any of the contact $3$-manifolds $(U(1+1/n),\xi_k)$, $k\in\{-n+1, -n+3, \ldots, n-1\}$, following Theorem~\ref{surgerytoS3}, is endowed with a contactomorphism $\phi_k \colon(U(1+1/n),\xi_k)\longrightarrow(S^3,\xi_{st})$. In the notation above, we define the following two moves.

\begin{defi}\label{def:contactmoves}
A Legendrian knot $L'_k$ in $(S^3,\xist)$ is said to be obtained by a \textit{contact $(T_n)$ move} from $L_{A,\gamma}$ if it is the image $\phi_k(L_{A,\gamma})$ of the Legendrian knot $L_{A,\gamma}$ under the contactomorphism $\phi_k\colon(U(1+1/n),\xi_r)\longrightarrow(S^3,\xi_{st})$. 

\noindent A Legendrian knot $L''_k$ in $(S^3,\xist)$ is said to be obtained by a \textit{contact $(*n)$ move} from $L_{A,\gamma}$ if it is the result of a contact annulus twist on $L_{A,\gamma}$ followed by a contact $(T_n)$ move.\hfill$\Box$
\end{defi}

The moves in Definition \ref{def:contactmoves}, in conjunction with Theorem~\ref{extendunknot}, can now be used to conclude the first half ($\tb\leq1$) of the proof of Theorem~\ref{thm:main}.(ii), as follows. The other cases $(\tb>1)$ will be discussed in Section~\ref{sec:RGB}.

\begin{proof}[Proof of Theorem~\ref{thm:main}.(ii) $($\mbox{The case of }$\tb\leq1)$]
Let $(A,\gamma)$ be a contact annulus presentation in $(S^3,\xist)$, with the pre-Lagrangian annulus $A$ in the type of the maximal-tb Legendrian unknot and, given $m\leq 0$, $m\in\Z$, let $L_m$ in $(S^3,\xist)$ be the Legendrian knot $L_{A,\gamma}$ in $(S^3,\xist)$ stabilized $|m|$ times. For instance, we may choose $L_{A,\gamma}$ as in Figure \ref{fig:mainex}. As noted in Remark~\ref{effectstab}, $L_m$ can obtained from the Legendrian $L_{A,\gamma}$ by performing contact $(1+1/m)$-surgery on a meridian $\mu$, if one chooses the appropriate contact structure $\xi_k$ on the surgery.

Let $L'_m$ in $(S^3,\xist)$ be the result of a contact annulus twist on $L_{A,\gamma}$ followed by a contact $(T_n)$ move on the image $\mu'$ of $\mu$ under the contact annulus twist, as depicted in Figure~\ref{fig:extend}. That is, in the language of Definition \ref{def:contactmoves}, $L'_m$ is obtained from $L_{A,\gamma}$ by a contact $(*m)$ move. Lemma~\ref{lem:twist} implies that the contact $(-1)$-surgery on $L_{A,\gamma}$ and contact $(1+1/m)$-surgery on $\mu$ is contactomorphic to contact $(-1)$-surgery on the contact annulus twisted $L_{A,\gamma}$ and contact $(1+1/m)$-surgery on $\mu'$, if we use the same contact structures on the $(1+1/m)$-surgeries. Now, by Theorem~\ref{surgerytoS3}, the first contact manifold is the result of contact $(-1)$-surgery on $L_m$ in $(S^3,\xist)$, and the second contact structure is the result of contact $(-1)$-surgery on $L'_m$. Hence, these contact structures are contactomorphic. 

It now suffices to argue that this contactomorphism extends over the Stein traces. Indeed, if one takes two meridians $\mu_1$ and $\mu_2$ to $L_{A,\gamma}$ that are unlinked, then their images under the contact annulus twist are unlinked. This can be pictorially verified by following the moves in Figure~\ref{fig:extend} with both meridians. In consequence, when one performs the contact $(T_m)$ move on $\mu_1$, $\mu_2$ is still a maximal-tb standard Legendrian unknot and Lemma~\ref{extend} gives the desired result.

It remains to show that $\tb(L'_m)=1-m$ and that $L'_m$ and $L_m$ are not isotopic. For the first equation, we use Figure~\ref{fig:extend} bottom right to compute that the linking number of the two knots is $1$ and thus the formula from~\cite[Lemma 6.2]{Kegel18} implies that $\tb(L'_m)=\tb(L_{A,\gamma})-m=1-m$. 

For the second statement, we will use the example from Figure \ref{fig:mainex} and show that in that example $L'_m$ and $L_m$ are not smoothly isotopic. For that we use again SnapPy and the limit formula for the volume from~\cite{NeumannZagier85} to show that the volume of $L'_m$ is always larger than the volume of $L_m$. (Note that the $L_m$ are all smoothly isotopic to $L_0$ and thus all have the same volume.) These computations are saved at~\cite{CasalsEtnyreKegel}.
\end{proof}

\begin{figure}[htbp] 
{\small
  \begin{overpic}
  {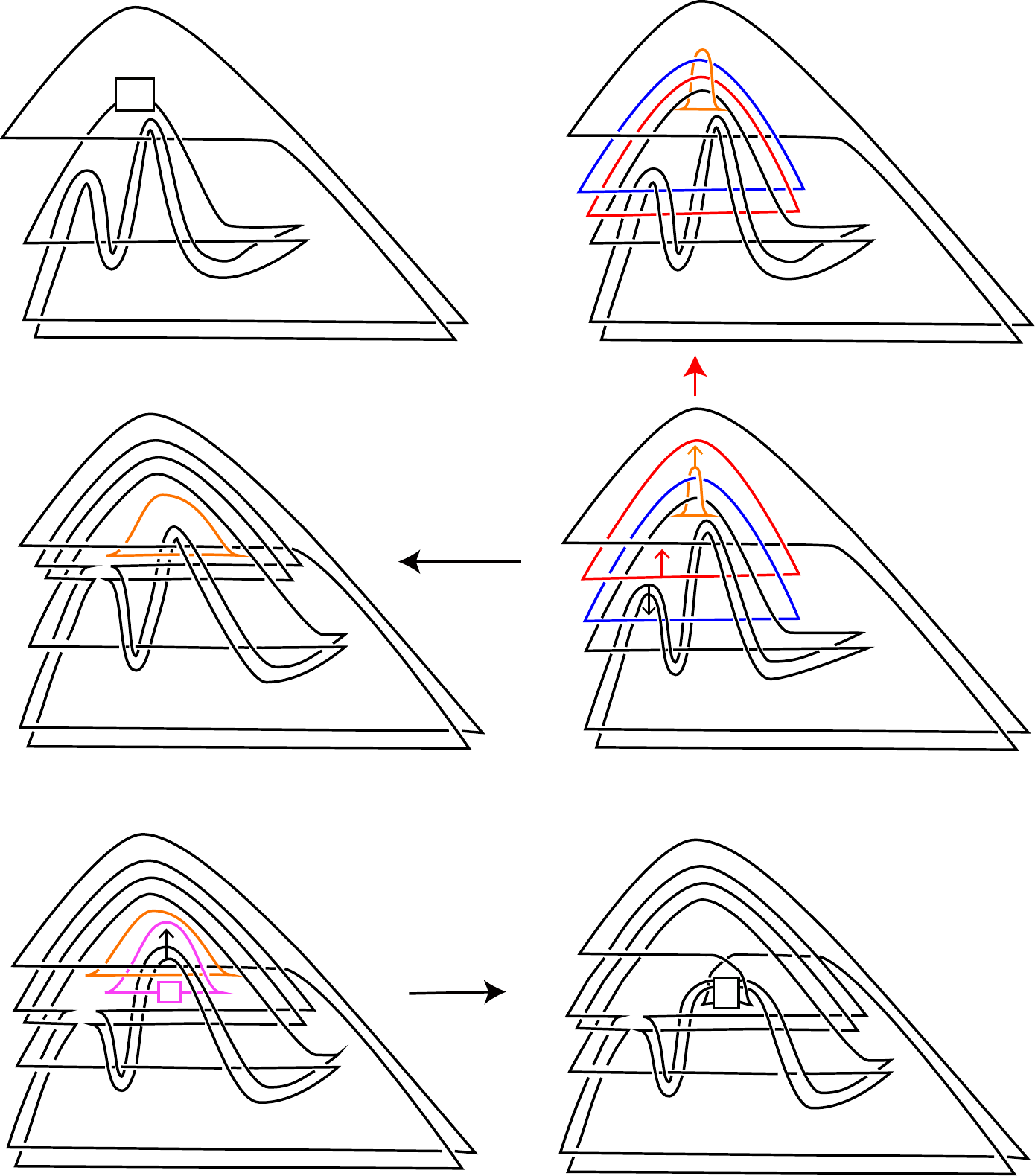}
    \put(120, 337){$(-1)$}
    \put(100, 275){\color{orange}$(1+\frac 1m)$}
     \put(350, 337){$(-1)$}
      \put(258, 440){\color{orange}$(1+\frac 1m)$}
      \put(313, 388){$\color{blue} (+1)$}
     \put(206, 369){$\color{red} (-1)$}
     \put(101, 431){$L_m$}
     \put(170, 390){$\cong$}
     \put(48, 417){$m$}
     \put(130, 178){$(-1)$}
      \put(350, 178){$(-1)$}
         \put(300, 290){\color{orange}$(1+\frac 1m)$}
           \put(205, 215){$\color{blue} (+1)$}
     \put(312, 230){$\color{red} (-1)$}
        \put(130, 15){$(-1)$}
      \put(350, 15){$(-1)$}
      \put(62.5, 69){\tiny $\color{npink} m$}
      \put(82, 67){ \tiny$\color{npink} (-1)$}
      \put(24, 70){\tiny $\color{orange} (-1)$}
       \put(278.5, 69){\tiny $m$}
        \put(330, 100){$L_m'$}
        \put(55, 153){\begin{rotate}{-90}$\cong$\end{rotate}}
  \end{overpic}}
	\caption{An example of a contact $(*n)$ move. The box stands for an $m$-fold stabilization, except on the bottom right where it stands for the Legendrian $3$-copy of the $m$-fold stabilization of an arc.}
	\label{fig:examplenmove}
\end{figure}
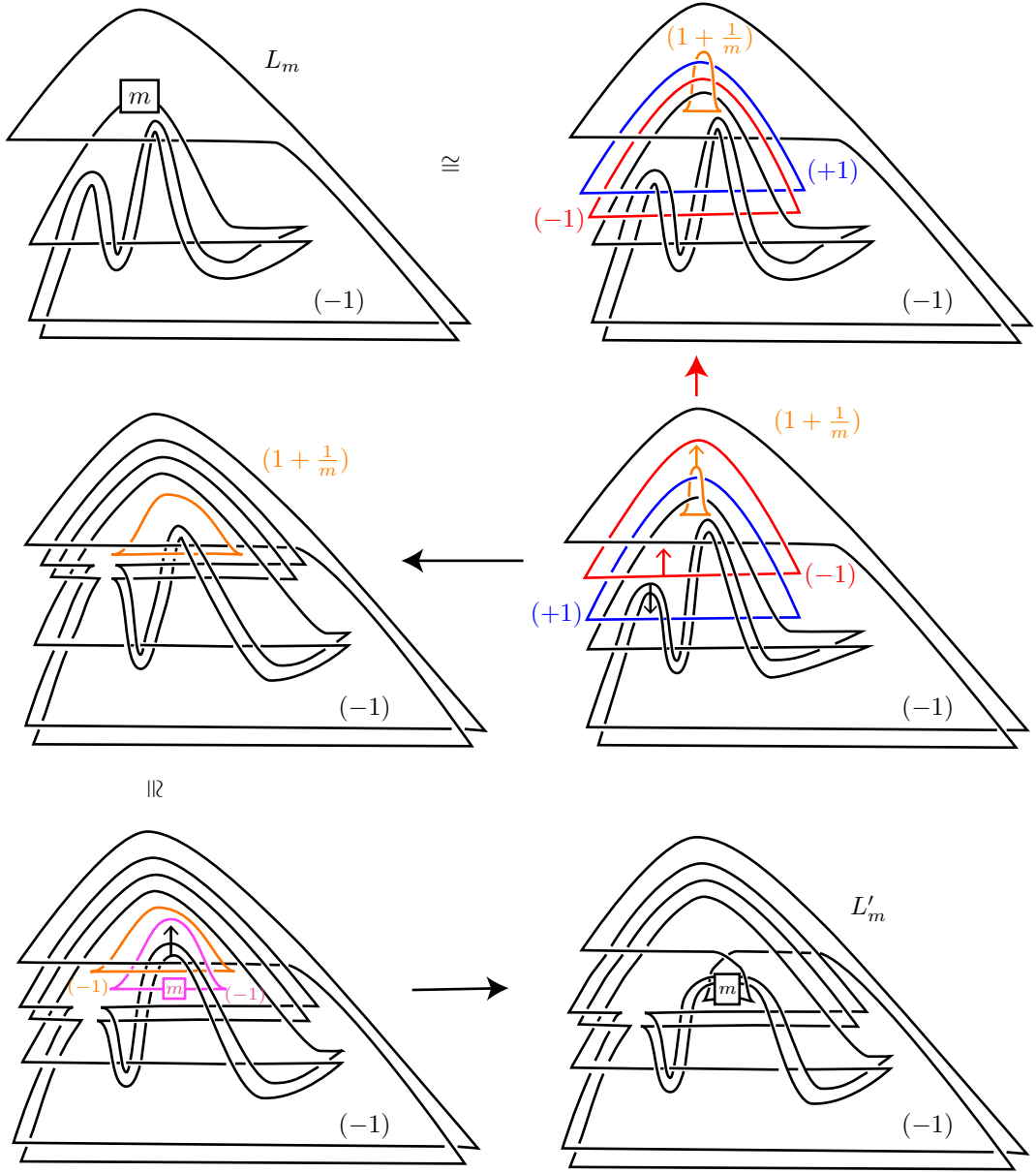

\begin{ex}
	Figure~\ref{fig:examplenmove} shows an example of a contact $(*n)$ move starting with the contact annulus presentation from Figure~\ref{fig:mainex}. The sequence of moves is following the arguments from the above proof and yields an explicit contactomorphism from $L_m(-1)$ to $L'_m(-1)$. Also note it is straightforward to compute in the  front projections of $L_m$ and $L'_m$ that both knots have $\tb=1-m$. 
	
	Finally, we remark that computer experiments suggest that $L'_m$ is always a stabilized knot. However, the destabilization is not straightforward to see. For example for $m=1$ we have transformed the front projection of $L'_1$ to the corresponding grid diagram. Via gridlink we could find a sequence of moves to a grid that corresponds to a destabilizable front projection. The data can be found at~\cite{CasalsEtnyreKegel}. \hfill$\Box$
\end{ex}

\subsection{Additional constructions}\label{sec:other}

In this section we present a few more ways to construct examples of non-isotopic Legendrian knots with the same Stein trace or the same Legendrian surgeries. In the smooth context, both dualizable patterns and RGB-diagrams generalize (special) annulus twists, see e.g.~\cite{MillerPiccirillo18,Tagami2,Tagami1}, especially~\cite[Section 6]{MillerPiccirillo18} and~\cite[Theorem 4.8]{Tagami1}. Let us present the contact and symplectic analogues of dualizable patterns and RGB-diagrams, with the connection between them and contact annulus twist left for future work.

\subsubsection{Dualizable patterns}
In \cite{Brakes80}, Brakes introduced dualizable patterns for smooth knots as a method to create pairs of knots with diffeomorphic zero surgeries. In~\cite{MillerPiccirillo18}, smooth dualizable patters are extensively studied, and a criterion for a pattern to be dualizable is given. In addition, the relation to the smooth annulus twist is discussed in~\cite{MillerPiccirillo18}, and see also~\cite[Section 3.2]{Tagami2}. In our context, Legendrian dualizable patterns can be defined as follows.

Let $V=(S^1\times D^2, \xi)$ be the solid torus with the unique tight contact structure having convex boundary with two longitudinal dividing curves. Note that any Legendrian knot has a neighborhood contactomorphic to $V$. By definition, a Legendrian knot $P$ in $V$ is called a {\it Legendrian pattern}. Given a Legendrian knot $L$ in $(S^3,\xist)$, we denote by $P(L)$ in $(S^3,\xist)$ the (satellite) Legendrian knot obtained by first embedding $P$ into $V$ and then $V$ into $(S^3,\xist)$ as the standard neighborhood of $L$. We denote a closed standard neighborhood of $P$ in $V$ by $\nu P$. 

\begin{defi}\label{def:LegDualizablePatterns}
Let $P$ in $V$ be a Legendrian pattern. A Legendrian pattern $P^*$ in $V^*$ is called the {\it dual pattern} of $P$ if there exists a contactomorphism $\phi:V\setminus\mathring{\nu P}\longrightarrow V^*\setminus\mathring{\nu P^*}$ between the contact complement $V\setminus\mathring{\nu P}$ and the contact complement $V^*\setminus\mathring{\nu P^*}$ such that $\phi$ restricted to the respective boundaries acts as:
\begin{align*}
\lambda_V&\longmapsto\lambda_{P^*} & & \lambda_P\longmapsto \lambda_{V^*}\\
\mu_V&\longmapsto -\mu_{P^*}& &\mu_P\longmapsto -\mu_{V^*},
\end{align*}
where $\mu_P$ denotes the meridian to $P$, $\lambda_P$ is the longitude for $\nu P$ determined by the contact framing, $\mu_V$ is the boundary of $\{pt\}\times D^2$ in $V$ and $\lambda_V$ is a curve parallel to the dividing set of $\partial V$. (Analogous definitions for $P^*$ in $V^*$.) We call $P$ a \textit{dualizable pattern} if there exists a dual pattern $P^*$ for $P$. \hfill$\Box$ 
\end{defi}

The Legendrian dualizable patterns in Definition~\ref{def:LegDualizablePatterns} can be used to construct Legendrian knots with the same Stein traces. The necessary result is the following proposition:

\begin{prop} \label{prop:dualizablepatterns}
	Let $U$ in $(S^3,\xist)$ be the Legendrian unknot with $\tb=-1$ and $P$ a dualizable pattern with dual pattern $P^*$. Then, the two Legendrian knots $P(U)$ and $P^*(U)$ in $(S^3,\xist)$ have the equivalent Stein traces.
\end{prop}

\begin{proof}
Let us choose a framing on the solid tori $V$ and $V^*$ such that $\lambda_V$ and $\lambda_{V^*}$ have slope $-1$. Note that $(S^3,\xist)$ is obtained from $V$ by Dehn filling along the curve $\lambda_V+\mu_V$, i.e. along a curve of slope 0, and then extending the contact structure over the filling torus so that it is tight (and there is a unique way to do this). Thus, contact Dehn surgery on $P(U)(-1)$ is simply the result of Dehn filling the contact complement $V \setminus\mathring{\nu P}$ with slopes  $\lambda_P-\mu_P$ and $\lambda_V+\mu_V$ of $V$. Since $P$ is dualizable, the last manifold is contactomorphic to the Dehn filling of $V^*\setminus\mathring{\nu P^*}$ with slopes $\mu_{V^*}+\lambda_{V^*}$ and $\lambda_{P^*}-\mu_{P^*}$ which is again the same as Legendrian surgery on $P^*(U)$. Now, the contactomorphism can be chosen to take the meridian of $P(U)$ to $\mu_{V^*}$ in $P^*(U)(-1)$ which bounds a Lagrangian disk in the associated Stein trace. Thus Lemma~\ref{extend} implies that such a contactomorphism must extend to a symplectomorphism of the Stein traces, possibly after a deformation. The conclusion follows.
\end{proof}

The construction of Legendrian dualizable patterns in the contact setting can be done similar to the smooth construction developed in~\cite{MillerPiccirillo18}. For that, we denote by $\Gamma\colon V\rightarrow (S^1\times S^2,\xist)$ the embedding of a solid torus $V$ as the tubular neighborhood of the Legendrian curve $\hat \lambda_V$, shown in Figure~\ref{fig:1-handle}, and by $\widehat P$ the Legendrian knot in $(S^1\times S^2,\xist)$ obtained by embedding $P$ into $V$, and then $V$ into $(S^1\times S^2,\xist)$ via $\Gamma$. 

\begin{lem}\label{lem:dualizable}
Suppose that the Legendrian knot $\widehat P$ in $(S^1\times S^2,\xist)$ is Legendrian isotopic to $\hat\lambda_V$ in $(S^1\times S^2,\xist)$. Then, $P$ is dualizable with dual pattern $P^*=\hat\lambda_V$, seen as a pattern in the contact complement $V^*=(S^1\times S^2,\xist)\setminus \Gamma(V)$.
\end{lem}

\begin{proof}
Let us choose a framing on the solid torus $V$ such that the Legendrian curve $\lambda_V$ has slope $0$. Then, the contact 3-manifold $(S^1\times S^2,\xist)$ is obtained from $V$ by first Dehn filling the solid torus $V$ with slope $\infty$, and then extending the contact structure to the surgery torus so that it is tight. Notice that $V\subset S^1\times S^2$ is a neighborhood of $\hat\lambda_V$ and $(S^1\times S^2)\setminus V$ is contactomorphic to $V$. Now, if a Legendrian knot $\widehat P$ is isotopic to $\hat\lambda_V$, then $V^*=(S^1\times S^2)\setminus  \nu \hat P$ is also contactomorphic to $V$. We set $P^*=\hat \lambda_V$ in the complement solid torus $V^*$. In this case, $(S^1\times S^2)\setminus \nu(\widehat P\cup P^*)$ is contactomorphic to both $V\setminus \nu(P)$ and $V^*\setminus \nu(P^*)$ and it is readily verified that the contactomorphism from $V\setminus \nu(P)$ to $V^*\setminus \nu(P^*)$ exchanges the meridians and longitudes as required.
\end{proof}

\begin{center}
	\begin{figure}[h!]
	{\small
  \begin{overpic}
  {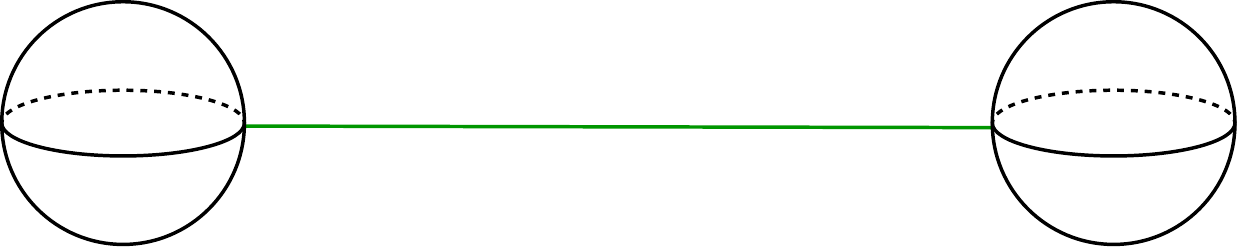}
    \put(165, 45){\color{darkgreen}$\hat\lambda_V$}
  \end{overpic}}
		\caption{The Legendrian knot $\hat\lambda_V$ in $(S^1\times S^2,\xist)$. The two $3$-balls represent a single Weinstein $1$-handle and thus the diagram displays the unique Stein fillable contact structure on $S^1\times S^2$.}
		\label{fig:1-handle}
	\end{figure}	
\end{center}

Lemma~\ref{lem:dualizable} is the contact analogue of~\cite[Proposition 3.5]{MillerPiccirillo18}. The analogue of~\cite[Proposition 6.3]{MillerPiccirillo18} should also hold in our context, i.e. given a Legendrian knot $L$ with an annulus presentation $(A,\gamma)$ where $A$ is a pre-Lagrangian annulus obtained form the maximal Thurston-Bennequin invariant unknot and $L'$ obtained from $L$ by a $1$-fold annulus twist, there is a dualizable pattern realizing $L$ with dual giving $L'$. This can be done by finding an appropriate Legendrian unknot associated to the annulus $A$ and $L$ in its complement will be the pattern. A detailed discussion of this relation and applications are left for future work. 

\begin{ex}
	We present a simple example, which is inspired by the topological example from~\cite{MillerPiccirillo18}. We consider the Legendrian pattern $P$ and its Legendrian satellite $P(U)$ depicted in Figure~\ref{fig:dualizablepattern}. Figure~\ref{fig:dualpattern} demonstrates that $\widehat P$ is Legendrian isotopic to $\hat\lambda_V$ in $(S^1\times S^2,\xist)$ and it follows by using Lemma~\ref{lem:dualizable} that $P$ is dualizable. Thus Proposition~\ref{prop:dualizablepatterns} implies that $P(U)$ and its dual Legendrian satellite $P^*(U)$ have equivalent Stein traces. While we have not constructed an explicit front of the Legendrian knot $P^*(U)$, it is not hard to see that $P(U)$ and $P^*(U)$ are not smoothly isotopic. (For a smooth knot diagram of $P^*(U)$ see Figure~\ref{fig:dual_patern}.)
	\hfill$\Box$
	\end{ex}

\begin{figure}[htbp] 
{\small
  \begin{overpic}
  {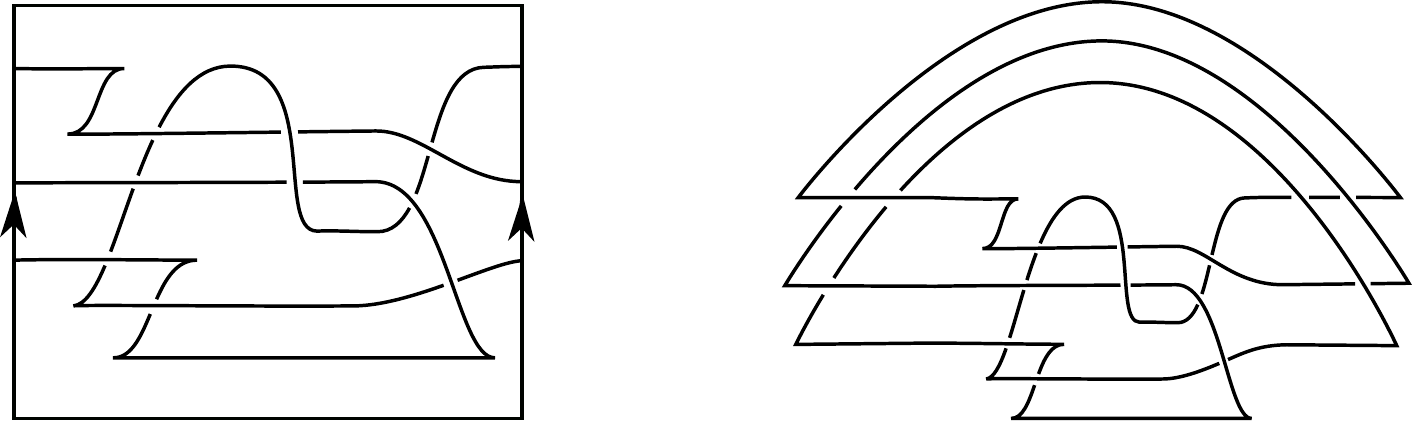}
    \put(100, 100){$P$}
    \put(310, 75){$P(U)$}
  \end{overpic}}
	\caption{A Legendrian dualizable pattern $P$ and its Legendrian satellite $P(U)$.}
	\label{fig:dualizablepattern}
\end{figure}

\begin{figure}[htbp] 
{\small
  \begin{overpic}
  {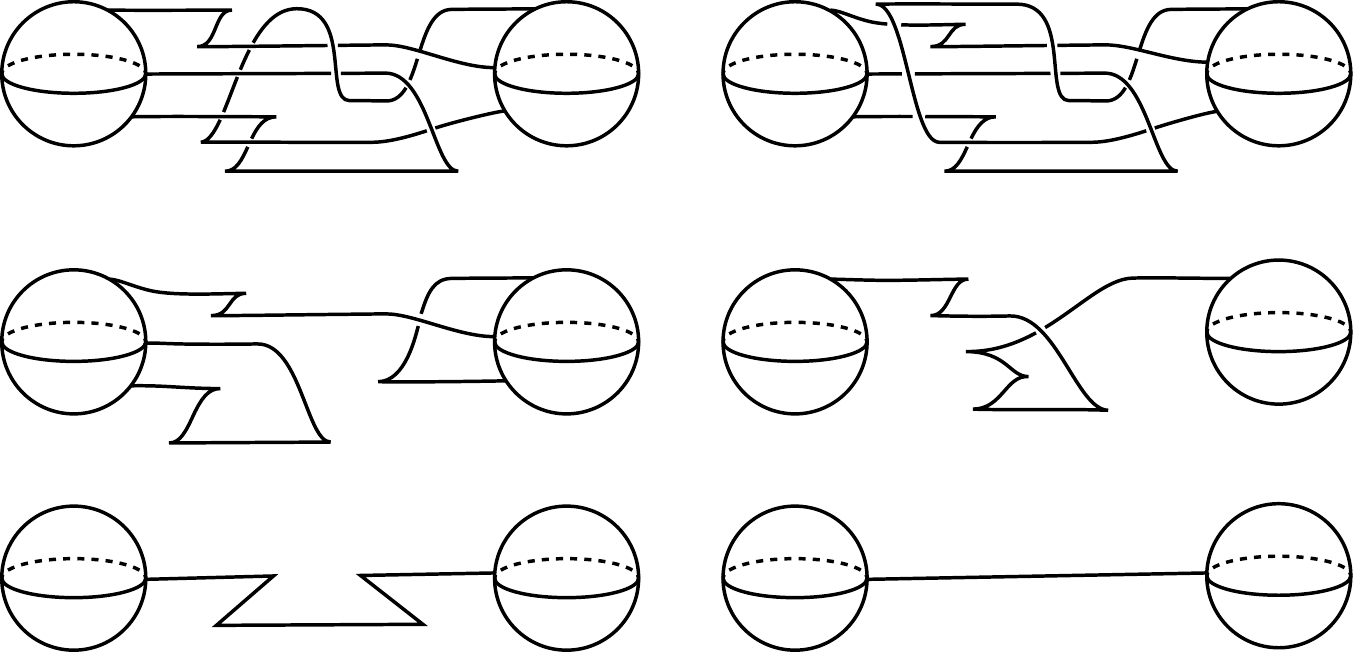}
  \end{overpic}}
	\caption{A sequence of Reidemeister moves shows that $\widehat P$ is Legendrian isotopic to $\hat\lambda_V$ in $(S^1\times S^2,\xist)$. Here the sequence goes from left to right and top to bottom.}
	\label{fig:dualpattern}
\end{figure}

\begin{figure}[htbp] 
	{\small
		\begin{overpic}
			{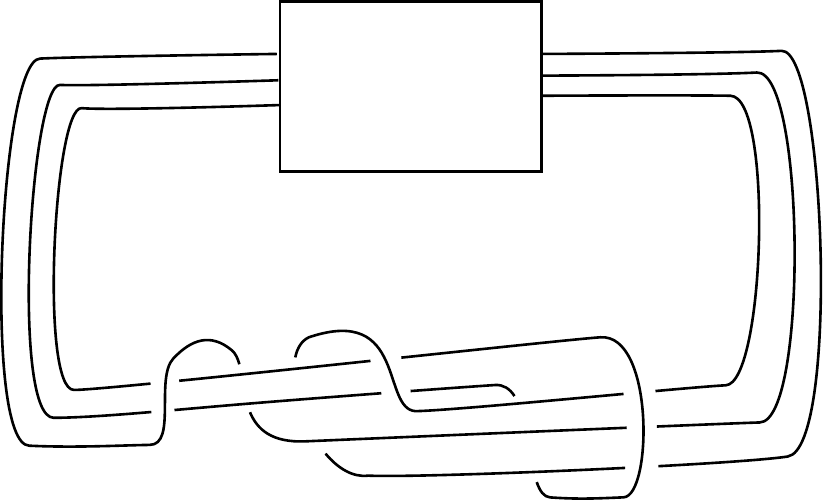}
			\put(110, 115){$-2$}
			\put(185, 75){$P^*(U)$}
	\end{overpic}}
	\caption{The smooth isotopy type of $P^*(U)$. Here the box labeled $-2$ represent two left-handed full twists.}
	\label{fig:dual_patern}
\end{figure}

\subsubsection{RGB links \& Cancelling 1-and 2-handles} \label{sec:RGB}
In~\cite{Piccirillo19}, the use of RGB links was introduced to construct pairs of knots with the same surgeries, see also~\cite[Section 4]{Tagami1}. This readily generalizes to the setting of contact surgeries, and here we record the definitions and basic results. As an application we will also construct pairs of Legendrian knots with any possible Thurston--Bennequin invariant that share the same Stein trace and thus finish the proof of Theorem~\ref{thm:main}.(ii).

\begin{defi}\label{def:RGB}
A 3-component Legendrian link $R\cup G\cup B$ in $(S^3,\xist)$ is said to be a {\it Legendrian RGB link} if the following properties hold:
\begin{enumerate} 
	\item $R\cup G$ is Legendrian isotopic to $\mu_G\cup G$, where $\mu_G$ is a meridian to $G$ with $\tb=-1$,
	\item $R\cup B$ is Legendrian isotopic to $\mu_B\cup B$, where $\mu_B$ is a meridian to $B$ with $\tb=-1$.
\end{enumerate}
By definition, $K_G$ in $(S^3,\xist)$ is the Legendrian knot obtained from $G$ in $(S^3,\xist)$ by performing a contact $(+1)$-surgery on $R$ and a contact $(-1)$-surgery on $B$. Similarly, $K_B$ in $(S^3,\xist)$ is the Legendrian knot obtained from $B$ in $(S^3,\xist)$ by performing a contact $(+1)$-surgery on $R$ and a contact $(-1)$-surgery on $G$. \hfill$\Box$
\end{defi}

Note that $R$ in $(S^3,\xist)$ must be a maximal-tb unknot, and thus the contact $(+1)$-surgery on $R$ will cancel the contact $(-1)$-surgery on either $G$ or $B$. Indeed, $R$ is isotopic to a Legendrian push-off of $G$ or $B$ after one has surgered $G$ or $B$ as we can see by performing a single handle slide. Hence, performing the given surgeries on $R$ and $B$ (or $R$ and $G$) yields the standard contact $3$-sphere $(S^3,\xist)$. These particular configurations of knots, forming a link as in Definition~\ref{def:RGB}, are useful thanks to the following result.

\begin{thm}\label{thm:easySteintraces}
Let $R\cup G\cup B$ in $(S^3,\xist)$ be a Legendrian RGB link. Then the Stein traces of $K_G$ and $K_B$ are Stein equivalent.
\end{thm}

\begin{proof}
The contact 3-manifold $R(+1)\cup G(-1)\cup B(-1)$ is contactomorphic to both $K_G(-1)$ and $K_B(-1)$, and thus the latter two are contactomorphic. The contactomorphism from $R(+1)\cup G(-1)\cup B(-1)$ to $K_G(-1)$ sends a meridian of $G$ to a meridian of $K_G$, while the contactomorphism from $R(+1)\cup G(-1)\cup B(-1)$ to $K_B(-1)$ sends a meridian of $G$ to a maximal Thurston-Bennequin unknot linking $K_B$ several times.  In consequence, the contactomorphism from $K_G(-1)$ to $K_B(-1)$ sends the meridian of $K_G$ to a Legendrian knot bounding a Lagrangian disk in the Stein trace. Hence, Lemma~\ref{extend} applies and we conclude that the contactomorphism extends to a symplectomorphism of the Stein traces of $K_G$ and $K_B$, up to deformation, as required.
\end{proof}

It is possible to use RGB links to construct pairs of Legendrian knots with arbitrary classical invariants that share the same Stein traces. In the following we present an explicit example that proves the remaining case of Theorem~\ref{thm:main}.(ii).

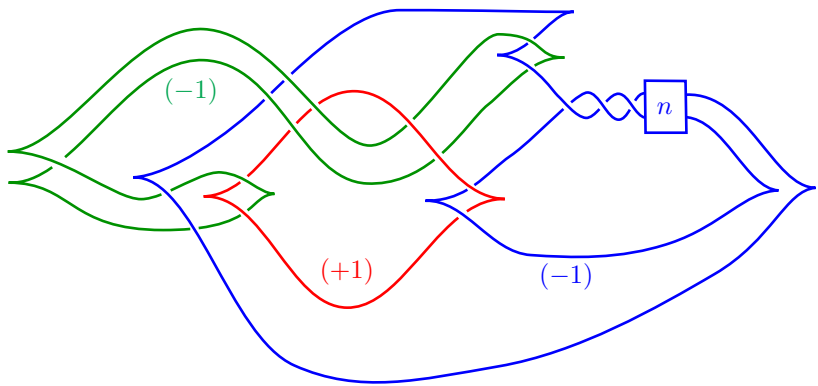
\begin{figure}[htbp] 
{\small
  \begin{overpic}
  {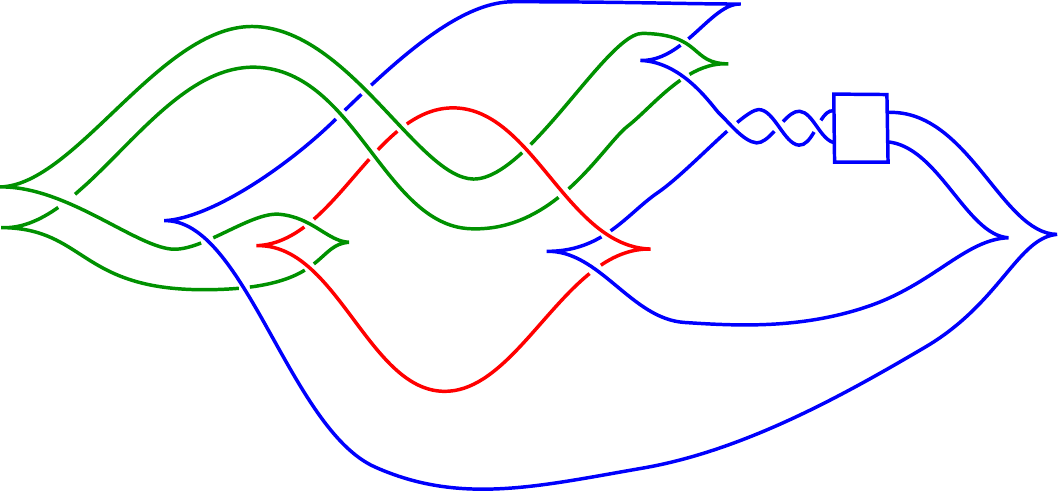}
    \put(118, 40){\color{red}$(+1)$}
    \put(59, 108){\color{darkgreen}$(-1)$}
    \put(201, 38){\color{blue}$(-1)$}
    \put(245.5, 102){\color{blue}$n$}
  \end{overpic}}
	\caption{An infinite family of Legendrian RGB links. The box labeled $n$ means here $n$ extra full positive twists.}
	\label{fig:RGB}
\end{figure}

\begin{proof}[Proof of Theorem~\ref{thm:main}.(ii) $($The case of $\tb>1)$]
We start with the family of RGB links, indexed by $n\in\N_0$, shown in Figure~\ref{fig:RGB}. By sliding the three green arcs that link $R$ over $B$ and then canceling $R$ and $G$ we are left with the Legendrian knot $K_B$ in $(S^3,\xist)$. And conversely, we can slide $B$ three times over $G$ and then cancel $R$ and $B$ to get the Legendrian knot $K_G$ in $(S^3,\xist)$. This yields an explicit contactomorphism between $K_G(-1)$ and $K_B(-1)$ which by Theorem~\ref{thm:easySteintraces} also extends to an equivalence of the Stein traces. Front projections of $K_G$ and $K_B$ are shown in the top of Figure~\ref{fig:RGB2} from which we compute $\tb(K_G)=\tb(K_B)=2n-6$. Thus, we get any positive even $\tb$. In order to obtain positive odd $\tb$, we stabilize once $B$ in Figure~\ref{fig:RGB} and then proceed analogously.

	\begin{figure}[htbp] 
	{\small
  \begin{overpic}
  {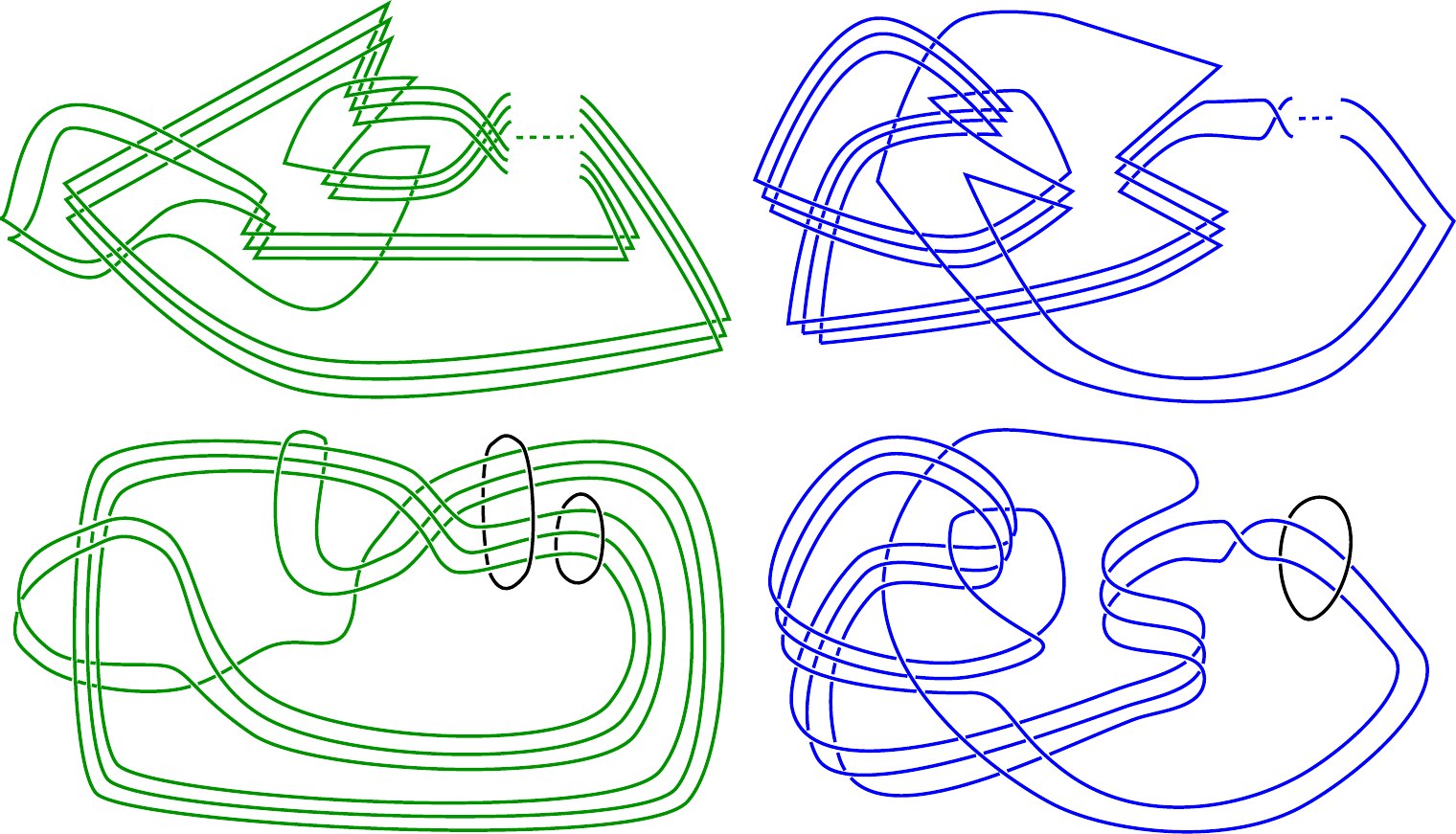}
    \put(13, 157){\color{darkgreen}$(-1)$}
    \put(130, 150){\color{darkgreen}$K_G$}
    \put(360, 145){\color{blue}$K_B$}
    \put(415, 138){\color{blue}$(-1)$}
    \put(45, 25){\color{darkgreen}$-3$}
    \put(140, 60){$-\frac 1n$}
    \put(160, 60){$- \frac 1{2n+4}$}
     \put(335, 15){\color{blue}$-2n-7$}
     \put(385, 50){$-\frac 1n$}
  \end{overpic}}
		\caption{Canceling $R$ with $G$ or $B$ in Figure~\ref{fig:RGB} yields the Legendrian knots $K_G$ and $K_B$ shown in the top row. The bottom row shows smooth surgery descriptions of these knots.}
		\label{fig:RGB2}
	\end{figure}
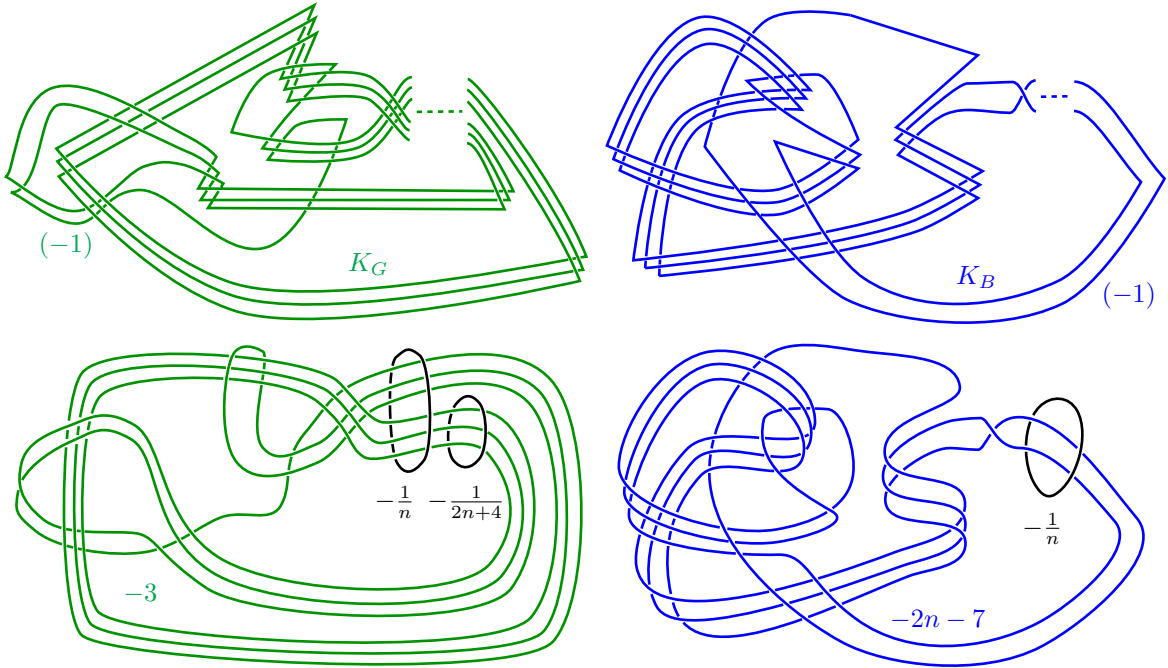
		
It remains to show that the knots $K_G$ and $K_B$ are not isotopic. In fact, we will argue that they are not even smoothly isotopic. For a particular value of $n$, it is again a simple task by computing some knot invariants. For the whole family, we consider the smooth surgery diagrams of $K_B$ and $K_G$ shown in the bottom row of Figure~\ref{fig:RGB2} and use again the limit formula from~\cite{NeumannZagier85} for the volume together with SnapPy to distinguish all pairs. The computations can be found at~\cite{CasalsEtnyreKegel}.
\end{proof}

\begin{rem}
	Note that by changing the right part of $B$ in Figure~\ref{fig:RGB} it is also straightforward to construct examples of pairs of Legendrian knots with any given values of $(\tb,\rot)$ that share the same Stein trace. We also remark that from the above proof it follows that the underlying smooth knot type of the $K_B$ has any integer slope non-characterizing, a phenomena previously discussed in~\cite{BakerMotegi18}.\hfill$\Box$
\end{rem}

In line with \cite{Tagami1}, and \cite[Theorem 4.8]{Tagami1}, there might be an equivalence between Legendrian dualizable patterns and RGB links, both generalizing contact annulus twist (when the annulus is obtained form the maximal Thurston-Bennequin invariant unknot); we hope to explore this in the future. Finally, we emphasize that using contact Kirby moves might also be useful for the endeavor of building distinct Legendrian knots with contactomorphic Legendrian surgeries. Along the same lines as in Theorem~1.2 of~\cite{ManolescuPiccirillo21}, we get:

\begin{thm}\label{thm:easy0surgery}
Let $L(+1)\cup K_1(-1)\cup K_2(-1)$ in $(S^3,\xist)$ be a $3$-component Legendrian surgery link such that $L(+1)\cup K_i(-1)$ is contactomorphic to $(S^3,\xist)$, $i=1,2$. Consider the Legendrian knot $ K_1'$ in $(S^3,\xist)$ given as the image of $K_1$ under the surgery along $L\cup K_2$, and similarly $K_2'$ in $(S^3,\xist)$. Then, $K_1'$ and $K_2'$ have contactomorphic Legendrian surgeries.
\end{thm}

\begin{proof}
Let $(M,\xi)$ be the contact 3-manifold represented by the surgery diagram $L(+1)\cup K_1(-1)\cup K_2(-1)$. By construction $(M,\xi)$ is contactomorphic to the Legendrian surgery along $K_1'$ and to the Legendrian surgery along $K_2'$.
\end{proof}

Note that, a priori, Theorem~\ref{thm:easy0surgery} works in more general situations than Theorem~\ref{thm:easySteintraces}, since $L$ might not be a standard Legendrian unknot. For instance, we can choose the surgery link such that $K_i$ is a push-off of $L$. On the other hand, the Stein traces of Legendrian knots produced via Theorem~\ref{thm:easy0surgery} might not be equivalent, in general, since even the smooth knot traces might not always be diffeomorphic. Here we work out an explicit example and thus prove Theorem~\ref{thm:notExtends}.

\begin{proof}[Proof of Theorem~\ref{thm:notExtends}]
	We consider Figure~\ref{fig:notExtend}. On the left column we see two isotopic copies of the same $3$-component surgery link together with an orange meridian of the green knot. We observe that the green and the blue knots are push-offs of the red knot. On the right column of Figure~\ref{fig:notExtend} we see the image of the green knot $L$ and the image of the blue knot $L'$ under the annulus twist of the two other knots. By construction (or by Theorem~\ref{thm:easy0surgery}) $L$ and $L'$ have contactomorphic Legendrian surgeries where an explicit contactomorphism $\phi$ is given by the handle slides and isotopies in Figure~\ref{fig:notExtend}
	
	Via SnapPy we can compute that both knots, $L$ and $L'$, are hyperbolic with different volumes and thus are smoothly non-isotopic. 
	
	To show that the Stein traces are not equivalent, we first verify that $L(-1)$ and $L'(-1)$ have vanishing hyperbolic symmetry group and thus any smooth diffeomorphism from $L(-1)$ to $L'(-1)$ is isotopic to $\phi$. Thus, it is enough to show that $\phi$ does not extend to a homeomorphism of the knot traces. For that we follow the strategy from Theorem~3.7 of~\cite{ManolescuPiccirillo21}. If $\phi$ extends to a homeomorphism $\Phi\colon W_L\rightarrow W_{L'}$ then $\Phi$ induces a homeomorphism $$DW_L:=W_L\cup -W_L\longrightarrow W_{L'}\cup_{\phi} -W_L=:X.$$
	Since $\tb(L)=1$, we conclude that the double $DW_L$ of $W_L$ has even intersection form and thus $X$ has even intersection form if $\phi$ extends to a homeomorphism of the Stein traces. 
	
	To conclude the theorem we will show that $X$ has odd intersection form. For that we will describe a Kirby diagram of $X$ from which it is straightforward to compute the intersection form. We recall that doubling the knot trace of a knot $L$ corresponds to adding a $0$-framed meridian $\mu_L$ of $L$ to the Kirby diagram~\cite{GompfStipsicz99}. Thus we get a Kirby diagram of $X$ by attaching a $2$-handle along $\phi(\mu_L)$ to the Stein trace of $L'$. In Figure~\ref{fig:notExtend} we have depicted a meridian $\mu_L$ of $L$ and its image $\phi(\mu_L)$ in orange. Since the topological $0$-framing of $\mu_L$ corresponds to a contact $(+1)$-framing, we also need to add a $2$-handle with contact framing $(+1)$ to $\phi(\mu_L)$ in the bottom right of Figure~\ref{fig:notExtend} to get a Kirby diagram of $X$.
	
	But then the intersection form of $X$ can be represented in the basis given by the $2$-handles by the odd matrix
	\begin{align*}
	\begin{pmatrix}
	\tb\big(\phi(\mu_L)\big)+1 & \lk\big(L',\phi(\mu_L)\big) \\
	\lk\big(L',\phi(\mu_L)\big) & \tb(L')-1 
	\end{pmatrix}=\begin{pmatrix}
	1 & 1 \\
	1 & 0 
	\end{pmatrix}.
	\end{align*}
\end{proof}

		\begin{figure}[htbp] 
	{\small
  \begin{overpic}[scale=1.44,,tics=10] 
  {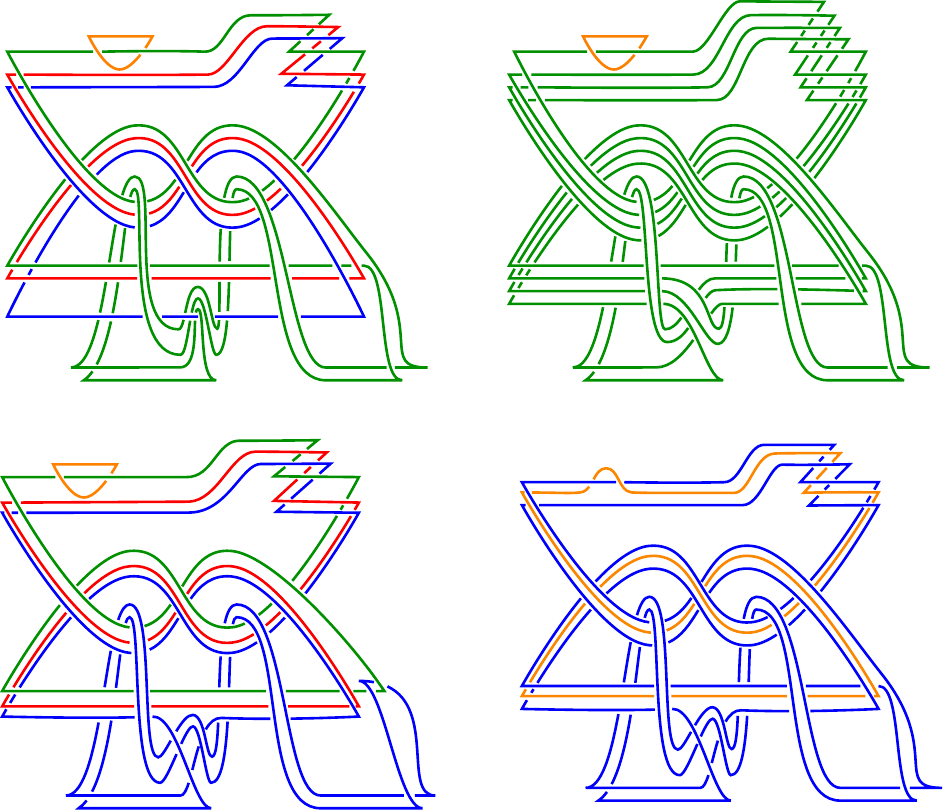}
    \put(155, 304){\color{red}$(+1)$}
    \put(-2, 270){\color{blue}$(-1)$}
    \put(153, 243){\color{darkgreen}$(-1)$}
    \put(358, 270){\color{darkgreen}$(-1)$}
    \put(154, 124){\color{red}$(+1)$}
    \put(-2, 88){\color{blue}$(-1)$}
    \put(145, 80){\color{darkgreen}$(-1)$}
    \put(355, 92){\color{blue}$(-1)$}
    \put(71, 165){\begin{rotate}{-90}$\cong$\end{rotate}}
  \end{overpic}}
	\caption{The blue and green Legendrian knots on the right are non-isotopic Legendrian knots with contactomorphic Legendrian surgeries, but non-equivalent Stein traces.}
	\label{fig:notExtend}
\end{figure}
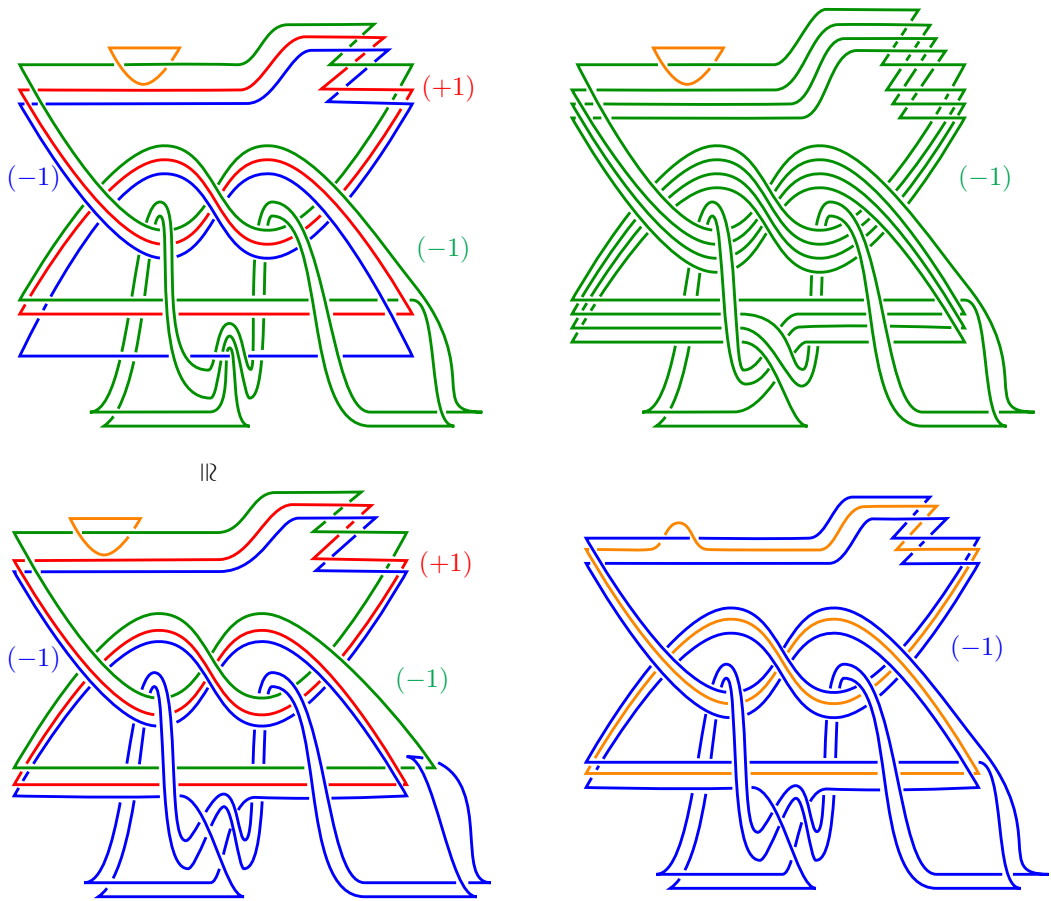

This concludes our discussion on distinct Legendrian links in the standard contact $3$-sphere with equivalent Stein traces, i.e. cases in which the $(-1)$-slope is not characterizing, and not even Stein characterizing. The manuscript now proceeds with results on characterizing slopes.


\section{Characterizing slopes}\label{sec:char} In this section we present results on characterizing slopes, in particular proving Theorems~\ref{Steinchar}, \ref{minus1char}, \ref{1characterize}, \ref{uknotchar} and \ref{thm:reciprokal}. In Subsection \ref{ssec:characterizing1}, we first consider characterizing $(\pm1)$-slopes and in the following Subsection \ref{ssec:characterizing_general} we consider general characterizing slopes.

\subsection{Contact $(\pm 1)$ characterizing slopes}\label{ssec:characterizing1} Let us prove Theorem~\ref{Steinchar} that says any Legendrian realization of the unknot, right or left handled trefoil, or figure eight knot is characterized by its Stein trace.

\begin{proof}[Proof of Theorem~\ref{Steinchar}]
	Let $L$ in $(S^3,\xi_{st})$ be a Legendrian realization of one of the knots listed in the statement of the theorem, with Thurston--Bennequin invariant $\tb(L)=t$, and $L'$ in $(S^3,\xi_{st})$ be another Legendrian knot such that its Stein traces $W_L$ and $W_{L'}$ are equivalent. Then the contact boundary $\partial W_L$ is diffeomorphic to $\partial W_{L'}$. Since the homology of $\partial W_L$ is isomorphic to $\Z_{|t-1|}$, $t'=\tb(L')$ must be $t$ or $2-t$. Since $W_L$ and $W_{L'}$ are orientation preserving diffeomorphic, their intersection forms coincide, and thus $t'$ must be equal to $t$. Now, the smooth knot type of $L$ is determined by smooth surgeries on it~\cite{KronheimerMrowkaOzsvathSzabo07, OzsvathSzabo19}, and thus $L'$ must be smoothly isotopic to $L$ in $S^3$. In addition, the Chern classes of the Stein traces $W_L$ and $W_{L'}$ are given by the rotation numbers of $L$ and $L'$, respectively, under the identification of the second cohomology of the Stein traces with the integers. Hence, since they are the same, we must have that $\rot(L)=\rot(L')$. Finally, it suffices to note that the knot types under consideration are Legendrian simple~\cite{EliashbergFraser98, EtnyreHonda01b}, and thus $L$ and $L'$ are Legendrian isotopic in $(S^3,\xi_{st})$. 
\end{proof}

Let us now proceed with Theorem \ref{minus1char} by showing that some of the Legendrian knots are contact $(-1)$-surgery characterized.

\begin{proof}[Proof of Theorem~\ref{minus1char}] For proving (1), let us consider a Legendrian realization $L$ in $(S^3,\xi_{st})$ of the unknot with $\tb(L)=t$, and $L'$ in $(S^3,\xi_{st})$ another Legendrian knot such that $L'(-1)$ is contactomorphic to $L(-1)$. As the homology of these manifolds is $\Z_{|t-1|}$, we must have that $t'=\tb(L')=t$ or $t'=2-t$.
	
	If $t'=t$, then we know that $L'$ is smoothly an unknot~\cite{KronheimerMrowkaOzsvathSzabo07}, and since, by Theorem~\ref{computInv}, the Euler classes of $L(-1)$ and $L'(-1)$ are determined by the rotation number $r$ of $L$ and $r'$ of $L'$, respectively, we must have that $r\equiv r'  \mod (1-t)$. From the classification of Legendrian unknots~\cite{EliashbergFraser98}, it follows that $0\leq r,r'\leq-1-t < 1-t$ and thus $r=r'$, implying that $L'$ is Legendrian isotopic to $L$. 
	
	Now consider the case $t'=2-t$, then we notice that $L(-1)$ is the lens space $L(1-t,1)$. The article \cite{McDuff90} has classified their Stein fillings, and the only simply connected Stein fillings are disk bundles over $S^2$ with negative Euler class. In particular, the intersection form of any such filling is negative definite. However, the Stein trace $W_{L'}$ of $L'$ is a symplectic filling of $L'(-1)$ with intersection pairing positive definite. Hence, if $L(-1)$ were contactomorphic to $L'(-1)$ we would have a symplectic filling of $L(1-t,1)$ that is not allowed by~\cite{McDuff90}. Thus $(-1)$ is a contact characterizing slope for any Legendrian unknot. This concludes Theorem~\ref{minus1char}.(1). 
	
	Let us proceed with Theorem~\ref{minus1char}.(3),
	the argument for the left-handed trefoil and the figure eight knot is identical. As above, $L'$ in $(S^3,\xi_{st})$ is a Legendrian knot such that $L'(-1)$ is contactomorphic to $L(-1)$, and we can assume that $t'=2-t$, since in the other case we already know that $L'$ is Legendrian isotopic to $L$, by a similar argument as above. (Here we use the classification of the Legendrian figure eight knots and trefoils from~\cite{EtnyreHonda01b}.)
	
	We proceed by computing the the $d_3$-invariant of the contact surgeries $L(-1)$ and $L'(-1)$, using Theorem~\ref{computInv}. Since the first homology of both $L(-1)$ and $L'(-1)$ is $\Z_{|t-1|}$ (and $t$ is negative since $L$ is a figure eight knot or a left-handed trefoil) we know that the Euler classes are torsion and thus the $d_3$-invariants are well-defined. Then it is readily seen that
	\begin{align*}
	d_3\big(L(-1)\big)&= \frac{-r^2+1-t}{4(1-t)},\\
	d_3\big(L'(-1)\big)&= \frac{{r'}^2-5+5t}{4(1-t)}.
	\end{align*}

	Hence, if $L(-1)$ and $L'(-1)$ are contactomorphic, we must have that $${r'}^2=-r^2+6(1-t),$$ where the possibilities for $r$ are in $\{-t+1,-t+3, \ldots, t-1\}$. For example, when $t=-1$, we must have $r=0$ and so ${r'}^2=12$: then $r'$ is not an integer and thus there is no such $L'$ with $t'=3$, and $L$ is contact $(-1)$-surgery characterized. Similarly, one verifies that for all $t\geq -10$, there is no integer solution for $r'$ except when $t=-5$ and $r=0$. (In this case $r'=6$ and we cannot rule out the possibility that such an $L'$ exists.) This concludes Theorem~\ref{minus1char}.(3b). 
	
	Let us focus on Theorem~\ref{minus1char}.(3a), i.e.\ showing that if $L$ has rotation number $r=0$, then it is $(-1)$ characterized. This condition would force ${r'}^2=6(1-t)$, and this only has integral solution when $t=1-6k^2$, for $k\in\N$, in which case $r'=6k$. In this case we can use Theorem~\ref{computInv} to deduce that the Euler class of $L(-1)$ is $0$ and the Euler class of $L'(-1)$ is $6k$. But in $H_1(K(-1))=\Z_{6k^2}$, we have $6k\not=0$ and thus $L(-1)$ cannot be contactomorphic to $L'(-1)$. 
	
	For Theorem~\ref{minus1char}.(3c) we want to show that if $r\geq \sqrt{6(1-t)}$, then $L$ is contact $(-1)$-surgery characterized. Indeed, in this case ${r'}^2\leq 0$ and thus we conclude that $r'=0$, but we can rule this case out using the Euler class, as we just did above. 
	
	Finally, for Theorem~\ref{minus1char}.(2), where $L$ is a Legendrian realization of the right-handed trefoil; suppose $L'$ is another Legendrian knot with $L'(-1)$ contactomorphic to $L(-1)$. First we address the case with $t=1$, in which the only possibility for $t'$ is $1$. Thus, since $L$ is smoothly determined by any surgery, $L'$ must be smoothly isotopic to $L$. In addition, since the Euler classes of $L(-1)$ and $L'(-1)$, which are determined by the rotation numbers of $L$ and $L'$, are the same, we must have the $r'=r=0$. Since the right-handed trefoil is Legendrian simple, we conclude that $L$ and $L'$ are Legendrian isotopic in this case. The cases when $t=0$ or $-1$ follows by a similar argument as above in Part~(3). 
\end{proof}

\begin{rem}
	Note that for a Legendrian figure eight knot $L$ with $\tb(L)=-11$ and $\rot(L)=6$, any Legendrian knot $L'$ with $\tb(L')=13$ and $\rot(L')=6$ will satisfy $d_3(L(-1))=d_3(L'(-1))$ and $L(-1)$ and $L'(-1)$ will also have the same Euler class. In this case, we still believe that $L$ is contact $(-1)$-surgery determined, but the above argument is not sufficient to establish this. One runs into a similar complication when $L$ is the right-handed trefoil with $\tb=-2$ and $\rot=3$.\hfill$\Box$
\end{rem}

Let us continue with the proof of Theorem~\ref{1characterize} that says a Legendrian unknot, right or left-handed trefoil, or figure eight knot is $(+1)$ characterized if its rotation number is zero.

\begin{proof}[Proof of Theorem~\ref{1characterize}]
	For $L$ the standard Legendrian unknot, with $t=-1$, any another Legendrian knot $L'$ for which $L'(+1)$ is contactomorphic to $L(+1)$, must satisfy $t'=-1$ as well (which we see by comparing the homologies). Thus we conclude that $L'$ is isotopic to $L$, as done in the previous proofs above. 
	
	Any other Legendrian realization $L$ of an unknot, a figure eight knot or a left-handed trefoil has $t<-1$. Again let $L'$ be another Legendrian knot such that $L'(+1)$ is contactomorphic to $L(+1)$. Then the arguments in the previous two proofs concludes that $t'$ must be $t$ or $-t-2$. As above, if $t'=t$, then we must have that $L'$ is isotopic to $L$. It remains to consider the case $t'=t-2$, and once again we compute the $d_3$-invariants which are well defined since the Euler classes are torsion. By applying Theorem~\ref{computInv} we obtain
	\begin{align*}
	d_3\big(L(-1)\big)&= -\frac{-r^2-5-5t}{4(1+t)},\\
	d_3\big(L'(-1)\big)&= -\frac{{r'}^2+1+t}{4(1+t)}.
	\end{align*}
	Thus ${r'}^2=r^2-6(t+1)$ and in particular for $r=0$ we deduce that $r'=0$ as above.
	
	Finally, let $L$ be a Legendrian realization of the right-handed trefoil with $r=0$ and let $L'$ be another Legendrian knot such that $L'(+1)$ is contactomorphic to $L(+1)$. By comparing the orders of the first homologies we observe again that $t'$ is either $t$ or $-t-2$. If $t'=t$ we conclude again that $L'$ is smoothly isotopic to $L$. Since $r=0$ the Euler class is $0$ and the $d_3$-invariants are defined from which we conclude that $r'=r=0$.
	
	Similarly, we can compare in the case $t'=-t-2$ the $d_3$-invariants to deduce that $r'=r=0$. (Here we need to distinguish the cases $t\leq-2$, $t=-1$ and $0\leq t$ since the signatures of the linking matrix will differ.) 
\end{proof}

\subsection{General characterizing slopes}\label{ssec:characterizing_general}
Let us discuss general rational slopes. First, we note that the notion for characterizing contact slopes presented in the introduction might appear at a first glance unnatural. The following example explains why this is necessary.

\begin{ex}
	Let $L$ in $(S^3,\xist)$ be the Legendrian unknot with Thurston--Bennequin invariant $t=-3$ and rotation number $r=2$, and $L'$ be the Legendrian unknot with Thurston--Bennequin invariant $t'=-3$ and rotation number $r'=0$. We consider the set of contact manifolds obtained from $L$ and $L'$ by contact $(-2)$-surgery. The Euler classes of the resulting contact structures compute to be
	\begin{align*}
	e\big(K_1(-2)\big)&=\{3,1\} \text{ and}\\
	e\big(K_2(-2)\big)&=\{1,-1\}.
	\end{align*}
	It follows that $L(-2)$ is not contactomorphic to $L(-2)$. On the other hand, the contact structures in $L(-2)$ and $L'(-2)$ with Euler class equal to $1$ are contactomorphic.\hfill$\Box$
\end{ex}

In this context, finding examples of non-characterizing slopes for these more general coefficients is (naturally) easier than for contact $(\pm1)$-surgeries. Here is a simple example:

\begin{ex}
	Let $L$ be a Legendrian realization of the torus knot $T_{5,4}$ with $t=7$ and $r=0$ and let $L'$ be a Legendrian realization of the torus knot $T_{11,2}$, with the same classical invariants. Then $L(14)$ is contactomorphic to $L'(14)$. Indeed, these surgeries topologically correspond to $21$-surgery along $T_{5,4}$ and $T_{11,2}$, yielding diffeomorphic lens spaces by~\cite{Moser71}. The knots $L$ and $L'$ are both stabilized twice with different signs~\cite{EtnyreHonda01b}, and therefore all contact manifolds obtained by positive surgery along them are overtwisted~\cite{EtnyreKegelOnaranPre}. Their homotopical invariants depend only on the classical invariants of the surgery knots and are therefore equal. The classification of overtwisted contact structures~\cite{Eliashberg89} concludes that the contact manifolds are contactomorphic.\hfill$\Box$
\end{ex}

Let us show Theorem~\ref{uknotchar}, proving that many slopes of Legendrian unknots are characterizing. 

\begin{proof}[Proof of Theorem~\ref{uknotchar}]
	We start with the non-integral slopes in Part (i). Let $L$ be a Legendrian unknot with $\tb(L)=t$ and suppose there is some non-integral, rational number $p/q$ (with $q>0$) and a Legendrian knot $L'$ such that $L(p/q)$ is contactomorphic to $L'(p/q)$. As a smooth manifold, $L(p/q)$ is obtained from $S^3$ by $t+p/q$ surgery on the unknot and thus its homology is $\Z_{|p+qt|}$, and it follows that $\tb(L')=t'=t$ or $t'=-t-2\frac pq$. In the former case, we can conclude that $L'$ is Legendrian isotopic to $L$ as in the proof of Theorem~\ref{minus1char}, so it remains to consider the latter case. 
	
	In this case, we first notice that, for $t'$ to be an integer, we must have that $q=2$ and then $t'=-t-p$. (Or $q=1$, but then $r$ is an integer and we are assuming that is not the case.) Now, since $L(p/q)$ is a lens space and has cyclic fundamental group, the Cyclic Surgery Theorem~\cite{CullerGordonLueckeShalen87} implies that $L'$ is either a Legendrian torus knot or the smooth surgery coefficient is an integer. Since $-t-p+\frac p2=\frac{-2t-p}{2}$ is not an integer, $L$ must be a torus knot. By~\cite{Moser71}, the only surgery slopes on an $(a,b)$--torus knot $T$ that yield a lens space are $ab\pm \frac 1m$. However, for positive torus knots, the maximal Thurston--Bennequin invariant of these torus knots is $ab-a-b$, and thus the contact surgery coefficient is larger than $a+b-1$. The smallest value $a+b-1$ can take for a non-trivial positive torus knot is $4$, and thus any non-integer contact surgery with coefficient less than $4$ on a Legendrian positive torus knot cannot produce a lens space. For negative torus knots, the smooth surgery coefficients yielding lens spaces are negative and strictly less that $-4$ for non-trivial torus knots, and thus are not of the form $\frac{-2t-p}{2}$ for $p/2<4$. This concludes the case were the slope $p/q$ is a non-integral, rational number.
	
	Let us proceed with the case of an integral contact surgery slope $n\leq 1$. If $n\leq -1$, then Theorem~\ref{minus1char} shows that $n$ is contact characterizing. Technically, the theorem shows that $(-1)$ is contact characterizing for any unknot. That said, contact $n$-surgery, for $n\leq-1$, on a Legendrian $L$ is simply contact $(-1)$-surgery on $L$ after an $(|n|-1)$-fold stabilization, so the result follows. In the case of $n=1$, the result is contained in Theorem~\ref{1characterize}. This finishes the proof of Theorem~\ref{uknotchar}.(i).
	
	Next, we discuss Theorem~\ref{uknotchar}.(ii). Let $L$ in $(S^3,\xi_{st})$ be a Legendrian realization of the unknot with Thurston--Benne\-quin invariant $t$. We need to show that $-t$ is a characterizing contact slope for $L$.
	
	First, we observe that $L(-t)$ is topologically the $0$-surgery along the unknot and therefore yields $S^1\times S^2$. Let $L'$ be another Legendrian knot with Thurston--Bennequin invariant $t'$ such that $L'(-t)$ is contactomorphic to $L(-t)$. Property $R$, proven in~\cite{Gabai87}, states that $S^1\times S^2$ has a unique topological surgery diagram along a single knot: the $0$-surgery along the unknot. Therefore, $L'$ has to be a Legendrian realization of the unknot with Thurston--Bennequin invariant also equal to $t$.
	
	Second, let us show that the rotation numbers $r$ of $L$ and $r'$ of $L'$ also agree. For that, we use Theorem~\ref{computInv} to compute the Euler-classes of the contact structures in $L(-t)$ and in $L'(-t)$, identified with elements in $H^2(S^1\times S^2)=\Z$ generated by the Poincar\'e dual of the meridian. We obtain:
	\begin{align*}
	e\big(L(-t)\big)&=t\pm r+1,\quad e\big(L'(-t)\big)=t\pm r'+1,
	\end{align*}
	from which we conclude that $r=r'$. Now, Legendrian unknots are classified by their classical invariants~\cite{EliashbergFraser98} and thus $L'$ is Legendrian isotopic to $L$. 
	
	For the contact slopes of the form $\frac{1-qt}{q}$, with $q\in\Z\setminus\{0\}$, we argue in a similar manner. Indeed, topologically, these slopes correspond to the $(1/q)$-surgery along the unknot yielding $S^3$. By~\cite{GordonLuecke89}, these are the only surgery diagrams along single knots yielding again $S^3$. Thus, any other Legendrian knot $L'$ with $L'(\frac{1-qt}{q})$ contactomorphic to $L(\frac{1-qt}{q})$ has to be a Legendrian unknot with Thurston--Bennequin invariant $t$. In~\cite{EtnyreKegelOnaranPre}, the $d_3$-invariants of the resulting contact structures were computed, from which it follows that $t=t'$ and $r=r'$. Thus, once again, $L'$ is Legendrian isotopic to $L$. This concludes Theorem~\ref{uknotchar}.$(ii)$.
\end{proof}

Let us now continue with the proof of Theorem~\ref{thm:reciprokal}. For that, we first prove a simple lemma, and recall that we only consider the rotation number up to sign.

\begin{lem}\label{lem:first}
	Let $L$ in $(S^3,\xist)$ be a Legendrian realization of a smooth knot $K$, with Thurston--Bennequin invariant $t=\tb(L)$ and rotation number $r=\rot(L)$. Suppose that $\frac{\pm1+nt}{n}$ is a characterizing slope of $K$ for $n\geq3$. Then, any Legendrian knot $L'$ in $(S^3,\xist)$ such that $L'(\pm1/n)$ is contactomorphic to $L(\pm1/n)$ must be Legendrian realization of $K$ with the same classical invariants as $L$.
\end{lem}

\begin{proof}
	Since $L(\pm1/n)$ is contactomorphic to $L'(\pm1/n)$, it follows that $K(\frac{\pm1+nt}{n})$ is diffeomorphic to $K'(\frac{\pm1+nt'}{n})$, where $K'$ denotes the underlying smooth knot type of $L'$ and $t'$ is the Thurston--Bennequin invariant of $L'$. By comparing the ranks of the first homology groups we conclude that
	\begin{equation*}
	|\pm1+nt|=|\pm1+nt'|
	\end{equation*}
	and, for $n\geq3$, it follows that $t'=t$. Since $\frac{\pm1+nt}{n}$ is a characterizing slope for $K$, we conclude that $K'$ is isotopic to $K$. It remains to show that the rotation number $r'$ of $L'$ is equal to $r$.
	
	For that, we compute the $d_3$-invariants of the contact structure $\xi$ on $L(\pm1/n)$ and the contact structure $\xi'$ on $L'(\pm1/n)$. It follows from $n\geq3$ that the first homology group of the underlying topological manifold is torsion and thus their $d_3$-invariants are well-defined. Via Theorem~\ref{computInv} we get
	\begin{align*}
	d_3(\xi)=&\frac{1}{4}\left(\frac{nr^2}{\pm1+nt}\pm(3-n)-3\operatorname{sign}(\pm1+nt)\right),\\
	d_3(\xi')=&\frac{1}{4}\left(\frac{nr'^2}{\pm1+nt}\pm(3-n)-3\operatorname{sign}(\pm1+nt)\right).
	\end{align*}
	This implies that $r'=r$. 
\end{proof}

Theorem~\ref{thm:reciprokal}.(iii), showing that the $(\pm 1/n)$-slope is contact surgery characterizing for the right and left-handed trefoil, and figure eight knot, when $n\geq 3$, follows from Lemma~\ref{lem:first}, as follows.

\begin{proof}[Proof of Theorem~\ref{thm:reciprokal}.(iii)]
	Any topological slope of the figure eight and the left and right-handed trefoil is characterizing by~\cite{OzsvathSzabo19}. Since those knots are all Legendrian simple~\cite{EtnyreHonda01b}, the statement follows directly from Lemma~\ref{lem:first}. 
\end{proof}

Lemma \ref{lem:first} can also be readily combined with known results on the existence on topological characterizing slopes to deduce Theorem~\ref{thm:reciprokal}.(i) and Theorem~\ref{thm:reciprokal}.(ii).

\begin{proof} [Proof of Theorem~\ref{thm:reciprokal}.(i) and (ii)]
	Let $K$ be a hyperbolic knot. By \cite{Lackenby19}, any slope $p/q$ is characterizing for $|q|$ sufficiently large, and thus Theorem~\ref{thm:reciprokal}.(i) follows from Lemma~\ref{lem:first}.
	
	For a (non-hyperbolic) topological knot, it follows from \cite{Lackenby19} that $p/q$ is characterizing if $|p|\leq|q|$ and $|q|$ is sufficiently large. The conditions from Theorem~\ref{thm:reciprokal}.(ii) on $\tb(L)$ ensure that the corresponding topological surgery coefficients are characterizing, and thus the statement follows again from Lemma~\ref{lem:first}.
\end{proof}

The proof of Theorem~\ref{thm:reciprokal} is now concluded once we establish Theorem~\ref{thm:reciprokal}.(iv), as follows:

\begin{proof}[Proof of Theorem~\ref{thm:reciprokal}.(iv)]
	The results about the characterizing slopes of the Legendrian trefoils and figure eight knots, stated in Theorem~\ref{thm:reciprokal}.$(iv.1)$ and $(iv.2)$ readily follow along the same lines as the proof of Theorem~\ref{uknotchar}.$(ii)$; we omit the details and only discuss the main points. First, we use that these knots are also classified by their classical invariants~\cite{EtnyreHonda01b} and that a contact $(-1-t)$-surgery along a left-handed Legendrian trefoil with Thurston--Bennequin invariant $t$ corresponds to the topological $(-1)$-surgery along the left-handed trefoil. Therefore, it yields the Poincar\'e homology sphere. By \cite{Ghiggini08b}, this is actually the unique topological surgery description of the Poincar\'e homology sphere. Then a computation of the homotopical invariants of the resulting contact structures finishes the proof as above. 
	
	For the results for the Legendrian realizations of right-handed trefoils and figure eight knots, we use again that these knots are classified by their classical invariants in~\cite{EtnyreHonda01b} and that the Brieskorn homology sphere $\Sigma(2,3,7)$ has exactly two topological surgery diagrams along a single knot: the $(+1)$-surgery along the figure eight and the $(-1)$-surgery along the right-handed trefoil~\cite{OzsvathSzabo19}. 
\end{proof}

\subsection{Legendrian knots with infinitely many surgeries the same} The results of this manuscript have discussed in detail similarities and differences between the smooth framework and contact topological statements. The last result that we present emphasizes a difference, as follows.

In smooth low-dimensional topology, a folk theorem states that if two knots $K$ and $K'$ in $S^3$ admit infinitely many slopes $r$ such that $K(r)$ is diffeomorphic to $K'(r)$, then $K$ and $K'$ are isotopic. For instance, the generic case, where $K$ and $K'$ are hyperbolic, follows directly from Thurston's hyperbolic Dehn filling theorem~\cite{Thurston80}. Indeed, for a sufficiently large slope, the core of the newly glued-in solid torus represents the shortest geodesic and thus any isometry of the surgered manifold has to preserve that geodesic and hence restricts to an isometry of the knot complements. That said, the analogous result in contact topology, for Legendrian knots, does not hold.

\begin{thm}
	Let $n\in\N$ be any positive integer. Then there exist $n$ pairwise non-isotopic Legendrian knots $L_1,\ldots, L_n$ in $(S^3,\xist)$, all smoothly isotopic, such that the contactomorphism type of $L_i(r)$ is independent of $i$, $1\leq i\leq n$, for every contact surgery coefficient $r\not\in [-1,0]$.
\end{thm}

\begin{proof}
We first show that there exist $n$ pairwise non-isotopic Legendrian knots $L_1,\ldots, L_n$ in $(S^3,\xist)$ such that the contactomorphism type of $L_i(r)$ is independent of $i$, $1\leq i\leq n$, for every contact surgery coefficient $r>0$.
Let $L_1,\ldots, L_n$ be pairwise non-isotopic Legendrian knots in $(S^3,\xist)$ with the same classical invariants that are all stabilized once negatively and once positively. For instance, cables of the right-handed trefoil provide examples of such $L_i$, see~\cite{EtnyreLafountainTosun12}.
	
Since the $L_i$ are stabilized with both signs, it is known that $L_i(r)$ is overtwisted for all positive $r$~\cite{EtnyreKegelOnaranPre}. Since the $L_i$ also have the same classical invariants, the contact structures on the $L_i(r)$ have the same homotopical invariants as $2$-plane fields and thus are isotopic by the classification of overtwisted contact structures~\cite{Eliashberg89}.
	
We now show there exist $n$ pairwise non-isotopic Legendrian knots $L'_1,\ldots, L'_n$ in $(S^3,\xist)$ such that the contactomorphism type of $L'_i(r)$ is independent of $i$, $1\leq i\leq n$, for every contact surgery coefficient $r<-1$. Let $L'_1,\ldots, L'_n$ be non-isotopic Legendrian knots in $(S^3,\xist)$ with the same classical invariants that become isotopic after a single positive or after a single negative stabilization. Negative twists knots provide such examples, see~\cite{EtnyreNgVertesi13}. In~\cite{DingGeigesStipsicz04},  it is shown that any contact structure in $L'_i(r)$ for $r<-1$ can be expressed as a sequence of contact $(-1)$-surgeries along stabilizations of $L'_i$. Since the stabilizations of the $L'_i$ are Legendrian isotopic, it follows that the surgered contact manifolds are contactomorphic.

To create the knots in the theorem, we just need $L_1,\ldots, L_n$ that are stabilized both positively and negatively, and become the same after a single positive or a single negative stabilization. To find such knots notice that the twist knots $K_{-2n-1}$ from \cite{EtnyreNgVertesi13} have $n$ Legendrian representatives with $\tb=-3$ and $\rot=0$, they moreover become isotopic after a single positive or a single negative stabilization. Now take the $(2,1)$-cable of these knots to get $n$ Legendrian knots with the same invariants. If we stabilize each of these one time positive and one time negatively they are still distinct Legendrian knots, but after another stabilization of either sign they will become isotopic, see \cite{ChakrabortyEtnyreMin20pre}.
\end{proof}

With similar methods we observe that Theorem~\ref{thm:main} also implies the existence of an infinite family of Legendrian knots that all share the same contact $(+1)$-surgery, a phenomena that was only known for finitely many Legendrian knots before~\cite{Etnyre08}.

\begin{cor}
	There exist infinitely many pairwise non-isotopic Legendrian knots $L_n$ in $(S^3,\xist)$, $n\in\N_0$, such that their contact $(+1)$-surgeries are all contactomorphic, for any $n \in \N_0$. These knots $L_n$ all have $\tb=-1$ and $\rot=0$ or $\rot=2$.
\end{cor}

\begin{proof}
	We define the Legendrian knots $L_n$ to be the two-fold stabilizations of the Legendrian knots in Theorem~\ref{thm:main}.(i) where we choose either all stabilizations of the same sign, or for every knot one positive and one negative stabilization. Since the underlying smooth knots types differ, it follows that the $L_n$ are pairwise non-isotopic. Their contact $(+1)$-surgeries agree smoothly with the Legendrian surgeries of the knots from Theorem~\ref{thm:main}.(i) and thus are all smoothly the same. Since the $L_n$ are all stabilized, the surgered contact manifolds are all overtwisted. Given that the $L_n$ have the same classical invariants, these contact structures are homotopic and thus, being overtwisted, also contactomorphic.
\end{proof}

\section{Questions and Conjectures}\label{sec:QuestionsConj}

The distinct Legendrian knots with contactomorphic surgeries, or symplectomorphic Stein traces, that we constructed are already distinguished by their underlying smooth knot type. Thus we ask the following natural question:

\begin{ques}\label{ques:classical}
	Do there exist non-isotopic Legendrian knots $L_0$ and $L_1$ in $(S^3,\xi_{st})$ with the same underlying smooth knot type, Thurston--Bennequin invariant and rotation number, such that their Stein traces $W_{L_0}$ and $W_{L_1}$ -- or their Legendrian surgeries $L_0(-1)$ and $L_1(-1)$ -- are equivalent?
\end{ques}

Note that the surgery exact triangle from~\cite{BourgeoisEkholmEliashberg12} implies that Legendrian knots $L_0$ and $L_1$ in $(S^3,\xi_{st})$ whose Legendrian Contact DGAs have distinct Hochschild homologies will have non-equivalent Stein traces. To eliminate this matter, a potential class of candidates is given by distinct Legendrian representatives of a given smooth knot, with the same $(\tb,\rot)$, which are stabilized. For instance, the twist knot $m(7_2)$ has two distinct stabilized Legendrian representatives with $(\tb,\rot)=(0,1)$ \cite[Theorem 1.1]{EtnyreNgVertesi13}. Nevertheless, it might be the case that all such knots have non-contactomorphic $(-1)$-surgeries if Conjecture~\ref{con:+and-} below is true. Note also that the analogue of Question \ref{ques:classical} in the higher-dimensional setting is known to have an affirmative answer, e.g.~see \cite[Corollary 1.22]{Lazarev19}, where the Stein manifold $(T^*S^n,\lambda_{st})$ is presented as a Stein trace of infinitely many Legendrian $(n-1)$-spheres in $(S^{2n-1},\xi_{st})$, for $n\geq3$, which are formally isotopic but not Legendrian isotopic.

\begin{con} \label{con:+and-}
	Let $L$ and $L'$ in $(S^3,\xi_{st})$ be two Legendrian knots. Suppose that $L(-1)$ is contactomorphic to $L'(-1)$, and $L(+1)$ is contactomorphic to $L'(+1)$. Then $L$ is Legendrian isotopic to $L'$ in $(S^3,\xi_{st})$.
\end{con}

Note that this would imply that any two stabilized Legendrian representatives of the same smooth knot with equal $(\tb,\rot)$ are characterized by their contact $(-1)$-surgeries, as the $(+1)$-surgeries are overtwisted (and thus contactomorphic).

Regarding the results presented in Section \ref{sec:main_char}, the fact that the knots discussed in those results are smoothly characterized by surgeries on them~\cite{KronheimerMrowkaOzsvathSzabo07, OzsvathSzabo19} indicates the following natural conjecture:

\begin{con}\label{con:torus}
	The slope $\pm1$ is a characterizing contact slope for any Legendrian figure eight knot or torus knot $($including the unknot$)$.
\end{con}

We also recall that we proved many slopes of Legendrian unknots to be characterizing in Theorem~\ref{uknotchar}. On the other hand, we have identified in Example~\ref{ex:noncharunknot} slopes of Legendrian unknots that are not characterizing. This naturally yields the following question:

\begin{ques}\label{ques:clascharunknot}
What is the classification of the characterizing slopes of Legendrian unknots?
\end{ques}

We would also like to add the following question:

\begin{ques}
	Does any Legendrian knot $L$ in $(S^3,\xi_{st})$ have a $($infinitely many$)$ characterizing contact slope(s)?
\end{ques}

Finally, we briefly discuss the relation of Legendrian knots with the same Stein traces and closed embedded Lagrangian surfaces in Stein traces, as follows.

Consider an infinite family of Legendrian knots $L_n$ in $(S^3,\xi_{st})$ with Stein equivalent traces, and suppose that each of the knots $L_n$ bounds an embedded exact Lagrangian surface $\Sigma_n$ in $(D^4,\lambda_{st})$. Then the union of the Lagrangian core of the Weinstein $2$-handle and $\Sigma_n$ defines a closed embedded exact Lagrangian surface $\overline{\Sigma}_n$ in the Stein trace $W_{L_n}$.

\begin{ques} Let $L_n$ in $(S^3,\xi_{st})$ be an infinite family of Legendrian knots with equivalent Stein traces,  $n\in\N_0$, each knot $L_n$ bounding an embedded exact Lagrangian surface $\Sigma_n$ in $(D^4,\lambda_{st})$. Having identified the Stein traces $W_{L_n}$ with $W_{L_0}$, are the closed embedded exact Lagrangian surface $\overline{\Sigma}_n$ Hamiltonian isotopic in $W_{L_0}$?
\end{ques}

It might be that the answer depends on the chosen family $L_n$ in $(S^3,\xi_{st})$ of Legendrian knots. It would be interesting to have examples where the Lagrangian surface $\overline{\Sigma}_n$ yield infinitely many different Hamiltonian isotopy classes in the same Stein trace.

As a starting example, each of the knots in the infinity family $L_n$ in $(S^3,\xi_{st})$ of Legendrian knots obtained by considering the knots $L_0$ and $L_1$ in Figure~\ref{fig:primeex} and performing contact annulus twist bounds an exact Lagrangian (punctured) 2-torus $\Sigma_n=T_n$. We do not know whether the $\overline{\Sigma}_n$ are Hamiltonian isotopic for different $n\in\N_0$. 
It might be the case that these tori are all smoothly isotopic, and even Hamiltonian isotopic, but this remains to be explored.

\bibliographystyle{alpha}
\bibliography{references}

\def\cprime{$'$} \def\cprime{$'$}
\begin{thebibliography}{KMOS07}

\bibitem[AG01]{ArnoldGivental01}
V.~I. Arnol'd and A.~B. Givental\cprime.
\newblock Symplectic geometry.
\newblock In {\em Dynamical systems, {IV}}, volume~4 of {\em Encyclopaedia
  Math. Sci.}, pages 1--138. Springer, Berlin, 2001.

\bibitem[AJLO15]{AbeJongLueckeOsoinach15}
Tetsuya Abe, In~Dae Jong, John Luecke, and John Osoinach.
\newblock Infinitely many knots admitting the same integer surgery and a
  four-dimensional extension.
\newblock {\em Int. Math. Res. Not. IMRN}, (22):11667--11693, 2015.

\bibitem[AJOT13]{AbeJongOmaeTakeuchi13}
Tetsuya Abe, In~Dae Jong, Yuka Omae, and Masanori Takeuchi.
\newblock Annulus twist and diffeomorphic 4-manifolds.
\newblock {\em Math. Proc. Cambridge Philos. Soc.}, 155(2):219--235, 2013.

\bibitem[Akb77]{Akbulut77}
Selman Akbulut.
\newblock On {$2$}-dimensional homology classes of {$4$}-manifolds.
\newblock {\em Math. Proc. Cambridge Philos. Soc.}, 82(1):99--106, 1977.

\bibitem[Avd13]{Avdek13}
Russell Avdek.
\newblock Contact surgery and supporting open books.
\newblock {\em Algebr. Geom. Topol.}, 13(3):1613--1660, 2013.

\bibitem[BEE12]{BourgeoisEkholmEliashberg12}
Fr\'ed\'eric Bourgeois, Tobias Ekholm, and Yasha Eliashberg.
\newblock Effect of {L}egendrian surgery.
\newblock {\em Geom. Topol.}, 16(1):301--389, 2012.
\newblock With an appendix by Sheel Ganatra and Maksim Maydanskiy.

\bibitem[BGL16]{BakerGordonLuecke16}
Kenneth~L. Baker, Cameron Gordon, and John Luecke.
\newblock Bridge number and integral {D}ehn surgery.
\newblock {\em Algebr. Geom. Topol.}, 16(1):1--40, 2016.

\bibitem[BM18]{BakerMotegi18}
Kenneth~L. Baker and Kimihiko Motegi.
\newblock Noncharacterizing slopes for hyperbolic knots.
\newblock {\em Algebr. Geom. Topol.}, 18(3):1461--1480, 2018.

\bibitem[Bra80]{Brakes80}
W.~R. Brakes.
\newblock Manifolds with multiple knot-surgery descriptions.
\newblock {\em Math. Proc. Cambridge Philos. Soc.}, 87(3):443--448, 1980.

\bibitem[CDGW]{CullerDunfieldGoernerWeeks}
M.~Culler, N.~Dunfield, M.~Goerner, and J.~Weeks.
\newblock Snappy, a computer program for studying the geometry and topology of
  3-manifolds.

\bibitem[CE12]{CieliebakEliashberg12}
Kai Cieliebak and Yakov Eliashberg.
\newblock {\em From {S}tein to {W}einstein and back}, volume~59 of {\em
  American Mathematical Society Colloquium Publications}.
\newblock American Mathematical Society, Providence, RI, 2012.
\newblock Symplectic geometry of affine complex manifolds.

\bibitem[CEK]{CasalsEtnyreKegel}
Roger Casals, John~B. Etnyre, and Marc Kegel.
\newblock Code and data to accompany this paper.
\newblock Accesible as notebook on the author's webpages or downloadable as
  auxiliary files from the arXiv.

\bibitem[CEM20]{ChakrabortyEtnyreMin20pre}
Apratim Chakraborty, John~B. Etnyre, and Hyunki Min.
\newblock Cabling legendrian and transverse knots, 2020.

\bibitem[CET21]{ConwayEtnyreTosun21}
James Conway, John~B. Etnyre, and B\"{u}lent Tosun.
\newblock Symplectic fillings, contact surgeries, and {L}agrangian disks.
\newblock {\em Int. Math. Res. Not. IMRN}, (8):6020--6050, 2021.

\bibitem[CGLS87]{CullerGordonLueckeShalen87}
Marc Culler, C.~McA. Gordon, J.~Luecke, and Peter~B. Shalen.
\newblock Dehn surgery on knots.
\newblock {\em Ann. of Math. (2)}, 125(2):237--300, 1987.

\bibitem[DG04]{DingGeiges04}
Fan Ding and Hansj{\"o}rg Geiges.
\newblock A {L}egendrian surgery presentation of contact 3-manifolds.
\newblock {\em Math. Proc. Cambridge Philos. Soc.}, 136(3):583--598, 2004.

\bibitem[DG09]{DingGeiges09}
Fan Ding and Hansj{\"o}rg Geiges.
\newblock Handle moves in contact surgery diagrams.
\newblock {\em J. Topol.}, 2(1):105--122, 2009.

\bibitem[DGS04]{DingGeigesStipsicz04}
Fan Ding, Hansj{\"o}rg Geiges, and Andr{\'a}s~I. Stipsicz.
\newblock Surgery diagrams for contact 3-manifolds.
\newblock {\em Turkish J. Math.}, 28(1):41--74, 2004.

\bibitem[DK16]{DurstKegel16}
S.~Durst and M.~Kegel.
\newblock Computing rotation and self-linking numbers in contact surgery
  diagrams.
\newblock {\em Acta Math. Hungar.}, 150(2):524--540, 2016.

\bibitem[EF98]{EliashbergFraser98}
Yakov Eliashberg and Maia Fraser.
\newblock Classification of topologically trivial {L}egendrian knots.
\newblock In {\em Geometry, topology, and dynamics (Montreal, PQ, 1995)},
  volume~15 of {\em CRM Proc. Lecture Notes}, pages 17--51. Amer. Math. Soc.,
  Providence, RI, 1998.

\bibitem[EH01]{EtnyreHonda01b}
John~B. Etnyre and Ko~Honda.
\newblock Knots and contact geometry. {I}. {T}orus knots and the figure eight
  knot.
\newblock {\em J. Symplectic Geom.}, 1(1):63--120, 2001.

\bibitem[EKO]{EtnyreKegelOnaranPre}
John~B. Etnyre, Marc Kegel, and Sinem Onaran.
\newblock Contact surgery numbers.
\newblock Preprint 2021.

\bibitem[Eli89]{Eliashberg89}
Y.~Eliashberg.
\newblock Classification of overtwisted contact structures on {$3$}-manifolds.
\newblock {\em Invent. Math.}, 98(3):623--637, 1989.

\bibitem[ELT12]{EtnyreLafountainTosun12}
John~B. Etnyre, Douglas~J. LaFountain, and B{\"u}lent Tosun.
\newblock Legendrian and transverse cables of positive torus knots.
\newblock {\em Geom. Topol.}, 16(3):1639--1689, 2012.

\bibitem[ENV13]{EtnyreNgVertesi13}
John~B. Etnyre, Lenhard~L. Ng, and Vera V{\'e}rtesi.
\newblock Legendrian and transverse twist knots.
\newblock {\em J. Eur. Math. Soc. (JEMS)}, 15(3):969--995, 2013.

\bibitem[Etn08]{Etnyre08}
John~B. Etnyre.
\newblock On contact surgery.
\newblock {\em Proc. Amer. Math. Soc.}, 136(9):3355--3362, 2008.

\bibitem[Gab87]{Gabai87}
David Gabai.
\newblock Foliations and the topology of {$3$}-manifolds. {II}.
\newblock {\em J. Differential Geom.}, 26(3):461--478, 1987.

\bibitem[Gei08]{Geiges08}
Hansj{\"o}rg Geiges.
\newblock {\em An introduction to contact topology}, volume 109 of {\em
  Cambridge Studies in Advanced Mathematics}.
\newblock Cambridge University Press, Cambridge, 2008.

\bibitem[Ghi08]{Ghiggini08b}
Paolo Ghiggini.
\newblock Knot {F}loer homology detects genus-one fibred knots.
\newblock {\em Amer. J. Math.}, 130(5):1151--1169, 2008.

\bibitem[Gir00]{Giroux00}
Emmanuel Giroux.
\newblock Structures de contact en dimension trois et bifurcations des
  feuilletages de surfaces.
\newblock {\em Invent. Math.}, 141(3):615--689, 2000.

\bibitem[GL89]{GordonLuecke89}
C.~McA. Gordon and J.~Luecke.
\newblock Knots are determined by their complements.
\newblock {\em J. Amer. Math. Soc.}, 2(2):371--415, 1989.

\bibitem[Gom98]{Gompf98}
Robert~E. Gompf.
\newblock Handlebody construction of {S}tein surfaces.
\newblock {\em Ann. of Math. (2)}, 148(2):619--693, 1998.

\bibitem[Gro85]{Gromov85}
M.~Gromov.
\newblock Pseudoholomorphic curves in symplectic manifolds.
\newblock {\em Invent. Math.}, 82(2):307--347, 1985.

\bibitem[GS99]{GompfStipsicz99}
Robert~E. Gompf and Andr{\'a}s~I. Stipsicz.
\newblock {\em {$4$}-manifolds and {K}irby calculus}, volume~20 of {\em
  Graduate Studies in Mathematics}.
\newblock American Mathematical Society, Providence, RI, 1999.

\bibitem[Hon00]{Honda00a}
Ko~Honda.
\newblock On the classification of tight contact structures. {I}.
\newblock {\em Geom. Topol.}, 4:309--368 (electronic), 2000.

\bibitem[Keg18]{Kegel18}
Marc Kegel.
\newblock The {L}egendrian knot complement problem.
\newblock {\em J. Knot Theory Ramifications}, 27(14):1850067, 36, 2018.

\bibitem[Kir78]{KirbyList}
Rob Kirby.
\newblock Problems in low dimensional manifold theory.
\newblock In {\em Algebraic and geometric topology ({P}roc. {S}ympos. {P}ure
  {M}ath., {S}tanford {U}niv., {S}tanford, {C}alif., 1976), {P}art 2}, Proc.
  Sympos. Pure Math., XXXII, pages 273--312. Amer. Math. Soc., Providence,
  R.I., 1978.

\bibitem[KMOS07]{KronheimerMrowkaOzsvathSzabo07}
P.~Kronheimer, T.~Mrowka, P.~Ozsv\'{a}th, and Z.~Szab\'{o}.
\newblock Monopoles and lens space surgeries.
\newblock {\em Ann. of Math. (2)}, 165(2):457--546, 2007.

\bibitem[Lac19]{Lackenby19}
Marc Lackenby.
\newblock Every knot has characterising slopes.
\newblock {\em Math. Ann.}, 374(1-2):429--446, 2019.

\bibitem[Laz19]{Lazarev19}
Oleg Lazarev.
\newblock Geometric and algebraic presentations for {W}einstein domains, 2019.

\bibitem[McC19]{McCoy19}
Duncan McCoy.
\newblock On the characterising slopes of hyperbolic knots.
\newblock {\em Math. Res. Lett.}, 26(5):1517--1526, 2019.

\bibitem[McC20]{McCoy20}
Duncan McCoy.
\newblock Non-integer characterizing slopes for torus knots.
\newblock {\em Comm. Anal. Geom.}, 28(7):1647--1682, 2020.

\bibitem[McD90]{McDuff90}
Dusa McDuff.
\newblock The structure of rational and ruled symplectic {$4$}-manifolds.
\newblock {\em J. Amer. Math. Soc.}, 3(3):679--712, 1990.

\bibitem[Mos71]{Moser71}
Louise Moser.
\newblock Elementary surgery along a torus knot.
\newblock {\em Pacific J. Math.}, 38:737--745, 1971.

\bibitem[MP18]{MillerPiccirillo18}
Allison~N. Miller and Lisa Piccirillo.
\newblock Knot traces and concordance.
\newblock {\em J. Topol.}, 11(1):201--220, 2018.

\bibitem[MP21]{ManolescuPiccirillo21}
Ciprian Manolescu and Lisa Piccirillo.
\newblock From zero surgeries to candidates for exotic definite four-manifolds,
  2021.

\bibitem[NZ85]{NeumannZagier85}
Walter~D. Neumann and Don Zagier.
\newblock Volumes of hyperbolic three-manifolds.
\newblock {\em Topology}, 24(3):307--332, 1985.

\bibitem[NZ14]{NiZhang14}
Yi~Ni and Xingru Zhang.
\newblock Characterizing slopes for torus knots.
\newblock {\em Algebr. Geom. Topol.}, 14(3):1249--1274, 2014.

\bibitem[OS19]{OzsvathSzabo19}
Peter Ozsv\'{a}th and Zolt\'{a}n Szab\'{o}.
\newblock The {D}ehn surgery characterization of the trefoil and the figure
  eight knot.
\newblock {\em J. Symplectic Geom.}, 17(1):251--265, 2019.

\bibitem[Oso06]{Osoinach06}
John~K. Osoinach, Jr.
\newblock Manifolds obtained by surgery on an infinite number of knots in
  {$S^3$}.
\newblock {\em Topology}, 45(4):725--733, 2006.

\bibitem[Pic19]{Piccirillo19}
Lisa Piccirillo.
\newblock Shake genus and slice genus.
\newblock {\em Geom. Topol.}, 23(5):2665--2684, 2019.

\bibitem[Tag20a]{Tagami2}
Keiji Tagami.
\newblock Notes on constructions of knots with the same trace, 2020.

\bibitem[Tag20b]{Tagami1}
Keiji Tagami.
\newblock On annulus presentations, dualizable patterns and rgb-diagrams, 2020.

\bibitem[Thu80]{Thurston80}
William~P. Thurston.
\newblock The geometry and topology of 3-manifolds.
\newblock Princeton lecture notes, 1980.

\bibitem[Wei91]{Weinstein91}
Alan Weinstein.
\newblock Contact surgery and symplectic handlebodies.
\newblock {\em Hokkaido Math. J.}, 20(2):241--251, 1991.

\end{thebibliography}
\end{document}